\documentclass[reqno]{amsart}

\usepackage{amsmath,amssymb,amsthm,amsfonts}
\usepackage{color}
\usepackage{mathrsfs}
\usepackage{hyperref}
\usepackage{a4wide}

\newtheorem{theorem}{Theorem}[section]
\newtheorem{lemma}[theorem]{Lemma}
\newtheorem{prop}[theorem] {Proposition}
\newtheorem{cor}[theorem]  {Corollary}
\newtheorem{definition}[theorem] {Definition}

\theoremstyle{definition}

\theoremstyle{remark}
\newtheorem*{remark}{Remark}
\newtheorem*{example}{Example}

\numberwithin{equation}{section}

\newcommand{\e}{\mathrm{e}} 
\newcommand{\N}{\mathbb{N}}
\newcommand{\R}{\mathbb{R}}

\newcommand{\C}{\mathbb{C}}

\renewcommand{\P}{\mathbb{P}}
\newcommand{\E}{\mathbb{E}}
\newcommand{\dd}{\mathrm{d}} 
\newcommand{\eps}{\varepsilon}

\newcommand{\vect}[1]{\boldsymbol{#1}}

\renewcommand{\Re}{\mathrm{Re}\,} 

\newcommand{\be}{\begin{equation}}
\newcommand{\ee}{\end{equation}}
\newcommand{\ba}{\begin{equation} \begin{aligned}}
\newcommand{\ea}{\end{aligned}\end{equation}}
\newcommand{\bes}{\begin{equation*}}
\newcommand{\ees}{\end{equation*}}

\def\1{{\mathchoice {1\mskip-4mu\mathrm l}      
{1\mskip-4mu\mathrm l}
{1\mskip-4.5mu\mathrm l} {1\mskip-5mu\mathrm l}}}

\begin{document}

\title{Cluster expansions for Gibbs point processes}
\author{Sabine Jansen}
\address{Mathematisches Institut, Ludwig-Maximilians Universit{\"a}t, Theresienstr. 39, 80333 
 M{\"u}nchen, Germany}
 \email{jansen@math.lmu.de}

\date{17 June 2019}
\maketitle
\begin{abstract} 
	We provide a sufficient condition for the uniqueness in distribution of Gibbs point processes with non-negative pairwise interaction, together with convergent expansions of the log-Laplace functional, factorial moment densities and factorial cumulant densities (correlation functions and truncated correlation functions). The criterion is a continuum version of a convergence condition by Fern{\'a}ndez and Procacci (2007), the proof is based on the Kirkwood-Salsburg integral equations and is close in spirit to the approach by Bissacot, Fern{\'a}ndez and Procacci (2010). In addition, we provide formulas for cumulants of double stochastic integrals with respect to Poisson random measures (not compensated) in terms of multigraphs and pairs of partitions, explaining how   to go from cluster expansions to some diagrammatic expansions (Peccati and Taqqu, 2011). We also discuss relations with generating functions for trees, branching processes, Boolean percolation and the random connection model. The presentation is self-contained and requires no preliminary knowledge of cluster expansions. \\

	
\end{abstract}


\tableofcontents

\section{Introduction}

Gibbs point processes form an important class of models in statistical mechanics, stochastic geometry and spatial statistics~\cite{chiu-kendall-stoyan-mecke, moller-waagepetersen,  dereudre2017gibbsintro}.  In finite volume, they are defined, roughly, as modifications of Poisson point processes. The modification involves a factor $\exp( - H(\eta))$ where $H(\eta)$ incorporates interactions between points and the magnitude of $H$ captures how far the Gibbs point process might be from the a priori Poisson point process. 
In infinite volume, Gibbs measures are defined instead by structural properties such as the GNZ equation and the DLR conditions, named after Georgii, Nguyen  and Zessin, and Dobrushin, Lanford, and Ruelle, respectively, and proving the mere existence of such measures requires some work. 

A notorious difficulty when dealing with Gibbs measures is that many quantities cannot be computed explicitly. For example, the intensity measure $\rho$ of a Gibbs point process is a highly non-trivial function of the intensity measure $\lambda_z$ of the underlying Poisson point process, which leads to challenges when estimating, say, intensity parameters of the Poisson point process based on observations of the density of the Gibbs point process~\cite{baddeley-nair2012}. As a way out, physicists and mathematical physicists have long worked with power series expansions~\cite[Chapter 4.3]{ruelle1969book}: If interactions are sufficiently weak, then the Gibbs point process should be close to the non-interacting Poisson point process, and correction terms may be captured by convergent power series in the Poisson intensity parameter, called \emph{activity} or \emph{fugacity} in statistical mechanics. The expansions obtained in this way are called \emph{cluster expansions}; the name stems from combinatorial expressions for the expansion coefficients in terms of connected graphs. 

The mathematical theory of cluster expansions is rich and well-developed~\cite{brydges1986clustercourse}. Cluster expansions feature prominently in mathematical physics, and have found applications in combinatorics~\cite{scott-sokal2005,faris2010combinatorics} and random graphs~\cite{yin2012clustergraphs}. Their popularity in the community of point processes, however, stays somewhat limited (see,  nevertheless,~\cite{nehring-poghosyan-zessin2013} and the references therein), in stark contrast with  expansion techniques from the realm of Poisson-Malliavin calculus~\cite{peccati-taqqu2011book,last-penrose2017book}. Interestingly, in their seminal article on  combinatorics and stochastic integration~\cite{rota-wallstrom1997}, Rota and Wallstrom explicitly mention Feynman diagrams,  ``physicists aiming at the development of nonlinear quantum field theories,'' and ``probabilists in search of new point processes that would not turn out to be Poisson distributions in disguise,'' but do not mention statistical physics or cluster expansions at all. It is our hope that the present article helps make the theory of cluster expansions accessible to a broader community of probabilists working with point processes.

Our first main result is a sufficient criterion for the uniqueness of Gibbs point processes with non-negative pair interactions (Theorem~\ref{thm:uniqueness}), accompanied by convergent expansions for the log-Laplace functional of the Gibbs point process as well as its factorial moment and factorial cumulant densities (also known as correlation functions and truncated correlation functions), see Theorems~\ref{thm:log-laplace}, \ref{thm:truncated}, and \ref{thm:correlations}.  The  existence of such expansions, and their convergence in some non-empty domain, is well-known~\cite{ruelle1969book, poghosyan-ueltschi2009, nehring-poghosyan-zessin2013}, however our generalized convergence criterion was previously proven only for discrete polymer systems~\cite{fernandez-procacci2007} and hard spheres~\cite{fernandez-procacci-scoppola2007}. 

The results in the present article cover non-negative pairwise interactions only. For pairwise interactions with possibly negative values, many convergence conditions are available in the literature, see e.g.~\cite{poghosyan-ueltschi2009, procacci-yuhjtman2017} and Section~\ref{sec:attractive}. Cluster expansions for multi-body interactions are much less developed, some results can be found in~\cite{moraal1976multibody,procacci-scoppola2000,rebenko2005multibody}. 

Convergence conditions in the theory of cluster expansions often mirror fixed point equations for generating functions of trees~\cite{faris2010combinatorics,fernandez-procacci2007,jansen2015tonks}. The equations  are reminiscent of an equation satisfied by extinction probabilities in branching processes. In Section~\ref{sec:branching} we discuss these relations in more depth and show that if a convergence condition due to Ueltschi~\cite{ueltschi2004cluster} holds true, then there is extinction in an associated multi-type branching process (Proposition~\ref{prop:extinction}). 

This observation is interesting because in turn, extinction of the branching process implies absence of percolation in a related random connection model~\cite{meester-roy1996book}. In Section~\ref{sec:boolean} we discuss these relations in the context of disagreement percolation, an expansion-free method for proving uniqueness of Gibbs measures and exponential mixing~\cite{georgii-haggstrom-maes2001}. In the future these consideration may help systematize relations between Gibbs point processes and the random connection model, perhaps starting from the duality between hard spheres and Boolean percolation model discussed by  Torquato~\cite{torquato2012}, or expansion methods and integral equations for connectedness functions~\cite{coniglio-deangelis-forlani1977, given-stell1990, last-ziesche2017}.

Our second main result consists in two propositions on moments and cumulants of double stochastic integrals with respect to a random Poisson measure (not compensated). Proposition~\ref{prop:cumulants-diagrams} involves pairs of partitions, it is the analogue of similar formulas for multiple stochastic integrals with respect to compensated Poisson measures~\cite{peccati-taqqu2011book}. Proposition~\ref{prop:stochmoment} 
replaces pairs of partitions with multigraphs, i.e., graphs with multiple edges. The propositions elucidate the relation between cluster expansions and diagrammatic expansions of cumulants used in stochastic geometry~\cite{peccati-taqqu2011book, last-penrose-schulte-thale2014moments}. 
As this relation provides a modern point of contact between point processes and cluster expansions, let us try to convey the main idea; details are given in Sections~\ref{sec:cumulants} and~\ref{sec:stochmoment}. Let $\eta_z$ be a homogeneous Poisson point process in a bounded region $\Lambda\subset \R^d$ with intensity parameter $z$, and $v: \mathbb R^d\times\mathbb R^d\to \R$ a symmetric function ($v(x,y) = v(y,x)$). Consider the random variable $X_z := \int_{\Lambda^2} v\dd \eta_z^{(2)}$, with $\eta_z^{(2)}$ the second factorial measure of $\eta_z$. 
The cumulants of $X_z$, if they exist, are the coefficients $\kappa_m(X_z)$ in the Taylor expansion 
$$
	\log \mathbb E\Bigl[ \e^{\beta X_z}\Bigr] = \sum_{m\geq 1} \frac{\beta^m}{m!}\, \kappa_m (X_z) \qquad (\beta \to 0).
$$
The left-hand side is closely related to the so-called \emph{grand-canonical partition function} from statistical mechanics, given by 
$$
	\Xi_\Lambda(\beta,z):= 1+ \sum_{n=1}^\infty \frac{z^n}{n!}\int_{\Lambda^n}
		\exp\Biggl( - \beta \sum_{1\leq i <j\leq n} v(x_i,x_j)\Biggr) \dd \vect x
$$
provided the integrals and the sum converge. A close look reveals that 
$$
	 \Xi_\Lambda(\beta,z) = \e^{z|\Lambda|}\, \E\Bigl[ \e^{- \beta X_z / 2} \Bigr].
$$
The theory of cluster expansions yields an expansion of the logarithm of the partition function \emph{in powers of $z$}, of the form
$$
	\log \Xi_\Lambda(\beta,z) = z|\Lambda|+ \sum_{n=2}^\infty \frac{z^n}{n!} \int_{\Lambda^n} \sum_{G\in \mathcal C_n} \prod_{\{i,j\}\in E(G)} \bigl( \e^{-\beta v(x_i,x_j) }- 1\bigr) \dd \vect x
$$
with $\mathcal C_n$ the connected graphs with vertices $1,\ldots,n$, and $E(G)$ the set of edges of $G$. In order to arrive at an expansion in powers of $\beta$ rather than $z$, we need to expand $\exp( - \beta v(x_i,x_j))$. This results in a power series in two variables, $\beta$ and $z$. The coefficient of $z^n\beta^m$ is associated with a multigraph with $m$ edges on $n$ vertices, and the relation between cluster expansions and diagrammatic expansions of cumulants is this: \emph{Cluster expansions group multigraphs according to their number of vertices, cumulants group multigraphs according to their number of edges}. 

The article is organized as follows: Section~\ref{sec:main} provides the setting and formulates the main results. Section~\ref{sec:ks} is devoted to the proof of the uniqueness theorem, based  on the Kirkwood-Salsburg equations. Section~\ref{sec:correlations} addresses the expansions of correlation functions (factorial moment densities), Section~\ref{sec:truncated} deduces expansions for the  truncated correlation functions (factorial cumulant densities) and for the log-Laplace functional. The special case of bounded volumes $\Lambda$ is dealt with in Section~\ref{sec:finite-volume}. Section~\ref{sec:trees} explains the relation between Kirkwood-Salsburg relations on the one hand and trees and forests on the other hand. In Section~\ref{sec:stochmoment} we prove the representations for stochastic integrals that we sketched above. Appendix~\ref{app:knownstuff} shows how to go from the GNZ equation to the Kirkwood-Salsburg equation, Appendix~\ref{app:existence} gives a self-contained proof of existence of Gibbs measures adapted to our setup. Appendix~\ref{app:RCM} provides some mathematical details for the random connection model discussed informally in Section~\ref{sec:boolean}.

\section{Main results} \label{sec:main}

\subsection{Setting} \label{sec:setting} 
Let $(\mathbb X,\mathrm{dist})$ be a complete separable metric space, $\mathcal X$ the Borel $\sigma$-algebra, and $\mathcal X_\mathsf b$ the collection of bounded Borel sets. Let $\lambda$ be a  reference measure on $\mathbb X$ that is locally finite, i.e., $\lambda(B)< \infty$ for all $B\in \mathcal X_\mathsf b$. A \emph{locally finite counting measure} is a measure $\eta$ on $(\mathbb X,\mathcal X)$ with $\eta(B) \in \N_0$ for all $B\in \mathcal X_\mathsf b$. Let $\mathcal N$ and $\mathcal N_\mathsf f$ be the spaces of locally finite and finite counting measures, respectively. Each non-zero $\eta\in \mathcal N$ can be written as 
$	\eta =\sum_{j=1}^\kappa \delta_{x_j}$
with $\kappa \in \N\cup \{\infty\}$ and $x_1,x_2,\ldots \in \mathbb X$. 
We define  equip $\mathcal N$ with the $\sigma$-algebra $\mathfrak N$ generated by the integer-valued random variables $\eta \mapsto \eta(B)$,  $B\in \mathcal X_\mathsf b$. 
The $n$-th factorial measure of $\eta = \sum_{j=1}^ \kappa \delta_{x_j} \in \mathcal N$ is denoted $\eta^{(n)}$. It is the unique measure on $\mathbb X^n$ such that for all measurable non-negative $g:\mathbb X^n\to \R_+$, we have 
\bes
		\int_{\mathbb X^n} g\, \dd \eta^{(n)} = \sideset{}{^{\neq}} \sum_{i_1,\ldots,i_n\leq \kappa} g(x_{i_1},\ldots,x_{i_n})
\ees
with summation over pairwise distinct indices $i_1,\ldots,i_n$.

Fix a non-negative pair potential $v$, i.e., a measurable function  $v:\mathbb X\times \mathbb X\to \R_+\cup \{\infty\}$ that is symmetric, $v(x,y) = v(y,x)$. Set $H_0=0$ and 
\be
	H_n(x_1,\ldots,x_n) = \sum_{1\leq i < j \leq n} v(x_i,x_j)\quad (x_1,\ldots,x_n\in \mathbb X)
\ee
so that $H_1(x_1) = 0$. Further define
\bes
	W(x;\eta) =\int_\mathbb X v(x,y)\dd \eta(y)
	\quad (x\in \mathbb X,\ \eta\in \mathcal N).
\ees
Cartesian products such as $\mathbb X^n$ or $\mathbb X\times \mathfrak N$ are  equipped with product $\sigma$-algebras, e.g., $\mathcal X^{\otimes n}$, $\mathcal X\otimes \mathcal N$. 
We adopt the conventions $\log 0 = -\infty$, $0\log 0 =0$, $0^0 =1$, $\exp(-\infty) =0$, and  $\exp(\infty) = \infty$. 

\begin{definition}
	Let  $z:\mathbb X\to \R_+$ be  measurable map with $\int_B z\dd \lambda < \infty$ for all $B\in \mathcal X_\mathsf b$. 	A probability measure $\mathsf P$ on $(\mathcal N,\mathfrak N)$ is a Gibbs measure with pair interaction $v(x,y)$ and activity $z$  if 
	\be \label{eq:gnz} \tag{GNZ}
		\mathsf E\Bigl[ \int_{\mathbb X} F(x,\eta) \dd \eta(x)\Bigr] = \int_\mathbb X \mathsf E\Bigl[ F(x,\eta+ \delta_x)\e^{-W(x;\eta)}\Bigr] z(x) \dd \lambda(x)
	\ee
	for all measurable $F:\mathbb X\times \mathcal N \to[0,\infty)$. The set of Gibbs measures is denoted  $\mathscr G(z)$.
\end{definition} 

\noindent The GNZ equation named after Georgii, Nguyen and Zessin is equivalent to the DLR  conditions (named after Dobrushin, Lanford, Ruelle) more familiar to mathematical physicists, see \cite{georgii76,nguyen-zessin79}. It is a probabilistic cousin of the Kirkwood-Salsburg equations~\cite[Chapter 4.2]{ruelle1969book} for correlation functions recalled in Lemma~\ref{lem:ks} below. 
The GNZ equation is also intimately related to the detailed balance equation for the Markovian birth and death process with formal generator 
\be
	(Lg)(\eta) = \int_{\mathbb X} \bigl[ g(\eta - \delta_x) - g(\eta)\bigr] \dd \eta(x) + \int_\mathbb X z(x) \e^{- W(x;\eta)} \bigl[ g(\eta+ \delta_x) - g(\eta)\bigr] \dd \lambda(x),
\ee
sometimes called \emph{continuum Glauber dynamics}, see~\cite{preston1975birthanddeath,kondratiev-lytvynov2005} for the construction of such processes and~\cite{fernandez-ferrari-garcia2001lossnetwork,bertini-cancrini-cesi2003, kondratiev-kuna-ohlerich2013, schuhmacher-stucki2014, fernandez-groisman-saglietti2016} for some applications. When the interaction vanishes ($v\equiv 0$), the GNZ equation reduces to the Mecke formula for the Poisson point process with intensity measure $\lambda_z$ and, in the context of Poisson-Malliavin calculus the operator $L$ and the generated semi-group are called \emph{Ornstein-Uhlenbeck operator} and \emph{Mehler semi-group} ~\cite{last2016stochana}. 

The $n$-th \emph{factorial moment density} $\rho_n(x_1,\ldots,x_n)$, called \emph{$n$-point correlation function} in statistical mechanics, is the Radon-Nikod{\'y}m derivative $\rho_n = \frac{\dd \alpha_n}{\dd \lambda^n}$ of the $n$-th factorial moment measure $\alpha_n$ of $\mathsf P$~\cite{daley-verejones2008vol2, last-penrose2017book}. Thus 
\bes
	\mathsf{E} \Bigl[ \int_{\mathbb X^n} g\, \dd \eta^{(n)}\Bigr] = \int_{\mathbb X^n} g \rho_n \dd \lambda^n
\ees
for all non-negative measurable $g$. The correlation functions are symmetric, i.e.,  invariant with respect to permutation of the arguments.  For Gibbs measures, they admit the following concrete representation. 

\begin{lemma} \label{lem:correp}
	Let $\mathsf P\in \mathscr G(z)$. Then the $n$-point functions $\rho_n$ exist and satisfy 
	\be\label{eq:correp}
		\rho_n(x_1,\ldots,x_n) =z(x_1)\cdots z(x_n)\, \e^{-H_n(x_1,\ldots,x_n)}  \int_\mathcal N \e^{- \sum_{i=1}^n W(x_i;\eta)}  \dd \mathsf P(\eta)
	\ee
	for all $n\in \N$ and $\lambda^n$-almost all $(x_1,\ldots,x_n)\in \mathbb X^n$.
\end{lemma} 

\noindent The lemma allows us to adopt Eq.~\eqref{eq:correp} as the definition of the correlation functions, thus removing indeterminacies on null sets from the definition as a Radon-Nikod{\'y}m derivative. The lemma is well-known,
 for the reader's convenience we present a proof in Appendix~\ref{app:knownstuff}. A simple consequence that we use repeatedly is the estimate 
\be \label{eq:nonnegbound}
	\rho_n(x_1,\ldots,x_n) \leq z(x_1)\cdots z(x_n). 
\ee
The \emph{factorial cumulant densities} $\rho_n^\mathsf T(x_1,\ldots,x_n)$, called \emph{truncated correlation functions} in statistical mechanics, are uniquely defined by the requirement that they are symmetric and satisfy, for all $n\in \N$,
\be \label{eq:truncated-def}
	\rho_n(x_1,\ldots,x_n) = \sum_{r=1}^n \sum_{\{V_1,\ldots,V_r\}\in \mathcal P_n} \prod_{k=1}^r \rho_{\#V_k}^\mathsf T \bigl( (x_i)_{i\in V_k}\bigr) 
\ee
where $\mathcal P_n$ is the collection of set partitions of $\{1,\ldots,n\}$. For example, $\rho_1^\mathsf T(x_1) = \rho_1(x_1)$ and $\rho_2^\mathsf T(x_1,x_2) = \rho_2(x_1,x_2) - \rho_1(x_1) \rho_1(x_2)$.  

\begin{example}[Mixture of hard spheres and Poisson exclusion process with random radii]
	Let $\mathbb X = \R^d \times \R_+$, equipped with the Euclidean distance, and $\lambda$ the Lebesgue measure. Consider the interaction $v((x,r),(y,R)) := \infty\,\1_{\{|x-y|\leq R+r\}}$ so that 
	\bes
		\e^{- v((x,r),(y,R))} = \1_{\{ B_d(x,r) \cap B_d(y,R) = \varnothing\}}
	\ees
	with $B_d(x,r)\subset \R^d$ the closed ball of radius $r$ centered at $x$. Let $\mathsf P\in \mathscr G(z)$. Any $\eta\in \mathcal N$ can be written as $\eta = \sum_{j=1}^\kappa \delta_{(x_j,r_j)}$, and we have $B_d(x_i,r_i)\cap B_d(x_j,r_j) =\varnothing$ for all $i\neq j$, $\mathsf P$-almost surely. The measure describes a mixture of spheres of different radii, distinct spheres are not allowed to overlap. 
 $\mathsf P$ is the distribution of a marked point process on $\R^d$, with marks in $\R_+$. A priori, the points are Poisson distributed with intensity $z_0(x) = \int_0^\infty z(x,r) \dd r$ (assuming that each $z_0(x)$ is finite), and sphere at point $x$ has a random radius with probability density function $r\mapsto p(x,r)$ given by $p(x,r) = z(x,r) /z_0(x)$. We can think of $\mathsf P$, roughly, as the a priori distribution conditioned on non-overlap of the spheres. 
\end{example} 

\noindent The example is interesting because of deep connections between mixtures of hard spheres and Boolean percolation, see~\cite{hofertemmel-houdebert2017,torquato2012} and    Section~\ref{sec:boolean}. 

\subsection{Uniqueness and convergent expansions in infinite volume} \label{sec:mainmain}
 
Our main results are  (i) a sufficient condition for the uniqueness of Gibbs measures, and (ii) expansions of the log-Laplace functional, correlation functions and truncated correlation functions (factorial moment and factorial cumulant densities) of the Gibbs measure. Our sufficient condition for the absolute convergence of the expansions generalizes a condition by Fern{\'a}ndez and Procacci~\cite{fernandez-procacci2007} for hard-core interactions and discrete spaces $\mathbb X$ to non-negative interactions and more general spaces $\mathbb X$. 
 
Define \emph{Mayer's $f$-function} 
\bes
	f(x,y) := \e^{ - v(x,y)} -1\quad (x,y\in \mathbb X). 
\ees
Our first result is a sufficient criterion for  uniqueness of Gibbs measures. The criterion also implies convergence of the expansions and analyticity of generating functions. The reader primarily intereted in uniqueness should thus be warned that the condition might be way too strong for his needs and other approaches may work in a bigger domain, see the discussion in~\cite{fernandez-ferrari-garcia2001lossnetwork}.

\begin{theorem} \label{thm:uniqueness}
	Assume $f(x,y) \leq 0$ for $\lambda^2$-almost all $(x,y)\in \mathbb X^2$ and $\int_B z\dd \lambda < \infty$ for all $B\in \mathcal X_\mathsf b$. Suppose that there exists a measurable function $a:\mathbb X\to \R_+$ and some $t>0$ such that 
	\be \label{eq:suff1}
		 \sum_{k=1}^\infty \frac{\e^{tk} }{k!}\int_{\mathbb X^k} \prod_{j=1}^k  |f(x_0,y_j)| \prod_{1\leq i < j \leq k} (1+ f(y_i,y_j)) \prod_{j=1}^k z(y_j) \e^{a(y_j)} \dd \lambda^k(\vect y) \leq \e^{a(x_0)} - 1\tag{$\mathsf{FP}_t$}
	\ee
	for $\lambda$-almost all $x_0\in \mathbb X$. Then $\#\mathscr G(z) = 1$.
\end{theorem} 

\noindent The theorem is proven in Section~\ref{sec:ks}, results for $t=0$ are provided in Section~\ref{sec:fivo}.  In view of $1+f(y_i,y_j)\leq 1$, a sufficient condition for~\eqref{eq:suff1} to hold true is that 
\be \label{eq:suff2} 
	 \e^t\int_\mathbb X |f(x_0,y)|\e^{a(y)} z(y) \dd \lambda(y) \leq a(x_0)\tag{$\mathsf{KPU}_t$}
\ee
for $\lambda$-almost all $x_0\in \mathbb X$. Condition~\eqref{eq:suff2} with $t=0$ is discussed in depth in~\cite{ueltschi2004cluster}; it is a continuum version of the Koteck{\'y}-Preiss condition for polymer systems~\cite{kotecky-preiss1986}. Condition~\eqref{eq:suff2} with $t>0$ is a special case of  Assumption~3 in~\cite{poghosyan2013}, who considered both non-negative and attractive potentials. Condition~\eqref{eq:suff1} with $t=0$ corresponds to the criterion for discrete polymer models by Fern{\'a}ndez and Procacci~\cite{fernandez-procacci2007} and for hard spheres by Fern{\'a}ndez, Procacci and Scoppola~\cite{fernandez-procacci-scoppola2007}. 

\begin{example}[Hard spheres / Poisson exclusion process] \label{ex:hard-spheres} 
	Choose $\mathbb X = \R^d$ with the Euclidean norm $|\cdot|$ and the Lebesgue measure $\lambda = \mathrm{Leb}$. Let $v(x,y) =\infty \1_{\{|x-y|\leq r\}}$ with $r>0$. Suppose that 
	\be \label{eq:hard-spheres} 
		 z\,\mathrm{Leb}(B(0,r))<1 /\mathrm{e}\simeq 0.36787
	 \ee
   Because of $1/\mathrm e = \max_{a>0} a \exp(-a)$, we can find $a,t>0$ such that 
\bes
	\e^t \int_{\mathbb R^d} |f(x,y)| \e^a z\dd y = \e^t	 \mathrm{Leb}\, (B(0,r)) \e^a \leq a
\ees
and conditions~\eqref{eq:suff2} and~\eqref{eq:suff1} holds true for some constant function $a(y) = a$.
The criterion~\eqref{eq:hard-spheres} is well-known~\cite[Chapter 4]{ruelle1969book}. As shown in~\cite{fernandez-procacci-scoppola2007}, for a two-dimensional gas of hard spheres, condition~\eqref{eq:suff1} allows one to improve~\eqref{eq:hard-spheres} to 
$z\,\mathrm{Leb}(B(0,r))<0.5107$. 
\end{example}

The condition~\eqref{eq:suff1} is enough to ensure that generating functionals and correlation functions admit convergent expansions; for homogeneous models (constant activity $z(x) =z$), the expansions are power series in $z$. Furthermore the expansion coefficients can be expressed as sums over weighted graphs. Let $\mathcal G_n$ be the collection of graphs $G=(V,E)$ with vertex set $V=\{1,\ldots,n\}=[n]$ and edge set $E\subset \{\{i,j\}\mid i,j\in V,\, i \neq j\}$. We write $E(G) =E$ for the edge set of $G$. We define graph weights 
\be
	w(G;x_1,\ldots,x_n):= \prod_{\{i,j\} \in E(G)} f(x_i,x_j)
\ee
Let $\mathcal C_n\subset \mathcal G_n$ be the collection of connected graphs with vertex set $[n]$ and 
\be \label{eq:ursell}
	\varphi_n^\mathsf T(x_1,\ldots,x_n):= \sum_{G\in \mathcal C_n} w(G;x_1,\ldots,x_n)
\ee
the $n$-th \emph{Ursell function}. 
Let $\lambda_z$ be the measure on $\mathbb X$ that is absolutely continuous with respect to $\lambda$ with Radon-Nikod{\'y}m derivative $z(\cdot)$. For $h:\mathbb X\to \R \cup\{\infty\}$, we introduce the integrability condition
\be \label{eq:hcondition} 
	\int_\mathbb X |\e^{-h(x)}-1| \e^{a(x)} \dd \lambda_z(x) < \infty.
\ee

\begin{theorem} \label{thm:log-laplace}
	Under the conditions of Theorem~\ref{thm:uniqueness}, the log-Laplace functional of the unique Gibbs measure $\mathsf P\in \mathscr G(z)$ satisfies  
	\be \label{eq:log-laplace}
		\log \mathsf E\Bigl[ \e^{-\int_\mathbb X h\dd \eta}\Bigr]	
			= \sum_{n=1}^\infty \frac{1}{n!} \int_{\mathbb X^n}\bigl( \e^{-\sum_{j=1}^n h(x_j)} - 1\bigr) \varphi_n^\mathsf T(x_1,\ldots,x_n)  \dd \lambda_z^n(\vect x)
	\ee
	for all measurable $h:\mathbb X\to [-t,\infty) \cup \{\infty\}$ or $h:\mathbb X\to \{ s \in \mathbb C\mid \, \Re s \geq - t \}$ that satisfy~\eqref{eq:hcondition}. Moreover 	
	\be \label{eq:keyconv}
		\sum_{n=2}^\infty \frac{\e^{t(n-1)}}{(n-1)!}\int_{\mathbb X^{n-1}} |\varphi_n^\mathsf T(x_1,\ldots,x_n)|  \dd \lambda_z(x_2)\cdots \dd \lambda_z(x_n) \leq \e^{a(x_1)}-1
	\ee
	for $\lambda$-almost all $x_1\in \mathbb X$. 	 
\end{theorem} 

\noindent 
The convergence of the right-hand side of~\eqref{eq:log-laplace} follows from the inequality~\eqref{eq:keyconv} and the estimate 
\be \label{eq:trick}
	\bigl|\e^{-\sum_{j=1}^n h(x_j)} - 1\bigr| = \bigl| \sum_{j=1}^n (\e^{- h(x_j)} - 1) \e^{ - \sum_{i=1}^{j-1} h(x_i)}\bigr| \leq \e^{(n-1) t} \sum_{j=1}^n \bigl|\e^{- h(x_j)}-1\bigr|.
\ee

%

\begin{remark} [Cumulants] 
Complex parameters allow us to extract bounds on cumulants via contour integrals. For example, let $\Lambda\in \mathcal X$. 
Applying Theorem~\ref{thm:log-laplace} to $- h = s N_\Lambda$ with $|s|\leq t$, we obtain 
\bes
	\log \mathsf E\Bigl[\e^{sN_\Lambda}\Bigr] = \sum_{\ell=1}^\infty \frac{\kappa_\ell}{\ell!} s^\ell 
\ees
with 
\bes
	|\kappa_\ell| = \ell ! \Bigl| \frac{1}{2\pi \mathrm{i}}\oint \log \mathsf E\Bigl[\e^{sN_\Lambda}\Bigr] \frac{\dd s}{s^{\ell+1}}\Bigr| \leq \frac{\ell!}{t^\ell} \int_{\Lambda} \e^a \dd \lambda_z. 
\ees
See \cite{bryc1993} for applications to central limit theorems. 
\end{remark}

\begin{remark}[Signed L{\'e}vy measure]
Eq.~\eqref{eq:log-laplace} can be rewritten as
\be \label{eq:log-laplace2}
	\log \mathsf E\Bigl[ \e^{-\int_\mathbb X h\dd \eta}\Bigr]	
			= \int_{\mathcal N_\mathsf f} (\e^{-\int_\mathbb X h\dd \eta} - 1) \dd \Theta(\eta)
\ee
where the signed measure $\Theta$ on $\mathcal N_\mathsf f$ is defined by 
\be
	\int_{\mathcal N_\mathsf f} g\dd \Theta = \sum_{n=1}^\infty \frac{1}{n!}\int_{\mathbb X^n} g(\delta_{x_1}+\cdots + \delta_{x_n})  \varphi_n^\mathsf T(x_1,\ldots,x_n) \dd \lambda_z^n(\vect x).
\ee
Log-Laplace transforms such as~\eqref{eq:log-laplace2}---but with non-negative measures $\Theta$---appear naturally in the context of \emph{cluster point processes} and \emph{infinitely divisible point processes} \cite[Chapter 10.2]{daley-verejones2008vol2}. Following Nehring, Poghosyan and Zessin~\cite{nehring-poghosyan-zessin2013} we may call $\Theta$ the \emph{pseudo-} or \emph{signed L{\'e}vy measure} of $\mathsf P$ (compare Eq.~\eqref{eq:levy} below). 
\end{remark}

\begin{theorem} \label{thm:truncated}
	Under the conditions of Theorem~\ref{thm:uniqueness}, the truncated correlation functions (factorial cumulant densities) of the unique Gibbs measure $\mathsf P \in \mathscr G(z)$ satisfy 
	\begin{multline} \label{eq:truncated}
		\rho_n^\mathsf T(x_1,\ldots,x_n) \\= \prod_{j=1}^n z(x_j)  \left(\varphi_{n}^\mathsf T(x_1,\ldots,x_n)+ \sum_{k=1}^\infty \frac{1}{k!}\int_{\mathbb X^k} \varphi_{n+k}^\mathsf T(x_1,\ldots,x_n,y_1,\ldots,y_k) \dd \lambda_z^k(\vect y)\right)
	\end{multline}
	for all $n\in \N$ and $\lambda^n$-almost all $(x_1,\ldots,x_n) \in \mathbb X^n$, with 
	\be \label{eq:truncated-convergence}
		\sum_{k=0}^\infty \frac{1}{k!}\int_{\mathbb X^{n-1+k}} |\varphi_n^\mathsf T(x_1,\ldots,x_{n+k})| \dd \lambda_z(x_2)\cdots \dd \lambda_z(x_{n+k}) \leq \frac{(n-1)!}{(\e^t- 1)^{n-1}}\,  \e^{a(x_1)}. 
	\ee
\end{theorem} 

\noindent The theorem  is proven in Section~\ref{sec:truncated}. The bound~\eqref{eq:truncated-convergence} proves the convergence of the expansion~\eqref{eq:truncated} in a suitable $L^1$-norm, in agreement with the spaces used by Kuna and Tsagkarogiannis~\cite{kuna-tsagkaro2016} and Last and Ziesche~\cite{last-ziesche2017} for the two-point function $\rho_2^\mathsf T$. Pointwise estimates better suited to estimating decay of correlations and mixing properties can be found in~\cite{ueltschi2004cluster}, see also~\cite[Corollary 5.2]{hanke2018} and the references therein for the two-point function. 

\begin{remark} [Positivity of $t$] For the bound~\eqref{eq:truncated-convergence} the strict positivity of $t$ is essential. Something similar happens for the factorial cumulants of random variables: Let $N$ be an $\N_0$-valued random variable that is infinitely divisible with finite L{\'e}vy measure $\nu =\sum_{n\in \N} \nu_n \delta_n$, then 
\be \label{eq:levy}
	\log \E\bigl[\e^{-\lambda N} \bigr] = \sum_{n=1}^ \infty (\e^{-\lambda n} - 1) \nu_n \qquad (\lambda\geq 0).
\ee
The factorial cumulants $(\alpha_n)_{n\in \N}$, if they exist, are defined as the coefficient in the asymptotic expansion 
$$
	\log \E\bigl[\e^{-\lambda N} \bigr] = \sum_{n=1}^ \infty \frac{\alpha_n}{n!} \bigl( \e^{-\lambda} - 1\bigr)^n \qquad (\lambda \searrow 0). 
$$
If the series $\sum_n \nu_n r^n$ has a radius of convergence $\e^t>1$, then $u\mapsto \log \E[(1+u)^N]$ is analytic in $|u|<\e^t-1$, which leads to bounds on its Taylor coefficients $\alpha_n$. If on the other hand $\sum_n \nu_n r^n$ diverges for $r>1$, then nothing can be said about the $\alpha_n$'s without additional information on the $\nu_n$'s. 
\end{remark}

 The correlation functions are associated with graphs subject to slightly weaker connectivity constraints. For $k \leq n$, let $\mathcal D_{k,n}\subset \mathcal G_n$ be the set of multi-rooted graphs $G$ with vertices $1,\ldots,n$, roots $1,\ldots,k$, such that every non-root vertex $j\in \{k+1,\ldots,n\}$ is connected to some root vertex $i\in\{1,\ldots,k\}$ by a path in $G$.  Notice $\mathcal D_{n,n} = \mathcal G_n$ and $\mathcal D_{1,n} = \mathcal C_n$. Set 
\be \label{eq:psidef}
	\psi_{k,n}(x_1,\ldots,x_n) = \sum_{G\in \mathcal D_{k,n}}w(G;x_1,\ldots,x_n).
\ee

\begin{theorem} \label{thm:correlations}
	Under the conditions of Theorem~\ref{thm:uniqueness}, the correlation functions (factorial moment densities) of the unique Gibbs measure $\mathsf P \in \mathscr G(z)$ are given by 
	\begin{multline}\label{eq:correlations-exp} 
		\rho_n(x_1,\ldots,x_n) \\= \prod_{j=1}^n z(x_j)  \left(\psi_{n,n}(x_1,\ldots,x_n)+ \sum_{k=1}^\infty \frac{1}{k!}\int_{\mathbb X^k} \psi_{n,n+k}(x_1,\ldots,x_n,y_1,\ldots,y_k) \dd \lambda_z^k(\vect y)\right)
	\end{multline}
	with 
	 \begin{multline}\label{eq:correlations-bound}
	 	\bigl|\psi_{n,n}(x_1,\ldots,x_n)\bigr|+ \sum_{k=1}^\infty \frac{\e^{tk}}{k!}\int_{\mathbb X^k} \bigl|\psi_{n,n+k}(x_1,\ldots,x_n,y_1,\ldots,y_k)\bigr| \dd \lambda_z^k(\vect y)\\
	 		\leq  \prod_{1\leq i<j\leq n}(1+f(x_i,x_j)) \e^{\sum_{i=1}^n a(x_i)}. 
	 \end{multline} 
\end{theorem} 

\noindent The theorem is proven at the end of Section~\ref{sec:correlations}. \\

\noindent Let us briefly comment on the order in which Theorems~\ref{thm:uniqueness}, \ref{thm:log-laplace}, \ref{thm:truncated} and~\ref{thm:correlations} are proven. Theorem~\ref{thm:uniqueness} about the uniqueness of Gibbs measures is proven first with the help of Kirkwood-Salsburg equations, in Section~\ref{sec:ks}. Theorem~\ref{thm:correlations} is proven next, in Section~\ref{sec:correlations}, building on the Kirkwood-Salsburg equations introduced in Section~\ref{sec:ks}. The representation for one-point correlation functions is then used to prove Theorem~\ref{thm:log-laplace} in Section~\ref{sec:truncated}, and finally Theorem~\ref{thm:truncated} is deduced from Theorem~\ref{thm:log-laplace}, also in Section~\ref{sec:truncated}. 

\subsection{About the case $t=0$} \label{sec:fivo}

The strict positivity $t>0$ in condition~\eqref{eq:suff1} ensures that a certain linear operator $\vect K_z$ has operator norm $||\vect K_z||\leq \e^{-t}<1$, leading to the uniqueness of a fixed point problem underpinning the uniqueness of infinite-volume Gibbs measure (see Section~\ref{sec:ks}). If $t=0$, then the contractivity of the operator is lost and our approach no longer guarantees uniqueness of the infinite-volume Gibbs measure. Nevertheless, the condition~\eqref{eq:suff1} with $t=0$ stays useful. Roughly, the condition ensures that the cluster expansions for the correlation functions of Gibbs measures in finite volume $\Lambda$ (with empty boundary conditions) are absolutely convergent, uniformly in the volume $\Lambda$. The uniformity in the volume guarantees the existence of the infinite-volume limit. This reasoning is fairly standard in mathematical physics.

Let $\Lambda\in \mathcal X_\mathsf b$ be a bounded set (thus  $\lambda_z(\Lambda) < \infty$, ``finite volume'').  Assume $\lambda_z(\Lambda)>0$.  Let $\mathcal N_\Lambda = \{\eta\in\mathcal N\mid \forall A\in \mathcal X:\, \eta(A) = \eta(A\cap \Lambda) \}$ be the space of point configurations in $\Lambda$.  The space $\mathcal N_\Lambda$ is equipped with the trace of the $\sigma$-algebra $\mathfrak N$.
It is not too difficult to see that there is a uniquely defined measure $\mathsf P_\Lambda$ 
such that 
$$
	\mathsf E_\Lambda\Bigl [\int_\Lambda F(x,\eta) \dd \eta(x)\Bigr] = \int_\Lambda \mathsf E_\Lambda \Bigl[ F(x,\eta+ \delta_x) \e^{- W(x,\eta)} \Bigr] z(x) \dd \lambda(x)
$$
for all measurable $F:\Lambda\times \mathcal N\to [0,\infty)$. Moreover, the measure $\mathsf P_\Lambda$ is absolutely continuous with respect to the distribution of the Poisson point process with intensity measure $ z\1_\Lambda\,\dd \lambda$, with a Radon-Nikod{\'y}m derivative proportional to $\exp( - H(\eta))$, where $H(\delta_{x_1}+\cdots+ \delta_{x_n}) = H_n(x_1,\ldots,x_n)$. The measure $\mathsf P_\Lambda$ is usually called a finite-volume Gibbs measure with empty boundary conditions (because $H(\eta)$ does not include interactions with points outside $\Lambda$).

We formulate analogues of Theorems~\ref{thm:log-laplace} and~\ref{thm:correlations} before commenting on Theorem~\ref{thm:uniqueness}.  As explained earlier, Theorem~\ref{thm:truncated} really needs the condition $t>0$ and it has no direct  counterpart for $t=0$ (see however~\cite{ueltschi2004cluster}).

\begin{theorem} \label{thm:log-laplace-fivo}
	Assume $v(x,y)\geq 0$ for $\lambda^2$-almost all $(x,y)\in \mathbb X^2$ and $\int_B z \dd \lambda<\infty$ for all $B\in \mathcal X_\mathsf b$. Suppose that there exists a measurable function $a:\mathbb X\to \R_+$ such that condition~\eqref{eq:suff1} holds true with $t=0$. Then 
	\be \label{eq:fivo2}
		\sum_{n=1}^\infty \frac{1}{n!}\int_{\mathbb X^{n}} |\varphi_n^\mathsf T(x_0,\ldots,x_n)|  \dd \lambda_z(x_1)\cdots \dd \lambda_z(x_n) \leq \e^{a(x_0)}-1
	\ee
	for $\lambda$-almost all $x_0\in \mathbb X$. 	Moreover, for every non-empty $\Lambda \in \mathcal X_\mathsf b$, and all measurable $h:\Lambda\to [0,\infty]$ (or $h$ complex-valued with non-negative real part) with $\int_\Lambda |\e^{-h} - 1| \e^a \dd \lambda_z <\infty$, we have 
	\be \label{eq:log-laplace-fivo}
		\log \mathsf E_\Lambda\Bigl[ \e^{-\int_\Lambda h\dd \eta}\Bigr]	
			= \sum_{n=1}^\infty \frac{1}{n!} \int_{\Lambda^n}\bigl( \e^{-\sum_{j=1}^n h(x_j)} - 1\bigr) \varphi_n^\mathsf T(x_1,\ldots,x_n)  \dd \lambda_z^n(\vect x).
	\ee
\end{theorem} 

\begin{theorem} \label{thm:finite-volume-corr}
	Under the conditions of Theorem~\ref{thm:log-laplace-fivo}, the bound~\eqref{eq:correlations-bound} holds true for $t=0$ and integration over $\mathbb X^n$, and the $n$-point functions $\rho_{n,\Lambda}$ of $\mathsf P_\Lambda$ are given by~\eqref{eq:correlations-exp} with integrals over $\Lambda^n$ instead of $\mathbb X^n$. 
\end{theorem} 

\noindent Theorem~\ref{thm:log-laplace-fivo} and~\ref{thm:finite-volume-corr} are  proven in Section~\ref{sec:finite-volume}. Notice that the convergence estimate~\eqref{eq:fivo2} really has integrals over $\mathbb X^{n}$ (not $\Lambda^{n}$), and the expansions for the log-Laplace transform and correlation functions depend on  $\Lambda$ only through the domains of integration. As a consequence, we can pass to the infinite-volume limit: Let $(\Lambda_n)_{n\in\N} $ be a sequence of non-empty sets in  $\mathcal X_\mathsf b$ with $\Lambda_n \nearrow \Lambda$. Then for 
all measurable $h:\Lambda\to [0,\infty]$ (or $h$ complex-valued with non-negative real part) with $\int_\mathbb X |\e^{-h} - 1| \e^a \dd \lambda_z <\infty$, we have 
	\be \label{eq:lapla-co}
		\lim_{n\to \infty} \log \mathsf E_{\Lambda_n}\Bigl[ \e^{-\int_{\Lambda_n} h\dd \eta}\Bigr]	
			= \sum_{n=1}^\infty \frac{1}{n!} \int_{\mathbb X^n}\bigl( \e^{-\sum_{j=1}^n h(x_j)} - 1\bigr) \varphi_n^\mathsf T(x_1,\ldots,x_n)  \dd \lambda_z^n(\vect x).
	\ee
A similar convergence for correlation functions $\rho_{n,\Lambda_n}$ holds true as well, the precise statement is left to the reader. In view of~\eqref{eq:lapla-co}, we should expect that the sequence $(\mathsf P_{\Lambda_n})_{n\in \N}$ converges in some suitable sense to a uniquely defined probability measure $\mathsf P$ on $(\mathcal N,\mathfrak N)$. The measure $\mathsf P$ should be in $\mathscr G(z)$ and its log-Laplace functional (restricted to the functions $h$ under consideration) should be given by the right-hand side of~\eqref{eq:lapla-co}. A rigorous statement and proof are beyond the scope of this article, but we emphasize that the uniqueness of the infinite-volume limit of Gibbs measures with empty boundary conditions replaces the uniqueness in Theorem~\ref{thm:uniqueness}. 

Mathematical physicists often work with the  \emph{grand-canonical partition function} 
\bes 
	\Xi_\Lambda(z) = 1+ \sum_{n=1}^\infty \frac{1}{n!} \int_{\Lambda^n}  \e^{- H_n(x_1,\ldots,x_n)} \dd \lambda_z^n(\vect x).
\ees
instead of log-Laplace transforms. 

\begin{theorem} \label{thm:finite-volume} 
		Under the conditions of Theorem~\ref{thm:log-laplace-fivo}, if $\Lambda\in \mathcal X_\mathsf b$ with $\int_\Lambda \e^a \dd \lambda_z <\infty$, then 
	\bes
		\log \Xi_\Lambda(z) = \sum_{n=1}^\infty \frac{1}{n!} \int_{\mathbb X^n}\varphi_n^\mathsf T(x_1,\ldots,x_n)  \dd \lambda_z^n(\vect x)
	 \ees
	 with absolutely convergent integrals and series.	
\end{theorem} 

\noindent The theorem is proven in Section~\ref{sec:finite-volume}. For $\int_\Lambda \e^a \dd \lambda_z<\infty$ and measurable, non-negative $h$, we have 
 \bes
 	\mathsf E_\Lambda\Bigl[\e^{-\int_\Lambda h \dd \eta}\Bigr] = \log \Xi_\Lambda( z\,\e^{-h}) - \log \Xi_\Lambda(z),
 \ees
and the formula~\eqref{eq:log-laplace-fivo} is easily recovered. 

\subsection{Attractive pairwise interactions} \label{sec:attractive} 
The results proven in this article cover non-negative pairwise interactions only, for the reader's convenience we summarize some known results on attractive pairwise interactions. Let $v:\mathbb X\times \mathbb X\to \R\cup\{\infty\}$ with $v(x,y) = v(y,x)$ on $\mathbb X\times \mathbb X$. The pair potential $v$ is called \emph{stable}~\cite[Chapter 3.2]{ruelle1969book} if for some constant $B\geq 0$ and all $n\in \N$ and $x_1,\ldots, x_n\in \mathbb X$, we have
$$
	\sum_{1\leq i < j \leq n} v(x_i,x_j) \geq - B n. 
$$
Sometimes it is more convenient to allow for $x$-dependent $B$, and replace the previous condition by 
\be \label{eq:inhom-stable}
	\sum_{1\leq i < j \leq n} v(x_i,x_j) \geq - \sum_{i=1}^n B(x_i)
\ee
for some measurable function $B:\mathbb X\to \R_+$ and all $n\in \N$ and all $x_1,\ldots, x_n\in \mathbb X$, see~\cite{poghosyan-ueltschi2009}. A sufficient condition for stability in the sense of~\eqref{eq:inhom-stable}  is \emph{local stability}, a notion somewhat more popular in stochastic geometry. Suppose that 
\be \label{eq:locally-stable} 
	\int_\mathbb X v(x,y) \dd \eta(y) \geq - C(x)
\ee
for some function $C:\mathbb X\to \R_+$ and all finite configurations $\eta = \sum_{i=1}^k \delta_{y_i}$ such that $\eta + \delta_x$ has finite total energy. Then for all $n\in \N$ and $x_1,\ldots, x_n\in \mathbb X$, we have 
$$
	\sum_{1\leq i < j \leq n} v(x_i,x_j)  = \frac12 \sum_{i=1}^n  \sum_{\substack{1\leq j \leq n:\\ j \neq i}} v(x_i,x_j) \geq - \frac12 \sum_{i=1}^n C(x_i). 
$$
(Notice that the inequality is trivial when the left-hand side is infinite.) Thus local stability~\eqref{eq:locally-stable} implies stability in the sense of~\eqref{eq:inhom-stable}. 

The following result is due to Procacci and Yuhjtman~\cite{procacci-yuhjtman2017}. Assume that $v$ is stable with stability function $B(x)$, and that 
\be \label{eq:proyu} 
	\int_\mathbb X \bigl( 1- \e^{- |v|} \bigr) \e^{B(y) + a(y)}  \dd \lambda_z(y) \leq a(x)
\ee
for some measurable function $a:\mathbb X\to \R_+$ and $\lambda$-almost all $x\in \mathbb X$. Then
$$
	\sum_{n=1}^\infty \frac{1}{n!}\int_{\mathbb X^n} \bigl| \varphi_n^\mathsf T(q,x_1,\ldots, x_n)\bigr|\, \dd \lambda_z^n(\vect x) \leq \e^{B(q)} \bigl(\e^{a(q)} - 1\bigr)
$$
for $\lambda$-almost all $q\in \mathbb X^n$, and a suitable analogue of Theorem~\ref{thm:finite-volume} (hence, Theorems~\ref{thm:log-laplace-fivo} and~\ref{thm:finite-volume-corr}) holds true. The proof in~\cite{procacci-yuhjtman2017} is written for $\mathbb X= \R^d$, $\lambda$ the Lebesgue measure, and translationally invariant pair potentials, building on a  novel tree-graph inequality. The result is easily extended to the general setup by using the formulation of the tree-graph inequality from~\cite{ueltschi2017}. For further results on attractive pairwise interactions, see~\cite{poghosyan-ueltschi2009} and the references therein. 

The condition~\eqref{eq:proyu} extends~\eqref{eq:suff2} for $t=0$; we leave as an open problem the  extension of~\eqref{eq:suff1} with $t=0$ to stable pair potentials. We also leave as an open question the extension of~\eqref{eq:proyu} to $t>0$, i.e., whether multiplying the left-hand side of~\eqref{eq:proyu} with $\e^t$ for some $x$-independent $t>0$ yields a sufficient condition for the uniqueness of the infinite-volume Gibbs measure as in Theorem~\ref{thm:uniqueness}, and mention instead that related uniqueness criteria based on Kirkwood-Salsburg equations are available~\cite[Chapter~4.2]{ruelle1969book}.

\subsection{Trees and branching processes} \label{sec:branching}

One of the standard techniques to prove convergence of cluster expansions is by \emph{tree-graph estimates}:  sums over graphs are estimated by sums over trees with the help of \emph{tree partition schemes}, see~\cite{fernandez-procacci2007,poghosyan-ueltschi2009} and the references therein. Even though we follow another standard route, the method of integral equations, it is helpful to interpret the convergence conditions in the context of trees: Propositions~\ref{prop:kptrees} and~\ref{prop:fptrees} below say that the convergence conditions are in fact \emph{equivalent} to the convergence of certain generating functions for trees---convergence conditions mirror fixed point equations for trees. This observation is not entirely new~\cite{fernandez-procacci2007, faris2010combinatorics} but usually the focus is on sufficient convergence conditions and the equivalence is rarely explicitly stated (see, however,~\cite{jansen2015tonks}). 

Trees are interesting because they are related to branching processes: the main result of this section says that the convergence condition for tree generating functions (and cluster expansions) as formulated by~\cite{ueltschi2004cluster} implies extinction of a related multi-type, discrete-time branching process with type space $\mathbb X$. The branching process is of interest because extinction of the branching process implies absence of percolation in a random connection model, see Section~\ref{sec:boolean}.

\subsubsection{Convergence of generating functions for trees}
Let $\mathcal T_n$ be the set of trees with vertex set  $[n]=\{1,\ldots,n\}$, 
 $\mathcal T_n^\bullet = \mathcal T_n \times [n]$ the rooted trees with vertex set $[n]$, and $\mathcal T_n^\circ$ the set of trees with vertex set $\{0,1,\ldots,n\}$. The notation with black and white circles is borrowed from the theory of combinatorial species: black circles refer to roots or the operation of ``pointing''~\cite{bergeron-labelle-leroux1998book}, white circles to ``ghosts'' that do not come with powers of $z$ in the generating functions.  Define a family $(T_q^\circ)_{q\in \mathbb X}$ of generating functionals 
 \be \label{eq:tree-ghost}
 	T_q^\circ (z) := 1+ \sum_{n=1}^\infty \frac{1}{n!}\int_{\mathbb X^n}\sum_{T\in\mathcal T_n^\circ} |w(T;x_0,\ldots,x_n)| \dd \lambda_z(x_1)\cdots \dd \lambda_z(x_n),\quad x_0:=q
 \ee
 and  
 \be \label{eq:tree-rooted}
 	T_q^\bullet (z) := z(q) + \sum_{n=1}^\infty \frac{1}{n!}\int_{\mathbb X^n}\sum_{(T,r) \in\mathcal T_n^\bullet} |w(T;x_1,\ldots,x_n)| \prod_{i=1}^n z(x_i) \dd \lambda(x_1)\cdots \dd\delta_q(x_r)\cdots \dd\lambda(x_n).
 \ee
 Notice 
 \be \label{eq:kptrees}
 	T_q^\bullet (z) = z(q) T_q^\circ (z) = z(q)\exp\Bigl( \int_\mathbb X |f(q,y)| T_y^\bullet(z)\dd \lambda(y) \Bigr)
 \ee
 see Faris~\cite[Section 3.1]{faris2010combinatorics}. In Eq.~\eqref{eq:kptrees}, both sides are either infinite or finite. If they are finite, the fixed point problem may have more than one solution, and $T_q^\bullet (z)$ is the smallest solution \cite{faris2010combinatorics}. 
 
 \begin{prop} \label{prop:kptrees}
 	Let $z:\mathbb X\to \R_+$ be a measurable map. 
 	The following two conditions are equivalent: 
 	\begin{enumerate} 
 		\item [(i)] $T_q^\circ(z) < \infty$ for all $q \in \mathbb X$.
 		\item [(ii)] There exists a measurable function $a:\mathbb X\to \R_+$ such that $\int_{\mathbb X} |f(q,y)| \e^{a(y)} \dd\lambda_z(y) \leq a(q)$ for all $q\in \mathbb X$. 
 	\end{enumerate} 
 \end{prop} 
 
 \begin{proof} 
 	The implication (ii) $\Rightarrow$ (i) is proven by induction as in~\cite{poghosyan-ueltschi2009}; see also~\cite[Theorem 3.1]{faris2010combinatorics}. The implication (i) $\Rightarrow$ (ii) follows from the fixed point equation~\eqref{eq:kptrees}, which proves that $a(q):= \log T_q^\circ(z)$ is an admissible choice. 
 \end{proof}

Let $(T,r)\in \mathcal T_n^\bullet$. The root $r$ induces a notion of generations and descendants.
For $k\in T$, let $C_k^{(r)} \subset[n]$ be the collection of children of $k$. 
Set

\bes
	\tilde w(T,r;x_1,\ldots,x_n) = \prod_{k \in T}\Bigl( \prod_{i\in C_k^{(r)}} |f(x_k,x_i)| \prod_{\substack{i,j\in  C_k^{(r)}:\\ i < j}} (1+f(x_i,x_j))\Bigr).
\ees
Define $\tilde T_q^\bullet$ as in~\eqref{eq:tree-rooted} but with weights $\tilde w(T,r;x_1,\ldots,x_n)$ instead of $|w(T;x_1,\ldots,x_n)|$, and $\tilde T_q^\circ$ as in~\eqref{eq:tree-ghost} but with weights $\tilde w(T,0;q,x_1,\ldots,x_n)$.
Then 
\be \label{eq:fptrees}
	\tilde T_q^\bullet(z)  =  z (q) \tilde T_q^\circ (z) =   z(q) \Bigl( 1+\sum_{k=1}^\infty \frac{1}{k!}  \int_{\mathbb X^k} \prod_{1\leq i < j \leq k } (1+ f(x_i,x_j)) \prod_{i=1}^k |f(q,x_i)|\tilde T_{x_i}^\bullet (z) \dd \lambda^k(\vect x) \Bigr). 
\ee

 \begin{prop} \label{prop:fptrees}
 	Let $z:\mathbb X\to \R_+$ measurable. 
 	The following two conditions are equivalent: 
 	\begin{enumerate} 
 		\item [(i)] $\tilde T_q^\circ(z) < \infty$ for all $q \in \mathbb X$.
 		\item [(ii)] There exists a measurable function $a:\mathbb X\to \R_+$ such that condition~\eqref{eq:suff1} holds true for all $n\in \N$ and all $x_0\in \mathbb X$ with $t=0$. 
 	\end{enumerate} 
 \end{prop} 
 
\begin{proof} 
The implication (i)$\Rightarrow$(ii) is again an immediate consequence of the fixed point equation~\eqref{eq:fptrees} (set $a(q) = \log \tilde T_{q}^\circ(z)$). The implication (ii)$\Rightarrow$(i) can be proven by an induction similar to~\cite{ueltschi2004cluster}, or by adapting the arguments from~\cite{fernandez-procacci2007}. For the reader's convenience, we provide an alternative proof, based on Kirkwood-Salsburg equations for forests, in Section~\ref{sec:trees}.
\end{proof} 

\noindent The key bound~\eqref{eq:keyconv} is recovered from Proposition~\ref{prop:fptrees} with the help of the tree-graph inequality
\bes
	\bigl|\varphi_n^\mathsf T(x_0,\ldots,x_n) \bigr| \leq \sum_{T\in \mathcal T_n^\circ} \tilde w(T,0;x_1,\ldots,x_n)
\ees
which can be proven as in~\cite{fernandez-procacci2007}. 

\subsubsection{Extinction probabilities in branching processes}
It is instructive to compare the tree generating functions with extinction probabilities of  branching processes. For simplicity we treat the function $T_q^\circ$ only. Assume that 
\be \label{eq:bfinite}
	b(q) = \int_{\mathbb X} |f(q,y)| \dd \lambda_z(y) < \infty 
\ee
for all $x\in \mathbb X$. Let $\xi^q$ be a Poisson point process with intensity measure $|f(q;y)| \dd\lambda_z(y)$, defined on some probability space $(\Omega,\mathcal F,\mathbb P)$. It has Laplace functional 
\bes
	\mathbb E\Bigl[\e^{-\int_\mathbb X h \dd \xi^q}\Bigr] = \exp\Bigl( \int_{\mathbb X} (\e^{- h(x)}- 1) |f(q,x)| \dd \lambda_z(x)\Bigr),
\ees
the condition~\eqref{eq:bfinite} ensures that the expected number of points $b(q)$ of $\xi^q$ is finite hence in particular, $\xi^q$ has only finitely many points, $\mathbb P$-almost surely.

We can define a discrete-time, multi-type branching process  $(\eta^q_n)_{n\in \N}$ with type space $\mathbb X$ and ancestor $q$ as a Markov chain with state space $\mathcal N$  as follows: 
 $\eta_0^q:=\delta_q$, $\eta_1^q:= \xi_q$, and $\eta^q_{n+1}$ is obtained from $\eta^q_n$ by asking, roughly, that each point $x\in \eta^q_n$, has an offspring equal in distribution to $\xi^x$, with the usual independence assumptions. Precisely, $(\eta^q_n)_{n\in \N_0}$ is  a Markov chain with 
\bes
	\mathbb E\Bigl[ \e^{- \int_\mathbb X h \dd \eta^q_{n+1}}\Big|\, \eta^q_n\Bigr]
		= \exp\Biggl( \int_\mathbb X \Bigl(\int_{\mathbb X}  |f(x,y)|(\e^{- h(y)}- 1) \dd \lambda_z(y)\Bigr) \dd \eta_{n}^q(x)\Biggr) \quad \mathbb P\text{-a.s.}
\ees
for all $n\in \N_0$.  For $B\in \mathcal X_\mathsf b$, the quantity 
$\eta_n^q(B)$ represents the number of points in generation $n$ with type in $B$. 
The dependence on $z$ is suppressed from the notation of the branching process. 
The extinction probability 
\bes
	p(q) := \mathbb P\Bigl( \exists n \in \N:\ \eta_n^ q = 0 \Bigr)
\ees
as a function of $q$, is the smallest non-negative solution 
 of the fixed point problem
\be \label{eq:extinction-fixedpoint}
	p(q) =  \exp\Bigl( \int_{\mathbb X}(p(x)-1) |f(q,x)| \dd \lambda_z(x)\Bigr) \qquad (q\in \mathbb X).
\ee
I.e., if $\tilde p(\cdot)$ is another non-negative solution, then $p(q) \leq \tilde p(q)$ for all $q\in \mathbb X$.

For finite type spaces, this statement is proven in~\cite[Theorem 4.2.2]{jagers-branchingbook}, the proof for infinite $\mathbb X$ is similar. 

The similarity of the fixed point equations~\eqref{eq:extinction-fixedpoint} and~\eqref{eq:kptrees} for extinction probabilities and trees suggests a relation between convergence and extinction. The next proposition says that convergence of tree generating functions implies extinction of the branching process. 

\begin{prop} \label{prop:extinction} 
	Assume $\int_\mathbb X |f(q,x)| \dd \lambda_z(x) < \infty$ for all $q\in \mathbb X$ and let $p(q)$ be the extinction probability of the $z$-dependent branching process with ancestor $\delta_q$. If 
$T_q^\circ(z)< \infty$ for all $q\in \mathbb X$, then $p(q) =1$ for all $q\in \mathbb X$.
\end{prop} 

\noindent The proof adapts a classical result~\cite[Chapter III, Theorem 12.1]{harris-branchingbook}: if the principal eigenvalue of the so-called expectation operator is smaller or equal to $1$ and some additional regularity conditions hold true, then the branching process goes extinct. 

\begin{proof}[Proof of Proposition~\ref{prop:extinction}]
	Consider the kernel on $\mathbb X$ given by $M(q, \dd x):= |f(q,x)|\lambda_z(\dd x)$ and use the same letter for the associated integral operator 
	$$
		(M g)(q):= \int_{\mathbb X} |f(q,x)| g(x) \dd \lambda_z(x) \qquad \bigl(q\in \mathbb X,\, g\in \mathscr L^ \infty(\mathbb X, \mathcal X)\bigr). 
	$$
	$M(q,B)$ represents the expected number of children with type in $B$ of a type-$q$ individual. The \emph{expectation operator} $M$ replaces the mean number of offspring in a single-type branching process. Condition~\eqref{eq:bfinite} guarantees that the expected value $M(q,\mathbb X)$ of the total number of children of a type-$q$ individual is finite for all $q\in \mathbb X$. 
	However $\sup_{q\in \mathbb X} M(q,\mathbb X) =\sup_{q\in \mathbb X}b(q)$ can be infinite, which is why Theorem~12.1 in~\cite[Chapter III.12]{harris-branchingbook} is not applicable; moreover the image $M$ of a bounded function $g$ is not necessarily bounded.  Therefore it is preferable to work with weighted supremum norms. By the convergence of $T_q^\circ (z)$ and Proposition~\ref{prop:kptrees}, there is a function $a:\mathbb X\to \R_+$ such that $\int_{\mathbb X} |f(q,x)|\exp( a(x))\dd \lambda_z(x) \leq a(q)$ for all $q$. We define 
	$$
		||g||_a:= \sup_{q\in \mathbb X} |g(q)|  \e^{-a(q)} 
	$$
	and note that $M$ is a contraction with respect to this norm. Indeed, if $g$ is a function with $||g||_a<\infty$, then 
	\begin{align*} 
		\bigl|(M g)(q) \bigr|& \leq \int_{\mathbb X} |f(q,x)| \e^{a(x)} || g||_{a} \dd \lambda_z(x) 
			\leq a(q) ||g||_a 
	\end{align*} 
	and 
	\bes
		||Mg||_a \leq ||g||_a  \sup_{q\in \mathbb X} a(q)\, \e^{- a(q)}  \leq  ||g||_a \sup_{\alpha \geq 0}\alpha\, \e^ {-\alpha} =\frac{1}{\mathrm e} ||g||_a.  
	\ees
	The fixed point equation for the extinction probability and the inequality $\exp( s) \geq 1+s$ yield 
	\bes
		p(q) \geq 1 + \int_{\mathbb X} |f(q,x)| (p(x) - 1) \dd \lambda_z(x)  = 1+ \bigl(M(p-\mathbf 1)\bigr)(q)
	\ees
	with $\mathbf 1$ the constant function with value $1$. 
	Consequently $\mathbf 1 - p \leq M (\mathbf 1- p)$ pointwise on $\mathbb X$ and 
	\bes
		|| \mathbf 1 - p||_a \leq || M (\mathbf 1 - p)||_a \leq \frac{1}{\mathrm e} || \mathbf 1 - p||_a.
	\ees
	It follows that $||\mathbf 1  - p||_a =0$ and $p(q) =1$ for all $q\in \mathbb X$. 
\end{proof} 

\noindent Below we provide an example for which the branching process goes extinct but the tree generating functions diverge. First we have a closer look at the relation between trees and extinction probabilities. Since every individual has only finitely many children by~\eqref{eq:bfinite}, the extinction probability $p(q)$ is equal to the probability that the total offspring $N_\infty^q:= \sum_{n=0}^ \infty \eta_n^q(\mathbb X)$ is finite (to simplify formulas below, we include the ancestor in the offspring count). The probability that the total offspring consists of exactly $n$ individuals is in turn a sum over trees on $n$ vertices, rooted at the ancestor.  Leaves correspond to individuals without offspring. 
 Hence
\begin{align} \label{eq:pqalt}
	p(q) & = \sum_{n\in \N_0} \P( N_\infty^q = n) \\
			& =	\e^{- z(q) b(q)}\Bigl( 1+ \sum_{n=1}^\infty \frac{1}{n!}\int_{\mathbb X^n} \sum_{T\in \mathcal T_n^\circ} |w(T;q,x_1,\ldots,x_n)|  \prod_{i=1}^n \e^{- z(x_i) b(x_i)}  \dd \lambda_z^n(\vect x) \Bigr) \notag  \\
		& = \e^{- z(q) b(q)} T_q^\circ( z\e^{-b z}). 
\end{align}
More generally, let $\eta^q:=\sum_{n=0}^\infty \eta^q_n$ be the total progeny of $\delta_q$, including the ancestor itself. Then for all $h:\mathbb X\to \R$, we have 
\bes
	 \e^{h-bz} T_q^\circ( z \e^{h-bz}) = \mathbb E\Bigl[  \e^{\int_\mathbb X h \dd \eta^q} \1_{\{N_\mathbb X(\eta^q) < \infty\} } \Bigr].
\ees
and 
\bes
	 T_q^\circ( z) = \mathbb E\Bigl[  \e^{\int_\mathbb X b z \dd \eta^q} \1_{\{N_\mathbb X(\eta^q) < \infty\} } \Bigr].
\ees
This equality together with Proposition~\ref{prop:extinction} shows that $T_q^\circ (z)$ is finite for all $q\in \mathbb X$ if and only if the branching process goes extinct for all possible ancestor choices and in addition the exponential moment above of the total progeny is finite. 

\begin{example}[Galton-Watson process and hard spheres] 	
	Suppose that $z(q) = z$ and $b(q) = b$ are independent of $q$. This is the case, for example, for a gas of hard spheres of fixed radius $R$ in $\R^d$, and then $b = |B(0,2R)|$, compare Example~\ref{ex:hard-spheres}. Then the generation counts  $N_n:= N_\mathbb X(\eta^q_n)$, $n\in \N_0$, form a Galton-Watson process with offspring distribution $\mathrm{Poi}(z b)$. The distribution of the total progeny (including ancestor) $N= \sum_{n=0}^\infty N_n$ satisfies 
	\bes
		\mathbb P\bigl(N=n\bigr) = \frac{(b n)^{n-1}}{n!} z^{n-1} \e^{-n bz} \qquad (n\in \N),
	\ees
	compare the \emph{Borel distribution}~\cite{borel1942}. 
	The tree generating function becomes 
	\bes
		T^\circ (z) = \sum_{n=1}^\infty\frac{(bn)^{n-1}}{n!} z^{n-1} = \sum_{n\in \N} \e^{nb z} \P(N=n). 
	\ees 
	We have $T^\circ(z)< \infty$ if and only if $b z\leq 1/\mathrm{e}$, a condition stronger than subcriticality $b z \leq 1$. 
\end{example} 

We leave open whether the analogue of Proposition~\ref{prop:extinction} holds true  for the  Fern{\'a}ndez-Procacci trees from Proposition~\ref{prop:fptrees}, but sketch how some first steps might be implemented. Define 
\bes
	B_n(q;x_1,\ldots,x_n):= \prod_{i=1}^n |f(q;x_i)| \prod_{1\leq i <j \leq n} \bigl(1+ f(x_i,x_j)\bigr)
\ees
and assume that
\bes
G_q(z):= 1 + 	\sum_{n=1}^ \infty \frac{1}{n!} \int_{\Lambda^n} B_n(q;x_1,\ldots,x_n) \dd \lambda_z^n(\vect x) < \infty 
\ees
for all $q\in \mathbb X$. We may then consider a branching process whose branching mechanism is such that 
\bes
	\mathbb E\Bigl[ \e^ {- \int_{\mathbb X} h \dd \eta_{n+1}} \mid \eta_n = q\Bigr] = \frac{G_q(z\e^ {-h})}{G_q(z)}. 
\ees
The expectation operator is 
\bes
	 M g(q) = \frac{1}{G(z)} \sum_{n=1}^\infty \frac{1}{n!} \int_{\Lambda^{n}} B_n(q;x_1,\ldots,x_n) \Bigl(\sum_{i=1}^n g(x_i)\Bigr) \dd\lambda_z^n(\vect x). 
\ees
The fixed point problem for extinction probabilities becomes $p(q) = G_q(z p) /G_q(z)$. The inequality $\prod_{i=1}^n p(x_i)\geq 1 + \prod_{i=1}^n ( p(x_i)-1)$ yields 
$
	p (q)  \geq 1+ \bigl( M (p - \mathbf 1)\bigr)(q)
$
thus again, we have the pointwise inequality $\mathbf 1 - p \leq M(\mathbf 1 - p)$.  However it is not clear how to turn the convergence condition $G_q(z\e^a) \leq\e^{a(q)}$ into a contraction estimate for the operator $M$.

\subsection{Random connection model and disagreement percolation}  \label{sec:boolean}

\emph{Disagreement percolation} provides an expansion-free route to proofs of uniqueness and exponential mixing for Gibbs measures, see~\cite[Chapter 7]{georgii-haggstrom-maes2001} for Gibbs measures on lattices and~\cite{hofertemmel2015,hofertemmel-houdebert2017} as well as~\cite[Section 2.7]{dereudre2017gibbsintro} for models in $\R^d$, and~\cite{benes-hofer-last-vecera2019} for Gibbs particle processes.  It is an interesting question which approach---percolation or convergence of cluster expansions---trumps the other as far as proving uniqueness of the infinite-volume Gibbs measure goes. In full generality, the question remains open. The purpose of this section is, nevertheless, to  provide some very partial answers that might help attack the question in a more systematic way in the future. 

First we recall some known facts on the relation between percolation and uniqueness of Gibbs measures. If the interaction $v(x,y)$ has finite range $R$, i.e., $v(x,y) =0$ when $\mathrm{dist}(x,y)\geq R$ and there is absence of percolation in the Boolean percolation model with deterministic connectivity radius $R$, then the Gibbs measure is unique. This applies, in particular, to a gas of hard spheres with fixed radius $R$. A variant for Boolean percolation models with random connectivity radius~\cite{hofertemmel-houdebert2017} allows for an extension to mixtures of hard spheres. Uniqueness of Gibbs measures and absence of percolation are linked for finite-range interactions that take negative values as well, as long as some additional conditions are satisfied~\cite{dereudre2017gibbsintro}. Thus  absence of Boolean percolation implies uniqueness of Gibbs measures. 

For a gas of hard spheres, condition~\eqref{eq:suff2} with $t=0$ implies extinction of a branching process by Proposition~\ref{prop:extinction}, which in turn implies absence of percolation~\cite[Chapters~3 and~6]{meester-roy1996book}.  As a consequence, for a gas of hard spheres, condition~\eqref{eq:suff2} is more restrictive than asking for absence of Boolean percolation. We do not know, however, if convergence of cluster expansions or condition~\eqref{eq:suff1} also imply absence of Boolean percolation. 

For a non-negative pair potential that has finite range and is everywhere finite,  in constrast, if $\beta$ is sufficiently small, conditions ~\eqref{eq:suff1} and~\eqref{eq:suff2} can be  \emph{less} restrictive than asking for absence of Boolean percolation (we thank C. Hofer-Temmel and the anonymous referee for this remark). 
Indeed, asking for absence of Boolean percolation imposes a condition $z\leq z^*$ (or $z<z^*$) for some finite, $\beta$-independent $z^*$. Indeed, consider for simplicity $\mathbb X= \R^d$ and a pair potential that is translationally invariant, then condition~\eqref{eq:suff2} reads $z< (\e\,  C(\beta))^{-1}$ with $C(\beta) = \int_{\R^d} (1 - \exp( - \beta v(0,x)))\dd x$. 
Clearly $\lim_{\beta \to 0} C(\beta) =0$ hence for small $\beta >0$, the condition $z < (\e\, C(\beta))^{-1}$ is \emph{less} restrictive than the condition $z<z^*$. 

We conclude this section with an open question. The only information about the interaction kept by the associated Boolean percolation model is the range of the interaction. The natural question arises whether there is a more refined model whose connectivity might relate to properties of the Gibbs measures. For non-negative interactions, a natural candidate is the \emph{random connection model} with connectivity probability given by 
$$
	\varphi(x,y):= |f(x,y)|\qquad (x,y\in \mathbb X). 
$$
See~\cite{meester-roy1996book,last-ziesche2017} for definitions in the translationally case in $\R^d$ and~\cite{last-nestmann-schulte2018} for the general setup. Notice that, by the non-negativity of the pair potential, we have $|f(x,y)| = |\e^{-v(x,y)} - 1|\leq 1$ and may indeed interpret $\varphi(x,y)$ as a probability. The extinction of the branching process from Proposition~\ref{prop:extinction}, together with arguments from~\cite{meester-roy1996book, meester-penrose-sarkar1997}, implies absence of percolation in the random connection model, see Appendix~\ref{app:RCM} for a more precise statement. 
As a consequence, condition~\eqref{eq:suff2} for non-negative interactions implies absence of percolation in the random connection model, much in the same way as condition~\eqref{eq:suff2} for hard spheres implies absence of percolation in a Boolean model. To the best of our knowledge, however, there is no theorem in the literature that shows that absence of percolation in the random connection model implies uniqueness of the Gibbs measure---we do not know whether there is a result in the spirit of disagreement percolation that would use the random connection model rather than Boolean percolation. 

\subsection{Cumulants of  double stochastic integrals} \label{sec:cumulants} 

Here we explain how cluster expansions relate to another kind of diagrammatic expansion, namely, expansions of cumulants of multiple stochastic integrals with respect to compensated Poisson random measures, see~\cite[Chapter~7]{peccati-taqqu2011book}, \cite[Chapter 12.2]{last-penrose2017book}, 
and~\cite{last-penrose-schulte-thale2014moments}.
To that aim we provide two expressions for the cumulants of random variables of the form $\int_{\mathbb X^2} u \dd \eta^{(2)}$ where $\eta$ is a Poisson point process with intensity measure $\lambda_z$ and $u:\mathbb X\times \mathbb X\to \R$ is a function that satisfies some integrability conditions.  We may view such random variables as double integrals with respect to a Poisson random measure $\eta$ (not compensated). The first formula involves edge-labelled graphs with multiple edges (Proposition~\ref{prop:stochmoment}), the second formula  involves precisely the pairs of partitions  (Proposition~\ref{prop:cumulants-diagrams}) that appear in the literature on multiple stochastic integrals~\cite{peccati-taqqu2011book}. 

To the best of our knowledge, these formulas are new, however our principal interest lies in the reasoning that allows us to go from the connected graphs of  cluster expansions to cumulants and pairs of partitions, via multigraphs. Roughly, the $n$-th coefficient of the cluster expansion is a sum over graphs on $n$ vertices while the $m$-th moment of a double stochastic integral is a sum over multigraphs with $m$ edges. 

We start from the following observation.  Let $u:\mathbb X\times \mathbb X\to \R$ and $\beta \in \R$. Suppose that  $v= - \beta u$ satisfies the assumptions of Theorem~\ref{thm:finite-volume}. Then the latter theorem provides an expansion of the cumulant generating function of $\frac12 \beta \int_{\mathbb X^2} u\dd \eta^{(2)}$, with $\eta$ a Poisson point process of intensity $\lambda_z$. The expansion is not in powers of $\beta$ but rather, if $z(x)\equiv z$ is independent of $x$, in powers of $z$. It is a sum over connected graphs, each graph comes with a product of edge weights $\exp( \beta u(x_i,x_j))-1$. In order to obtain the cumulants, we need to understand the expansion in powers of $\beta$ rather than $z$. Now, every edge weight is expanded as 
\bes
	\e^{\beta u(x_i,x_j)}-1 = \sum_{m_{ij}=1}^ \infty \frac{\beta^ {m_{ij}}}{m_{ij}!}\, u(x_i,x_j)^ {m_{ij}}.
\ees
The right-hand side is best interpreted as a sum over edge multiplicities, and thus graphs with multiple edges (but no self-edges $i- i$) naturally appear. It is convenient to label not only the vertices but also the edges. 

\begin{definition}
	Let $A$ and $V$ be two non-empty sets and $\mathcal E_2(V) = \{e\subset V\mid \#e =2\}$. 	\begin{itemize} 
		\item A \emph{labelled multigraph} $\gamma$ with vertex labels $V$ and edge labels $A$ is a map from $A$ to $\mathcal E_2(V)$. The set of such multigraphs is denoted $\mathcal M(V,A)$. 
		\item The \emph{multiplicity} of an edge $\{i,j\} \in \mathcal E_2(V)$ in $\gamma\in \mathcal M(V,A)$ is $m_{ij}(\gamma) = \# \{a \in A\mid \gamma(a) = \{i,j\} \}$. 
		\item A multigraph $\gamma$ is \emph{spanning} if every vertex $i\in V$ belongs to some edge in $\gamma$, i.e., $m_{ij}(\gamma)\geq 1$ for some $j\in V$. The set of spanning multigraphs is denoted $\mathcal M_s(V,A)$.
		\item A multigraph $\gamma$ is \emph{connected} if, for all $i,j\in V$, there exist $k\in \N$ and a sequence $(a_1,\ldots,a_k)\in A^k$ such that $i\in \gamma(a_1)$, $j\in \gamma(a_k)$, and $\gamma(a_r) \cap \gamma(a_{r+1})\neq \varnothing$ for all $r\in \{1,\ldots,k-1\}$.  The set of connected multigraphs is denoted $\mathcal M_c(V,A)$. 
	\end{itemize}
\end{definition} 

\begin{prop} \label{prop:stochmoment}
	 Let $\eta$ be a Poisson point process with intensity measure $\lambda_z$ and $u:\mathbb X\times \mathbb X\to \R$ a symmetric function. Suppose that 
	 \be \label{eq:smsuff}
	 	\int_{\mathbb X^n} \prod_{1 \leq i < j \leq n} \bigl| u(x_i,x_j)\bigr|^{m_{ij}(\gamma)} \dd \lambda^n(\vect x) < \infty
	 \ee
	 for all $n\in \{2,\ldots,2m\}$ and $\gamma\in \mathcal M_s([n],[m])$.  Then $\mathbb E[(\int_{\mathbb X} |u|\dd \eta^{(2)})^m]< \infty$ and the $m$-th moment and the $m$-th cumulant of $\int_{\mathbb X^2 } u \dd \eta^{(2)}$  are given by sums over spanning and connected multigraphs as
	\begin{align*}
		\mathbb E\Bigl[ \Bigl( \frac12 \int_{\mathbb X^2} u\dd \eta^{(2)} \Bigr)^m\Bigr] & = \sum_{n=2}^{2m} \frac{1}{n!}\int_{\mathbb X^n} \sum_{\gamma \in \mathcal M_s([n],[m])}\prod_{1 \leq i < j \leq n} u(x_i,x_j)^{m_{ij}(\gamma)} \dd \lambda_z^n(\vect x) \\
		\kappa_m \Bigl( \frac12 \int_{\mathbb X^2} u\dd \eta^{(2)} \Bigr) & = \sum_{n=2}^{2m} \frac{1}{n!}\int_{\mathbb X^n} \sum_{\gamma \in \mathcal M_c([n],[m])}\prod_{1 \leq i < j \leq n} u(x_i,x_j)^{m_{ij}(\gamma)} 
	\dd \lambda^n_z(\vect x),
	\end{align*}
	with absolutely convergent integrals.
\end{prop} 

\noindent The proposition is proven in Section~\ref{sec:stochmoment}. The proof does not use the previous convergence theorems, which is why it works under different convergence conditions (compare Eq.~\eqref{eq:smsuff} and Eq.~\eqref{eq:suff1} with $v=-\beta u$), much in the same way as moments of a random variable may be finite even though exponential moments are infinite.  Nevertheless it is instructive to derive the formula for the cumulants from  Theorem~\ref{thm:finite-volume} on the expansion of the pressure in finite volume, which we now do.  Suppose that $v= -\beta u$ satisfies the conditions of Theorem~\ref{thm:finite-volume}, and that $\lambda_z(\mathbb X)< \infty$ so that we may set $\Lambda = \mathbb X$.   Then we have 
\bes
	\log \mathbb E\Bigl[ \e^{\frac12 \beta \int_{\mathbb X^2} u \dd \eta^{(2)}}\Bigr] 
	 = \sum_{n=2}^\infty \frac{1}{n!} \int_{\mathbb X} \Bigl(  \sum_{G\in \mathcal C_n}\prod_{\{i,j\}\in E(G)}\bigl( \e^{\beta u(x_i,x_j)}-1\bigr) \Bigr) \dd \lambda_z^n(\vect x). 
\ees
Expanding the exponential, we find that for each $n$, the integrand is a sum over pairs $(G, \vect m)$ that consist of a connected graph $G$ and a vector of integers $\vect m = (m_{ij})_{i<j}$ with $m_{ij}\geq 1$ if and only if $\{i,j\} \in E(G)$. The pair $(G,\vect m)$ is best thought of as a graph with multiple edges (edges are non-labelled). The vector of multiplicities determines the graph uniquely, let $C_n$ be the set of multiplicity assignments for which the associated graph is connected. We get
\begin{align*}
	 \sum_{G\in \mathcal C_n}\prod_{\{i,j\}\in E(G)}\bigl( \e^{\beta u(x_i,x_j)}-1\bigr) 
	 & = \sum_{G\in \mathcal C_n}\prod_{\{i,j\}\in E(G)} \Bigl( \sum_{m_{ij}=1}^\infty \frac{1}{m_{ij}!} \bigl( \beta u(x_i,x_j)\bigr)^{m_{ij}} \Bigr) \\
	 & =\sum_{m=1}^\infty \frac{\beta^m}{m!} \sum_{\substack{ (m_{ij})_{1 \leq i <j \leq n}\\ m_{ij}\in \N_0 \\ \sum_{i<j} m_{ij} =m}  }\   \frac{m!}{\prod_{1\leq i <j \leq n} m_{ij}!}  \prod_{1 \leq i < j \leq n} u(x_i,x_j)^ {m_{ij}} \1_{C_n}(\vect m). 
\end{align*} 
The multinomial coefficient is equal to the number of multigraphs with edge labels $\{1,\ldots,m\}$  and prescribed multiplicities $m_{ij}$. Thus we find 
\be \label{eq:doublexp} 
		\log \mathbb E\Bigl[ \e^{\frac12 \beta \int_{\mathbb X^2} u \dd \eta^{(2)}}\Bigr] 
 = \sum_{m=1}^ \infty \frac{\beta^ m}{m!} 	 \sum_{n=2}^\infty\frac{1}{n!} 
 	\int_{\mathbb X^ n} \sum_{\gamma\in \mathcal M_c([n],[m])}\prod_{1 \leq i < j \leq n} u(x_i,x_j)^{m_{ij}(\gamma)}  \dd \lambda_z^n(\vect x).
\ee
Since the cumulants are defined by the relation
\bes
	\log \mathbb E\Bigl[ \e^{\frac12 \beta \int_{\mathbb X^2} u \dd \eta^{(2)}}\Bigr]  = \sum_{m=1}^ \infty \frac{\beta^m}{m!}\, \kappa_m\Bigl( \frac12 \int_{\mathbb X^2} u\dd \eta^{(2)} \Bigr) 
\ees
we may read them off from Eq.~\eqref{eq:doublexp} and obtain the expression from Proposition~\ref{prop:stochmoment}. \\

Next we explain how to go from multigraphs to partition pairs so as to obtain expressions closer to~\cite{peccati-taqqu2011book}. Let $\gamma \in \mathcal M_s(V,A)$ be a spanning multigraph on $V$ with edge labels $A$. We define an associated pair $(\pi,\sigma)$ of partitions as follows: \begin{itemize} 
	\item  First we define a new set $S$ of ``dedoubled'' vertices: every vertex $v\in V$ gives rise to as many points in $S$ as there are edges to which it belongs. The new vertices $s$ are labelled by the parent vertex $v\in V$ and the edge label $a$. Precisely, we set $S:= \{(v,a) \in V\times A  \mid  v\in \gamma(a)\}$. 
	\item The partition $\pi$ has the blocks $B_a = \{(v,a)\mid v\in \gamma(a)\}$, $a \in A$: it groups dedoubled vertices $s$ connected by an edge. Every block $B_a$ has cardinality exactly $2$. 
	\item  The partition $\sigma$ lumps together new vertices $s$ that come from a common underlying vertex $v\in V$.  Put differently, the blocks of $\sigma$ are the sets
	$T_v = \{(a,v)\mid v\in \gamma(a)\}$, $v\in V$. Note that the $T_v$'s are non-empty because the graph is spanning.
\end{itemize}

Notice that for all $a\in A$, $v\in V$, the set $B_a\cap T_v$ is a singleton if the vertex $v$ belongs to the edge with label $a$, and empty if it does not; in particular $B_a \cap T_v$ is either empty or a singleton. If the multigraph $\gamma$ is connected, then so is $(\pi,\sigma)$, in the sense given below.

\begin{definition}[\cite{peccati-taqqu2011book}]\label{def:nonflat}
	Let $S$ be a finite non-empty set and $(\pi,\sigma)\in \mathcal P(S)\times \mathcal P(S)$ a pair of set partitions of $S$. 
	\begin{itemize} 
		\item The pair is \emph{non-flat} if for all blocks $B\in \pi$ and $T\in \sigma$, the intersection $B\cap T$ is either empty or a singleton.
		\item The pair is \emph{connected} if for every strict subset $M\subsetneq S$, there is a block $B$ of $\pi$ or $\sigma$ such that $B$ intersects both $M$ and $S\setminus M$. 
	\end{itemize} 
\end{definition} 

\noindent Equivalently, a pair $(\pi,\sigma)$ is non-flat and connected if $\pi \wedge \sigma$ is the partition into singletons and $\pi \vee \sigma$ is the partition consisting of a single block $S$, where $\wedge$ and $\vee$ refer to the join and meet in the lattice of set partitions. 
	
\begin{remark} [Gaussian diagrams] 
	The pair of partitions $(\pi,\sigma)$ is not only non-flat and connected, but also has the property that every block of $\pi$ has cardinality exactly $2$, because we only allow for graphs with edges $\{i,j\}$ (hypergraphs associated with multi-body interactions would include hyperedges consisting of $3$ vertices or more). Peccati and Taqqu~\cite{peccati-taqqu2011book} call pairs where instead $\sigma$ has only blocks of cardinality $2$ \emph{Gaussian}, and associate graphs with  such pairs as well; their graphs allow for self-edges and restrict the degree of the vertices to $2$,  a type of graphs clearly different from ours. 
\end{remark}

 Let $S$ be a finite set and $\hat \sigma =(T_1,\ldots,T_n)$ an ordered set partition of $S$ with $n$ blocks; let $\sigma = \{T_1,\ldots,T_n\} \in \mathcal P(S)$ be the underlying set partition.
 For $s\in S$, let $i(s)\in \{1,\ldots,n\}$ be the label of the block to which $s$ belongs, i.e., $s\in T_{i(s)}$. 
  The ordered partition $\hat \sigma$ induces an embedding of $\mathbb X^n$ into $\mathbb X^S$ defined by $(x_1,\ldots,x_n)\mapsto (x_{i(s)})_{s\in S}$.  With any given map $g:\mathbb X^n\to \R$ we associate a new map $g_{\hat \sigma}: \mathbb X^S\to \R$ by $g_{\hat \sigma} (x_1,\ldots,x_n):= g((x_{i(s)})_{s\in S})$. For example, if $ S= \{1,2,3\}$ and $\hat \sigma = ( \{1,3\}, \{2\})$, then $g_{\hat \sigma}( x_1,x_2) = g(x_1,x_2,x_1)$. Changing the order of the blocks in $\hat \sigma$ permutes the variables in $g_{\hat \sigma}$ but leaves the integral $\int_{\mathbb X^n} g_{\hat \sigma} \dd \lambda_z^ n$ unchanged, by a slight abuse of notation we write $\int_{\mathbb X^n} g_\sigma \dd \lambda_z^n$ instead. 
 
In the following proposition the $m$-fold tensor $u\otimes\cdots \otimes u$ is the function $(x_1,\ldots,x_{2m}) \mapsto u(x_1,x_2) u(x_3,x_4)\cdots u(x_{2m-1},x_{2m})$, and $||\sigma||$ is the number of blocks of a partition $\sigma$. 
 
\begin{prop}\label{prop:cumulants-diagrams}
	Let $S_m=\{1,\ldots,2m\}$, and $\pi_m =\{ \{1,2\},\{3,4\},\ldots,\{2m-1,2m\}\}$.
	Under the assumptions of Proposition~\ref{prop:stochmoment}, the $m$-th cumulant of $\int_{\mathbb X^ 2} u \dd \eta^{(2)}$ is given by 
	\bes
		\kappa_m\Bigl( \int_{\mathbb X^2} u \dd \eta^ {(2)}\Bigr)  = \sum_{\substack{\sigma\in \mathcal P(S_m):\\ (\pi_m,\sigma)\, \text{non-flat, connected}}} \int_{\mathbb X^{||\sigma||}} (u\otimes \cdots \otimes u)_\sigma(x_1,\ldots,x_{||\sigma||})\dd \lambda_z^n(\vect x)
	\ees
	with absolutely convergent integrals. 
\end{prop} 

The proposition is proven in Section~\ref{sec:stochmoment}. It is  deduced from Proposition~\ref{prop:stochmoment} and the correspondence between connected multigraphs and non-flat connected pairs of partitions. Some combinatorial subtleties arise because the correspondence is not exactly one-to-one; this is also the reason why Proposition~\ref{prop:stochmoment} looks at $\frac12 \int_{\mathbb X^2} u \dd \eta^{(2)}$ while Proposition~\ref{prop:cumulants-diagrams} deals with $\int_{\mathbb X^2} u \dd \eta^{(2)}$.  

Proposition~\ref{prop:cumulants-diagrams} should be contrasted with a similar expression for the cumulants of the random variable 
\bes
		I_2(u) = \int_{\mathbb X^2} u \dd \eta^{(2)} -  \int_{\mathbb X^ 2} u\, \dd ( \eta\otimes \lambda_z ) -   \int_{\mathbb X^2} u\, \dd (\lambda_z \otimes \eta) + \int_{\mathbb X^2} u \dd (\lambda_z\otimes \lambda_z).
	\ees
	The cumulants of $I_2(u)$ are given by sums 
 over non-flat, connected diagrams $(\pi_m,\sigma)$ such that every block of $\sigma$ has cardinality at least $2$, see~\cite{peccati-taqqu2011book,last-penrose-schulte-thale2014moments}. This corresponds to connected multigraphs for which every vertex $i\in V$ has degree $\sum_{j\in V} m_{ij}(\gamma) \geq 2$. 

We leave open whether the known formulas for the cumulants of $I_2(u)$ have a simple explanation in terms of cluster expansions. Regardless of the answer, the considerations in this section show that seemingly different types of expansions can be put on a common footing, which might yield interesting insights in the future. 

\section{Kirkwood-Salsburg equation. Uniqueness} \label{sec:ks}

 Here we prove Theorem~\ref{thm:uniqueness} on the uniqueness of Gibbs measures. Our proof follows a standard strategy~\cite[Chapter 4.2]{ruelle1969book}: We start from a set of integral equations satisfied by the correlation functions, reformulate these equations as a fixed point problem in some Banach space, and show that the fixed point problem involves a contraction. A minor novelty of our proof lies in our choice of norms as weighted supremum norms with weights that incorporate some information on interactions.

\begin{lemma} \label{lem:ks}
	Assume $v\geq 0$, $\lambda^2$-almost everywhere, and $\int_{\mathbb X} |f(x,y)| \dd \lambda_z(x)<\infty$ for $\lambda_z$-almost all $x\in \mathbb X$. 
	Let $\mathsf P\in \mathscr G(z)$. The correlation functions $\rho_n$ of $\mathsf P$ satisfy the Kirkwood-Salsburg equations
	\begin{multline} \label{eq:ks} \tag{KS}
		\rho_{n+1}(x_0,\ldots,x_n) = z(x_0) \prod_{i=1}^n (1+ f(x_0,x_i)) \\
			\times \left( \rho_n(x_1,\ldots,x_n) + \sum_{k=1}^\infty \frac{1}{k!} \int_{\mathbb X^k} \prod_{j=1}^k f(x_0,y_i)\rho_{n+k}(x_1,\ldots,x_n,y_1,\ldots,y_k) \dd \lambda^k(\vect y)\right) 
	\end{multline} 
	for all $n\in \N_0$ and $\lambda^{n+1}$-almost all $(x_0,\ldots,x_{n})\in \mathbb X^{n+1}$, with the convention $\rho_0 =1$ and $\prod_{i=1}^0 (1+ f(x_0,x_i)) =1$. 
\end{lemma} 

\noindent The absolute convergence of the right-hand side of~\eqref{eq:ks} is ensured by the condition~\eqref{eq:suff1} and the bound~\eqref{eq:nonnegbound}.  
The lemma is a consequence of the well-known equivalence of the GNZ and the Dobrushin-Lanford-Ruelle (DLR) conditions~\cite{nguyen-zessin79} on the one hand, and the DLR conditions and the Kirkwood-Salsburg equation on the other hand~\cite{ruelle1970superstable}, see also~\cite{kuna1999dissertation,kuna-kondratiev2003}. For the reader's convenience we present a short self-contained proof of the implication~\eqref{eq:gnz}$\Rightarrow$\eqref{eq:ks} in Appendix~\ref{app:knownstuff}.

The Kirkwood-Salsburg equations can be rephrased as a fixed point equation in some suitable Banach space. 	Let $E_0$ be the  space of sequences $\vect \rho = (\rho_n)_{n\in\N}$ of real-valued measurable functions $\rho_n:\mathbb X^n\to \R$  such that 
\bes
	|\rho_n(x_1,\ldots,x_n)| \leq C_{\vect\rho} \e^{tn}\prod_{1\leq i <j\leq n} (1+ f(x_i,x_j)) \prod_{i=1}^n z(x_i) \e^{a(x_i)} 
\ees
for some $C_{\vect \rho} \geq 0$, all $n\in \N$, and $\lambda^n$-almost all $(x_1,\ldots,x_n)\in \mathbb X^n$.
Let $||\vect \rho||_0$ be the smallest constant $C_{\vect \rho}$. The quotient $E$ of $E_0$ with the null space $\{ \vect \rho: \, ||\vect \rho||_0=0\}$, together with the norm $|| [\vect \rho] ||:= ||\vect \rho||_0$, is a Banach space; by a slight abuse of notation we write $\vect \rho$ instead of $[\vect \rho]$. 
 	For $\vect \rho \in E$, set
	\be \label{eq:kdef1}
		(K_z \vect \rho)_1(x_0) = z(x_0) \sum_{k=1}^\infty \frac{1}{k!} \int_{\mathbb X^k} \prod_{j=1}^k f(x_0,y_j)\rho_{k}(y_1,\ldots,y_k) \dd \lambda^k(\vect y)
	\ee
	and if $n\in \N$
	\begin{multline}\label{eq:kdef2}
		(K_z \vect \rho)_{n+1}(x_0,\ldots,x_n)
			= z(x_0) \prod_{i=1}^n (1+ f(x_0,x_i)) 
			\times \Bigl( \rho_n(x_1,\ldots,x_n)  \\ + \sum_{k=1}^\infty  \frac{1}{k!} \int_{\mathbb X^k} \prod_{j=1}^k  f(x_0,y_j)\rho_{n+k}(x_1,\ldots,x_n,y_1,\ldots,y_k) \dd \lambda^k(\vect y)\Bigr). 
	\end{multline} 
Let $\vect e_z =(e_{z,n})_{n\in\N} \in E$ be the sequence given by $e_{z,1}(x_1)= z(x_1)$ on $	\mathbb X$ and $e_{z,n} =0$ for $n\geq 2$. The Kirkwood-Salsburg equation is rewritten as 
\be
	\vect \rho = \vect e_z +  K_z\vect \rho.
\ee

\begin{lemma} \label{lem:contract}
	Under the assumptions of Theorem~\ref{thm:uniqueness}, the operator $K_z$ is a  bounded linear operator in $E$ with operator norm $||K_z||\leq \e^{-t}< 1$. 
\end{lemma} 

\begin{proof} 
		Suppose that $\vect \rho \in E$. Then we have, for $n\in \N$, 
	\be
		\prod_{i=1}^n (1+ f(x_0,x_i)) |\rho_n(x_1,\ldots,x_n)| 
			\leq ||\vect \rho|| \e^{tn}
 \prod_{0\leq i< j \leq n}(1+ f(x_i,x_j)) \prod_{i=1}^n z(x_i) \e^{ a(x_i)} 
	\ee
	and for $n\in \N$, $k\in \N_0$, 
	\begin{multline}
		\prod_{i=1}^n (1+ f(x_0,x_i))\prod_{i=1}^k |f(x_0,y_i)|\, |\rho_{n+k}(x_1,\ldots,x_n,y_1,\ldots,y_k)|  \\
		\leq ||\vect \rho|| \e^{t(n+k)}  \prod_{0\leq i < j \leq n} (1+ f(x_i,x_j))  \prod_{i=1}^k |f(x_0,y_i)| \prod_{1\leq i < j \leq k}(1+ f(y_i,y_j)) \\
			\times \prod_{i=1}^n z(x_i)\e^{a(x_i)} \prod_{j=1}^kz(y_j) \e^{a(y_j)}.
	\end{multline} 
	Here the mixed terms $1+f(x_i,y_j)$ that arise from bounding $\rho_{n+k}(x_1,\ldots,x_n,y_1,\ldots,y_k)$ have been bounded by $1$. 
	In the Kirkwood-Salsburg equation we use the triangle inequality and plug in the previous estimates; this yields 
	\begin{multline} 
		|(K_z\vect \rho)_{n+1}(x_0,\ldots,x_n)| \leq ||\vect \rho||\, z(x_0)\, \prod_{0\leq i < j \leq n} (1+ f(x_i,x_j)) \e^{tn}\prod_{i=1}^n z(x_i) \e^{a(x_i)} \\
		\times \left( 1 + \sum_{k=1}^\infty \frac{\e^{tk}}{k!} \int_{\mathbb X^k} \prod_{i=1}^k |f(x_0,y_i)| \prod_{1\leq i < j \leq k}(1+ f(y_i,y_j)) \e^{\sum_{j=1}^k a(y_j)} \dd \lambda^k(\vect y)\right) \\
		\leq \e^{-t } ||\vect \rho||\, z(x_0)\,  \prod_{0\leq i < j \leq n} (1+ f(x_i,x_j))\prod_{i=1}^n z(x_i)  \e^{\sum_{i=0}^n a(x_i)}\e^{t (n+1)} .
	\end{multline} 
	For $n=0$, 
	\be
		|(K_z \vect \rho)_1(x_0)|\leq  \e^{-t} ||\vect \rho|| (\e^{a(x_0)}-1)\leq \e^{-t} ||\vect \rho|| \e^{t+a(x_0)}.
	\ee
	It follows that $K_z\vect \rho \in E$ and $||K_z\vect \rho||\leq\e^{-t} ||\vect \rho||$. 
\end{proof} 

\begin{proof}[Proof of Theorem~\ref{thm:uniqueness}]
	The existence of a Gibbs measure follows from Theorem~\ref{thm:existence} in Appendix~\ref{app:existence}. For the uniqueness, we start from Lemma~\ref{lem:correp} and note that the sequence of $n$-point functions of a Gibbs measure $\mathsf P\in \mathscr G(z)$ is in $E$. In view of $||K_z||\leq \e^{-t}<1$, the operator $(\mathrm{id} -   K_z)$ is invertible with bounded inverse given by a Neumann series. Therefore the vector of correlation functions $\vect \rho = (\rho_n)_{n\in \N}$, is uniquely determined by the Kirkwood-Salsburg equations and is given by 
	\be \label{eq:neumann} 
		\vect \rho = (\mathrm{id} - K_z)^{-1} \vect e_z = \vect e_z+ \sum_{\ell=1}^\infty K_z^\ell \vect e_z.
	\ee	
	By Lemma~\ref{lem:correp}, we have $\rho_n(x_1,\ldots,x_n)\leq \prod_{j=1}^n z(x_j)$ and 
	\bes
		 \mathsf E\Bigl[ \eta(B)\bigl( \eta(B) -1\bigr) \cdots \bigl( \eta(B)-n+1\bigr)\Bigr]=  \int_{B^n} \rho_n \dd\lambda^n \leq \bigl( \int_B z\dd\lambda\bigr)^n = \lambda_z(B)^n.
	\ees
	for all $B\in \mathcal X_\mathsf b$ and $n \in \N_0$. As a consequence the correlation functions determine the measure $\mathsf P$ uniquely~\cite[Proposition 4.12]{last-penrose2017book} and we find $\#\mathscr G(z) = 1$.
\end{proof} 

\section{Weighted graphs. Expansion of correlation functions} \label{sec:correlations} 

In the previous section we have proven that condition~\eqref{eq:suff1} guarantees the uniqueness of the Gibbs measure $\mathsf P$. We proceed with the proof of Theorem~\ref{thm:correlations} on the expansion of correlation functions. As noted earlier, the vector of correlation functions $\vect \rho = (\rho_k)_{k\in \N}$ is the unique solution of a fixed point equation $\vect \rho = \vect e_z + K_z \vect e_z$ and has the series representation~\eqref{eq:neumann}. It remains to compute the powers $K_z^n \vect e_z$, i.e., to show that they are indeed given by integrals and sums involving multirooted graphs. This is done by induction, noting that 
the partial sums of the series are in fact Picard iterates of the fixed point equation with initial value $\vect{e}_z$. Indeed, 
\bes
	\sum_{n =0}^{N+1} K_z^n \vect e_z  = \vect e_z + K_z \Bigl( \sum_{n =0}^{N} K_z^n \vect e_z\Bigr). 
\ees
The combinatorial counterpart to the partial sums of the Neumann series~\eqref{eq:neumann} are sums over graphs truncated at some maximum number $N$ of vertices. 
Remember the functions $\psi_{k,n+k}$ from~\eqref{eq:psidef} and the multirooted graphs $\mathcal D_{k,n}$ with $k$ roots and $n\geq k$ vertices. Define
\begin{multline} \label{eq:sndef}
	\bigl( \vect S_N(z) \bigr)_k(x_1,\ldots,x_k)  := z(x_1)\cdots z(x_k) \\
		\times \sum_{n=k}^N \frac{1}{(n-k)!} \int_{\mathbb X^{n-k}} \psi_{k,n}(x_1,\ldots,x_k,y_{k+1},\ldots,y_n) \dd \lambda_z^{n-k}(\vect y)
\end{multline} 
if $1\leq k \leq N$ and $(\vect S_N(z))_k(x_1,\ldots,x_k):=0$ if $k\geq N+1$. The summand for $k=n$ is understood as $\psi_{k,k}(x_1,\ldots,x_k)$. In order to prove that $\vect S_N(z)$ is equal to the partial sum $\sum_{n=1}^N K_z^n \vect e_z$, we show that $\vect S_1 (z) = \vect e_z$ and that $\vect S_{N+1}(z)$ is obtained from $\vect S_N(z)$ by a Picard iteration, see Proposition~\ref{prop:picard} below. The proof builds on recursive properties of multirooted graph with respect to removal of a root and as such generalizes the recursive proof with singly rooted graphs from~\cite{ueltschi2004cluster}, see also~\cite[Chapter 5.4]{friedli-velenik2017book}. 

In addition, we prove that if the condition~\eqref{eq:suff1} holds true with $t=0$, then the right-hand side of~\eqref{eq:sndef} converges pointwise as $N\to \infty$, and we provide bounds. This ensures convergence of expansions even if $K_z$ has operator norm equal to $1$. In that case the limit corresponds to a fixed point of $\vect \rho = \vect e_z + K_z \vect \rho$ though we no longer know whether the solution is unique. This part of the proof is similar to the inductive treatment 
of the Kirkwood-Salsburg equation in~\cite{bissacot-fernandez-procacci2010}.

Define $\vect{\tilde S}_N$ in a similar way as $\vect S_N$ but with additional absolute values, i.e., 
\begin{multline} \label{eq:sntildedef}
	\bigl( \vect{\tilde S}_N(z) \bigr)_k(x_1,\ldots,x_k)  := z(x_1)\cdots z(x_k) \\
		\times \sum_{n=k}^N \frac{1}{(n-k)!} \int_{\mathbb X^{n-k}} \bigl|\psi_{k,n}(x_1,\ldots,x_k,y_{k+1},\ldots,y_n)\bigr| \dd \lambda_z^{n-k}(\vect y)
\end{multline} 
if $1\leq k \leq N$ and $(\vect{\tilde S}_N(z))_k(x_1,\ldots,x_k):=0$ if $k\geq N+1$. Clearly $\vect S_1 = \vect{\tilde S}_1=\vect e_z$.

\begin{prop}\label{prop:picard}
Assume $v\geq 0$, $\lambda_z(B)< \infty$ for all $B\in \mathcal X_\mathsf b$, and suppose that condition~\eqref{eq:suff1} holds true with $t=0$. Then we have 
\be \label{eq:snpicard}
	\vect S_1(z) = \vect e_z,\quad \vect S_{N}(z) = \vect e_z + K_z \vect S_{N-1}(z) \quad (N\geq 2)
\ee
with 
\begin{multline}\label{eq:picard-bound}
	|(\vect S_N(z))_k(x_1,\ldots,x_k)| \leq (\vect{\tilde S}_N(z))_k(x_1,\ldots,x_k) 
		\leq \prod_{1\leq i < j \leq k} (1+ f(x_i,x_j)) \prod_{j=1}^k z(x_j) \e^{a(x_j)}
\end{multline}
for all $k,N\in \N$ and $\lambda^k$-almost all $(x_1,\ldots,x_k)\in \mathbb X^k$. 
\end{prop} 

\noindent The equality~\eqref{eq:snpicard} for $\vect S_N$ is complemented by an inequality for $\vect {\tilde S}_N$: define the operator $\tilde K_z$ just as $K_z$ but with $f(x_0,y_j)$ replaced with $|f(x_0,y_j)|$.
Then $\vect{\tilde S}_1(z) = \vect e_z$ and 
	\be \label{eq:tildepicard}
		\vect{\tilde S}_1(z) = \vect e_z,\quad \vect {\tilde S}_N(z) \leq \vect e_z + \tilde K_z \vect {\tilde S}_{N-1}(z)\quad (N\geq 2)
	\ee
where ``$\vect g\leq \vect h$'' refers to pointwise inequality $g_k\leq h_k$ of the components. The inequality~\eqref{eq:tildepicard} allows for an inductive proof of~\eqref{eq:picard-bound}. 

The proof of Proposition~\ref{prop:picard} builds on two lemmas which hold true for attractive interactions as well. It is convenient to extend the definitions $\mathcal D_{k,n}$ and $\psi_{k,n}$. Given two   finite non-empty sets $I,J$ with $I\subset J$, let $\mathcal D(I,J)$ as the set of connected graphs $G$ with vertex set $J$ such that every vertex $j\in J$ connects in $G$ to at least one vertex $i\in I$. Set
\be
	\psi\bigl(I,J;(x_j)_{j\in J}\bigr):= \sum_{G\in \mathcal D(I,J)} w\bigl(G;(x_j)_{j\in J}\bigr) \qquad \bigl((x_j)_{j\in J}\in \mathbb X^J \bigr).
\ee
Thus $\mathcal D_{k,n} = \mathcal D([k],[n])$ and $\psi_{k,n}(\cdot) = \psi([k],[n];\cdot)$. For $J$ a finite possibly empty set, define
\be
	\psi\bigl(\varnothing,J;(x_j)_{j\in J}\bigr) := 
		\begin{cases} 
			1, &\quad J=\varnothing,\\
			0, &\quad J \neq \varnothing. 
		\end{cases}
\ee
Notice that for all permutations $\sigma$ of $J$, 
\be \label{eq:symmetry} 
	\psi\bigl(I,J;(x_{\sigma(j)})_{j\in J}\bigr) = \psi\bigl(\sigma(I),J;(x_j)_{j\in J}\bigr).
\ee
In particular, $\psi (I, J;\cdot)$ is invariant with respect to permutations that act on $J\setminus I$ only. 

 The following set of equations is similar to~\cite[Lemma 6.2]{poghosyan-ueltschi2009}. In fact Lemma~\ref{lem:recursion} and~\cite[Lemma 6.2]{poghosyan-ueltschi2009} show that our functions $\psi(I,J;\cdot)$ are equal to the functions $g(I,J\setminus I)$ defined with Ruelle's algebraic approach in~\cite{poghosyan-ueltschi2009}. We may think of $\mathcal D(I,J)$ as a collection of multi-rooted graphs with root set $I$ and some connectivity constraints, then the lemma reflects a recursive combinatorial structure when removing a root. 

\begin{lemma} \label{lem:recursion}
	Let $I,J$ be two non-empty finite sets with $I\subset J\subset \N_0$. Fix $(x_j)_{j\in J}\in \mathbb X^J$. Let $\iota(I) = \min I$, $I' = I\setminus \{\iota (I)\}$, and $J'= J\setminus \{\iota (I)\}$. We have 
		$$
			\psi\bigl(I,J;(x_j)_{j\in J}\bigr) =\Biggl(\prod_{i\in I'} \bigl(1+f(x_{\iota(I)}, x_i)\bigr)\Biggr) \sum_{L\subset J\setminus I} \Bigl(\prod_{\ell\in L} f(x_ {\iota(I),\ell})\Bigr) \psi\bigl(I'\cup L, J';(x_j)_{j\in J'}\bigr).
		$$
\end{lemma} 

\noindent In the sum $L=\varnothing$ is allowed. Products over empty sets are equal to $1$. Lemma~\ref{lem:recursion} is also true for attractive pairwise interactions. 

\begin{proof} 
	For $G\in \mathcal D(I,J)$, let $L\subset J\setminus I$ the set of vertices $j\in J\setminus I$ that are adjacent to $\iota (I)$, i.e., 
	$$
		L = \{j\in J\setminus I\mid \{\iota (I), j\} \in E(G)\}. 
	$$
	Let $G'$ be the graph obtained from $G$ by removing the vertex $\iota (I)$ and all incident edges.  Assume first that $\#I \geq 1$ so that $I'\neq \varnothing$. Then every vertex of $G'$ must be connected in $G'$ to at least one vertex in $I'\cup L$, thus $G'\in \mathcal D(I'\cup L, J')$. The weight of $G$ is given by 
	$$
		w\bigl(G; (x_j)_{j\in J}\bigr) = \Biggl(\prod_{\substack{i\in I':\\ \{\iota(I),i\}\in E(G)}} f(x_{\iota(I)},x_i)\Biggr) \Biggl( \prod_{\ell \in L} f(x_{\iota (I)},x_\ell)) \Biggr) w\bigl(G'; (x_j)_{j\in J}\bigr) . 
	$$
	Conversely, given a possibly empty subset $L\subset J\setminus I$ and a graph $G'\in \mathcal D(I'\cup L, J')$, we construct a graph $G\in \mathcal D(I,J)$ by adding the vertex $\iota (I)$, adding all edges $\{\iota(I),\ell\}$, $\ell \in L$, and adding none, some, or all of the edges $\{\iota(I),i\}$, $i\in I'$. 
	The proof for the case $\#I\geq 1$ is concluded by summing first over all graphs $G$ that correspond to a given pair $(L,G')$, and then summing over the pairs $(L,G')$. 
	
	If $\#I =1$ and $\# J\geq 2$, then for a given graph $G$ the element $\iota(I)$ has at least one adjacent vertex, i.e., $L\neq \varnothing$, and the proof is completed as before, taking into account that $\psi(\varnothing, J;\cdot) =0$ if $J'\neq \varnothing$. If $\#I = \#J=1$, then $\psi(I,J;(x_j)_{j\in J}) = \psi(\{1\},\{1\};x_1) =1$ and the formula from the lemma holds true because $\psi(\varnothing,\varnothing;\cdot) =1$. 
\end{proof}

\noindent The partial sums satisfy the following recursive inequality. 

\begin{lemma} \label{lem:recursion2}
	For all $N\geq 2$ and  $k\in \N$, and all $(x_1,\ldots, x_k)\in \mathbb X^k$,
			\begin{multline*}
	\bigl(\vect{\tilde S}_{N}(z)\bigr)_k(x_1,\ldots,x_k) \leq z(x_1) \prod_{j=2}^k (1+ f(x_1,x_j)) \\ 
		\times \sum_{m=0}^{N-k} \frac{1}{m!} \int_{\mathbb X^m} \prod_{\ell=k+1}^{k+m} \bigl|f(x_1,x_\ell)\bigr| \bigl(\vect{\tilde S}_{N-1}(z)\bigr)_{k+m-1}(x_2,\ldots,x_{k+m}) \dd\lambda_z(x_{k+1})\cdots \dd \lambda_z(x_{k+m})
\end{multline*} 
	where the term for $k=1$, $m=0$ is to be read as $z(x_1)$. 
\end{lemma} 

\noindent Thus~\eqref{eq:tildepicard} holds true. 

\begin{proof} 
	Let $N\geq 2$. 
	For $k\geq N+1$ we have $\bigl(\vect{\tilde S}_{N}(z)\bigr)_k(x_1,\ldots,x_k) =0$ and the inequality is trivial. Consider $k\in \{2,\ldots, N\}$. For  $n\in \{k,\ldots,N\}$ and $(x_1,\ldots,x_n)\in \mathbb X^n$. By  Lemma~\ref{lem:recursion} applied to $I = \{1,\ldots, k\}$ and $J= \{1,\ldots, n\}$, 
	\be \label{eq:revrecursion}
		\psi_{k,n}(x_1,\ldots, x_n) = \Biggl(\prod_{i=2}^k \bigl(1+f(x_{1}, x_i)\bigr)\Biggr) \sum_{L\subset \{k+1,\ldots,n\}} \Bigl(\prod_{\ell\in L} f(x_ {1},x_\ell)\Bigr) \psi\bigl(I'\cup L,J';(x_j)_{j\in J'}\bigr)
	\ee
	with $I'= \{2,\ldots,k\}$ and $J' =  \{2,\ldots,n\}$. Let $\sigma\in \mathfrak S_n$ be a permutation that acts on $\{k+1,\ldots,n\}$ only, i.e., $\sigma (i) = i$ for all $i\in \{1,\ldots,n\}$. By~\eqref{eq:symmetry}, we have for every $L\subset\{k+1,\ldots,n\}$, 
	\be \label{eq:revsymmetry}
		\Bigl(\prod_{\ell\in L} f(x_ {1},x_{\sigma(\ell)}) \Bigr)\psi\bigl(I'\cup L,J';(x_{\sigma(j)})_{j\in J'}\bigr) = \Bigl(\prod_{\ell\in \sigma(L)} f(x_ {1},x_\ell) \Bigr) \psi\bigl(I'\cup \sigma(L),J';(x_j)_{j\in J'}\bigr).
	\ee
	We apply the triangle inequality in~\eqref{eq:revrecursion}, integrate over $x_{k+1},\ldots, x_n$, note that the contribution of each set $L$ to the integral depends on the cardinality of $L$ alone, and obtain
	\begin{multline*}
		\int_{\mathbb X^{n-k}} \bigl|\psi_{k,n}(x_1,\ldots,x_n)\bigr| \dd \lambda_z(x_{k+1})\cdots \dd \lambda_z(x_n) \leq \Biggl(\prod_{i=2}^k \bigl(1+f(x_{1}, x_i)\bigr)\Biggr) \\
		\times \sum_{m=0}^{n-k} \binom{n-k}{m} \int_{\mathbb X^{n-k}}  \Biggl(\prod_{\ell=k+1}^{k+m-1}\bigl| f(x_1,x_\ell)\bigr|\Biggr) \bigl|\psi_{k+m-1,n-1}(x_2,\ldots,x_{n})\bigr| \dd \lambda(x_{k+1})\cdots \dd \lambda(x_n).
	\end{multline*} 
	We divide by $(n-k)!$, sum over $n=k,\ldots, N$,  and find that
	\be \label{eq:revisionmiddle}
	\begin{aligned} 
		&\sum_{n=k}^N \frac{1}{(n-k)!} \int_{\mathbb X^{n-k}} \bigl|\psi_{k,n}(x_1,\ldots,x_n)\bigr| \dd \lambda_z(x_{k+1})\cdots \dd \lambda_z(x_n) \\
		&\qquad \leq \Biggl( \prod_{i=1}^k \bigl(1+ f(x_1,x_i)\bigr) \Biggr) 
 \sum_{m=0}^{N-k} \Biggl( \prod_{\ell=k+1}^{k+m} \bigl| f(x_1,x_\ell)\bigr|\Biggr) \Bigl\{\cdots \Bigr\}  \dd \lambda_z(x_{k+1})\cdots \dd \lambda_z(x_{k+m}))
	\end{aligned} 
	\ee
	with 
	\begin{align*}
		\Bigl\{\cdots \Bigr\}  &= \sum_{n=k+m}^N \frac{1}{(n-k-m)!}
			  \int_{\mathbb X^{n-k}} \bigl|\psi_{k+m-1,n-1}(x_2,\ldots,x_{n})\bigr| \dd \lambda_z(x_{k+m+1})\cdots \dd \lambda_z(x_n).
	\end{align*} 
	Changing the summation index from $n$ to $n'= n-1$, we see that the expression is exactly $\bigl(\vect{\tilde S}_{N-1}(z)\bigr)_{k+\ell-1}(x_2,\ldots,x_k,y_{k+1},\ldots,y_{k+m})$, except for the missing product $z(x_2)\cdots z(x_k)$. Hence, multiplying~\eqref{eq:revisionmiddle} with $z(x_1)\cdots z(x_k)$ on both sides, we obtain the required inequality  when $k\geq 2$. The proof for $k=1$ is similar, with a careful consideration of the term $m=0$ which gives rise to $z(x_1)$. 
\end{proof} 

\begin{proof} [Proof of Proposition~\ref{prop:picard}]
	First we prove, by induction over $N$, that the inequality~\eqref{eq:picard-bound} holds true for all $N,k\in \N$. For $N=k=1$, we have $(\tilde{\vect S_1})_1(x_1) = z(x_1) \leq z(x_1)\exp( a(x_1))$ and the required inequality holds true.  For $N=1$ and $k\geq 2$, we have $(\tilde{\vect S_1})_k(x_1,\ldots, x_k)=0$ and the inequality holds true as well. 
	
	For the induction step, let $N\geq 2$ suppose that the bound ~\eqref{eq:picard-bound} holds true for all $k\in \N$ at $N-1$ instead of $N$. We insert the bound into right-hand side of the recursive inequality from Lemma~\ref{lem:recursion2}, then bound 
	\be \label{thisiswhere}
		\prod_{\substack{1\leq i \leq k\\ k+1\leq j \leq m}} (1+ f(x_i,x_j))\leq 1,
	\ee
	and finally use condition~\eqref{eq:suff1} with $t=0$ (compare the proof of Lemma~\ref{lem:contract}).  This yields the bound~\eqref{eq:picard-bound} for $N$ and completes the induction step. 
	
	Eq.~\eqref{eq:picard-bound} shows, in particular, that the integrals defining $\vect S_N$ are absolutely convergent. We may now revisit the induction step, but without absolute values and triangle inequalities. The inequalities then become equalities for $\vect S_N$: we find 
\be
	\bigl(\vect S_{N}(z)\bigr)_k = \bigl( K_z\vect S_{N-1}\bigr)_k,\quad   
	\bigl(\vect S_{N}(z)\bigr)_1(x_1)  = z(x_1) + \bigl( K_z \vect S_{N-1}(z)\bigr)_1(x_1)
\ee
and the proof of the proposition is complete. 
\end{proof} 

\noindent Theorem~\ref{thm:correlations} is a consequence of the Neumann series~\eqref{eq:neumann} and Proposition~\ref{prop:picard}. 

\begin{proof}[Proof of Theorem~\ref{thm:correlations}]
	Passing to the limit $N\to \infty$ in the inequality~\eqref{eq:picard-bound} from Proposition~\ref{prop:picard}, we obtain the estimate~\eqref{eq:correlations-bound} from Theorem~\ref{thm:correlations}. The equality~\eqref{eq:snpicard} shows 
	\be
		\vect S_N(z) = \vect e_z+ K_z \vect e_z+\cdots + K_z^{N-1} \vect e_z.
	\ee
	Thus $\vect S_N(z)$ is a partial sum of the Neumann series~\eqref{eq:neumann}. The bound~\eqref{eq:picard-bound} ensures that each $(\vect S_N)_k(x_1,\ldots,x_k)$ converges pointwise as $N\to \infty$. But we already know from the proof of Theorem~\ref{thm:correlations} that the Neumann series converges in the Banach space $E$ to the vector of correlation functions; thus $\rho_k(x_1,\ldots,x_k) = \lim_{N\to \infty} (\vect S_N)_k(x_1,\ldots,x_k)$ and we obtain the representation of the correlation functions. The bound~\eqref{eq:correlations-bound} with $t=0$ follows from Proposition~\ref{prop:picard}. For $t>0$, we note that $z\e^t$ satisfies~\eqref{eq:suff1} with $t=0$ so we can apply the inductive bound of Proposition~\ref{prop:picard} to $z\e^t$, and the proof is easily concluded. 
\end{proof}

\section{Log-Laplace functional and truncated correlation functions} \label{sec:truncated} 

Here we prove Theorems~\ref{thm:log-laplace} and~\ref{thm:truncated}. Theorem~\ref{thm:truncated} is deduced from Theorem~\ref{thm:log-laplace} by exploiting that the log-Laplace functional at $h$ is nothing else but the generating functional of the truncated correlation functions (factorial cumulant densities) at $u = \e^{-h} -1$, see Eq.~\eqref{eq:truncated-generating} below. Explicit bounds are proven with the complex contour integrals (here $t>0$ is crucial). 

 For the proof of Theorem~\ref{thm:log-laplace}, we first specialize Theorem~\ref{thm:correlations}, proven in the previous section, to the one-particle density ($k=1$). 
As noted earlier, the classes of graphs $\mathcal D_{1,n}$ and $\mathcal C_n$  are equal 
 (if every vertex $j\in \{2,\ldots,n\}$ connects to the vertex $1$, then the graph is connected, and vice-versa). It follows that $\psi_{1,n} = \varphi_n^\mathsf T$ and 
\be \label{eq:one-point-expansion}
	\rho_1(x_1)  =z(x_1) \sum_{n=1}^\infty \frac{1}{(n-1)!}\int_{\mathbb X^{n-1}} \varphi_n^\mathsf T(x_1,\ldots,x_n) \dd \lambda_z(x_2)\cdots \dd\lambda_z(x_n)
\ee
with absolutely convergent integrals and series, moreover the bound~\eqref{eq:keyconv} stated in Theorem~\ref{thm:log-laplace} is just the special case of the inequality~\eqref{eq:correlations-bound} in Theorem~\ref{thm:correlations}. 

For the proof of the identity~\eqref{eq:log-laplace}, the idea is to first prove a differentiated version of it. Formally, 
$$
	\left.\frac{\dd}{\dd t} \log \mathsf E\Bigl[\e^{- \int_\mathbb X(h+ t g) \dd \eta}\Bigr] \right|_{t=0} = \frac{\mathsf E[ (\int_\mathbb X g \dd \eta) \exp( - \int_{\mathbb X} h \dd \eta)]}{\mathsf E[ \exp( - \int_{\mathbb X} h \dd \eta)] } = \int_{\mathbb X} g(x_1) \rho_1^h(x_1) \dd\lambda(x_1),
$$
with $\rho_1^h(x)$ the one-particle density of a tilted measure $\mathsf P_h$, which we write succinctly with notation from variational derivatives as 
\be \label{eq:left}
	\frac{\delta}{\delta h(x_1)} \log \mathsf E\Bigl[\e^{- \int_\mathbb X  h \dd \eta}\Bigr]  = \rho_1^h(x_1).
\ee
The variational derivative of the right-hand of Eq.~\eqref{eq:log-laplace} is 
\be \label{eq:right}
	z(x_1) \e^{-h (x_1)} \sum_{n=1}^\infty \frac{1}{(n-1)!}\int_{\mathbb X^{n-1}} \varphi_n^\mathsf T(x_1,\ldots,x_n) \e^{- \sum_{i=2}^n h(x_i)} \dd \lambda_z(x_2)\cdots \dd\lambda_z(x_n).
\ee
This is nothing else but the right-hand side of~\eqref{eq:one-point-expansion} with $z$ replaced by $z \e^{-h}$. If the tilted measure $\mathsf P_h$ is a Gibbs measure at tilted activity $z\e^{-h}$, we can conclude that the expressions~\eqref{eq:left} and~\eqref{eq:right} are equal, i.e., the differentiated form of Eq.~\eqref{eq:log-laplace} holds true and it remains to undo the differentiation. 

The full proof is a little technical as we need to make sure that all expressions involved are convergent and that we can exchange differentiation and integration. We start with the proof that the tilted measure $\mathsf P_h$ is indeed a Gibbs measure with tilted activity $z\e^{-h}$. 

\begin{lemma} \label{lem:hgibbs}
	Let $\mathsf P\in \mathscr G(z)$ and $h:\mathbb X\to \R_+$ with $\int_\mathbb X h \dd \lambda_z< \infty$. Consider the measure $\mathsf P_h$ that is absolutely continuous with respect to $\mathsf P$, with Radon-Nikod{\'y}m derivative 
	\bes
		\frac{\dd \mathsf P_h}{\dd \mathsf P}(\eta) = \frac{\exp( - \int_\mathbb X h \dd \eta)}{\int_\mathcal N \exp(- \int_\mathbb X h\dd \gamma)\dd \mathsf P(\gamma)}. 
	\ees
	Then $\mathsf P_h \in \mathscr G(z\e^{-h})$. 
\end{lemma} 

\begin{proof}
	From $h\geq 0$, Jensen's inequality, and the bound $\rho_1(x) \leq z(x)$ we get
	\bes
		0< \e^{- \int_\mathbb Xh \dd \lambda_z} \leq  \mathsf E\bigl[ \e^{- \int_\mathbb X h \dd  \eta}\bigr] \leq 1
	\ees
	so the normalization constant is strictly positive and finite. Let $F:\mathbb X\times \mathcal N\to \R_+$ be a  measurable map. Notice 
	$\int_\mathbb X h \dd (\eta + \delta_x) = \int_\mathbb X h \dd \eta+ h(x)$. 
	Then by~\eqref{eq:gnz}, we have 
	\be
		\mathsf E\Bigl[ \int_{\mathbb X} F(x,\eta) \e^{-\int_\mathbb X h\dd \eta}\dd \eta(x) \Bigr]
			= \int_{\mathbb X}\mathsf E[ F(x,\eta+\delta_x) \e^{-\int_\mathbb X h \dd\eta - W(x;\eta)}\Bigr] z(x) \e^{-h(x)} \dd \lambda(x). 
	\ee
	We divide on both sides by $\int_\mathcal N \e^{- \int_\mathbb X h\dd \gamma}\dd \mathsf P(\gamma) $ and find that $\mathsf P_h$ satisfies~\eqref{eq:gnz} with $z(x)$ replaced by $z(x)\exp(-h(x))$, hence $\mathsf P_h \in \mathscr G(z\e^{-h})$. 
\end{proof} 

\noindent Remember the function $a: \mathbb X\to \mathbb R_+$ from the convergence  criterion~\eqref{eq:suff1}.

\begin{lemma}  \label{lem:log-laplace-nonneg}
	Let $h:\mathbb X\to \R_+$ be a bounded measurable function. Suppose that $h$ is bounded and supported in some set of the form $\Lambda\cap \{\vect x \in \mathbb X\mid a(x) \leq M\}=:\Lambda_M$ with $\lambda_z(\Lambda)<\infty$ and $M\in (0,\infty)$. Then Eq.~\eqref{eq:log-laplace} holds true. 
\end{lemma} 

\begin{proof}
	Under the assumptions of the lemma, we have 
	\be \label{eq:hfirst}
		\int_\mathbb X h (x)\e^{a(x)}\dd \lambda_z(x)< \infty. 
	\ee
	The modified activity  $z\e^{-h}$ satisfies the bound~\eqref{eq:suff1} as well, with the same $a$ and $t$ as $z$, so we have a  representation for the one-point function $\rho_1^h$ of $\mathsf P_h$ analogous to~\eqref{eq:one-point-expansion}. Let us introduce functions $F,G:\R_+ \to \R$ by 
	\be 
		F(s) = \log \mathsf E\bigl[\e^{-s \int_\mathbb X h\dd \eta}\bigr],\quad 
		G(s) = \sum_{n=1}^\infty \frac{1}{n!}\int_{\mathbb X^n }(\e^{-s \sum_{i=1}^n h(x_i)} -1 )\varphi_n^\mathsf T(x_1,\ldots,x_n) \dd \lambda_z^n(\vect x).
	\ee
	Clearly $F(0) = G(0)=0$, we want to prove $F(1)=G(1)$.  Assuming that  differentiation, integration and summations can be exchanged, we get 
	\be \label{eq:fprime}
		F'(s) = - \int_{\mathcal N} \Bigl(\int_\mathbb X h\dd \eta\Bigr) \dd \mathsf P_{sh}(\eta) = \int_\mathbb X h(x) \rho_1^{sh}(x) \dd \lambda(x) 
	\ee
	and, exploiting the symmetry of the Ursell function $\varphi_N^\mathsf T$,  
	\begin{align} \label{eq:gprime}
		G'(s) &=- \sum_{n=1}^\infty \frac{1}{n!}\int_{\mathbb X^n }\Bigl(\sum_{i=1}^n h(x_i)\Bigr)\e^{-s \sum_{i=1}^n h(x_i)}\varphi_n^\mathsf T(x_1,\ldots,x_n) \dd \lambda_z^n(\vect x) \notag\\
			& = - \int_\mathbb X h(x_1) z(x_1) \e^{- s h(x_1)}\notag \\
			&\qquad \times \left(\sum_{n=1}^\infty \frac{1}{(n-1)!} \int_{\mathbb X^{n-1}} \varphi_n^\mathsf T(x_1,\ldots,x_n) \dd \lambda_{z\e^{-sh}}(x_2)\cdots \dd \lambda_{z\e^{-sh}}(x_n)\right) \dd \lambda (x_1).
	\end{align} 
	Using the analogue of ~\eqref{eq:one-point-expansion} for $\rho_1^{sh}$, we see that $F'(s) = G'(s)$ for all $s\geq 0$ and it follows that $F(s) =G(s)$ for all $s\geq 0$. In particular, $F(1) = G(1)$ and Eq.~\eqref{eq:log-laplace} holds true for bounded $h$. 
	It remains to justify~\eqref{eq:fprime} and~\eqref{eq:gprime}. For $s>\eps> 0$ we have 
	\be \label{eq:ebound}
			\frac{1}{\eps}\bigl|\e^{-(s\pm \eps)\alpha} - \e^{-s\alpha}\bigr|\leq \alpha \e^{- (s-\eps)\alpha}\leq \alpha, 
	\ee
	and for $s=0$ we have $\eps^{-1}|\e^{- \eps \alpha} - 1|\leq \alpha$. 
		For the difference quotients of $s\mapsto \mathsf E[\exp(- s \int_\mathbb X h \dd \eta)]$, we apply the inequality to $\alpha = \int_\mathbb X h\dd \eta$, note 
	\be
		\mathsf E\Bigl[\int_\mathbb X h \dd \eta\Bigr] = \int_\mathbb X h \rho_1\dd \lambda \leq \int_\mathbb X h \dd\lambda_z < \infty
	\ee
	and conclude with dominated convergence that differentiation and expectation can be exchanged. 
	For the difference quotients of $G(s)$, we apply the bound~\eqref{eq:ebound} to $\alpha = \sum_{i=1}^n h(x_i)$ and note that, in view of~\eqref{eq:keyconv} and the inequality~\eqref{eq:trick}, we have 
	\be
		\sum_{n=1}^\infty \frac{1}{n!}\int_{\mathbb X^n} \sum_{j=1}^n h(x_j)|\varphi_n^\mathsf T(x_1,\ldots,x_n)| \dd \lambda_z^n(\vect x) \leq \int_{\mathbb X} h(x) \e^{a(x)}\dd \lambda_z(x) < \infty.
	\ee
	and we conclude with dominated convergence that~\eqref{eq:gprime} holds true. It follows that Eq.~\eqref{eq:log-laplace} holds true if $h$ satisfies~\eqref{eq:hfirst}. 		
\end{proof} 

\noindent 
Thus we have proven that 
\be \label{eq:sp}
		\mathsf E\Bigl[\e^{- \int_\mathbb X h \dd \eta} \Bigr] = 
			\exp\Bigl( \sum_{n=1}^\infty \frac{1}{n!}\int_{\mathbb X^n} (\e^{ - \sum_{j=1}^n h (x_j)} - 1)  \varphi_n^\mathsf T(\vect x) \dd \lambda_z^n(\vect x) \Bigr)
\ee
when $h$ is non-negative, bounded, and supported in some set $\Lambda_M = \Lambda\cap \{a\leq M\}$. A straightforward argument involving monotone and dominated convergence shows that the identity extends to all non-negative $h$, but the extension to functions that take negative or complex values requires more work. 

In order to get rid of the condition $h\geq 0$, we express the left- and right-hand sides of Eq.~\eqref{eq:sp} as power series in $\e^{-h}$. Roughly, the idea is that if two power series of some variable $s$ coincide and converge on $s\in [0,1]$, then their coefficients and their domain of convergence must be equal and the identity extends to the whole domain of convergence. 
We start with the right-hand side of Eq.~\eqref{eq:sp}. Let $\Lambda_M\subset \mathbb X$ be such that $\int_{\Lambda_M} \e^a \dd\lambda_z <\infty$. Set 
\bes
	\vartheta_m(x_1,\ldots,x_m) = \sum_{k=0}^\infty \frac{1}{k!} \int_{(\mathbb X\setminus \Lambda_M)^k} \varphi_{m+k}^\mathsf T(x_1,\ldots,x_n,y_1,\ldots,y_k) \dd \lambda_z^k(\vect y). 
\ees
To lighten notation, we suppress the $\Lambda_M$-dependence from $\vartheta_m$. 
By~\eqref{eq:keyconv}, we have 
	\begin{align} \label{eq:thetasti}
		\sum_{m=1}^\infty \frac{\e^{t m}}{m!} \int_{\Lambda_M^m} |\vartheta_m(x_1,\ldots,x_m)| \dd \lambda^m(\vect x) &\leq \sum_{m=1}^\infty \sum_{k=0}^\infty \frac{\e^{t(m+k)}}{m!k!} \int_{\Lambda_M^m\times (\Lambda_M^c)^k} |\varphi_{m+k}^\mathsf T(\vect x)| \dd \lambda_z^{m+k} (\vect x) \notag \\
		& = \sum_{n=1}^\infty \frac{\e^{tn}}{n!}\int_{\mathbb X^n}| \varphi_n^\mathsf T(x_1,\ldots,x_n)| \1_{\{\exists j:\, x_j \in \Lambda_M\}}   \dd \lambda_z^n(\vect x) \notag \\
		& \leq \int_{\Lambda_M} \e^a \dd \lambda_z \leq \e^M \lambda_z(\Lambda)< \infty,
	\end{align} 
so the integrals and series in the definition of $\vartheta_m$ are absolutely convergent for $\lambda^m$-almost all $(x_1,\ldots,x_m)\in \Lambda_M^m$, moreover the constant
	\bes
		C_z(\Lambda_M) := \sum_{m=1}^\infty \frac{1}{m!} \int_{\Lambda_M^m} \vartheta_m (x_1,\ldots,x_m)\dd \lambda_z^m(\vect x)
	\ees
	is finite.
	
	\begin{lemma} \label{lem:right} 
	Let $h:\mathbb X\to [-t,\infty)\cup \{\infty\}$ or $h:\mathbb X\to \{s\in \C:\ \Re\, s\geq - t\}$. Assume that $\int_{\mathbb X} |\e^{-h} - 1|\e^a \dd \lambda_z <\infty$ and that $h$ is supported in some set $\Lambda_M =\Lambda\cap \{a \leq M\}$ with $\lambda_z(\Lambda)< \infty$. 
	Then 
	 \begin{multline}  \label{eq:rightleft}
			\exp\Bigl( \sum_{n=1}^\infty \frac{1}{n!}\int_{\mathbb X^n} (\e^{ - \sum_{j=1}^n h (x_j)} - 1)  \varphi_n^\mathsf T(\vect x) \dd \lambda_z^n(\vect x) \Bigr)\\
	 	= 
		\e^{ - C_z(\Lambda_M)} \Bigl(1 + \sum_{n=1}^\infty \frac{1}{n!} \int_{\Lambda_M^n} \sum_{r=1}^n \sum_{\{V_1,\ldots,V_r\}\in \mathcal P_n} \prod_{k=1}^r \e^{ - \sum_{i\in V_k} h(x_i)} \vartheta_{\#V_k}(\vect x_{V_k}) \dd \lambda_z^n(\vect x)  \Bigr)=:\mathcal R(h)
	 \end{multline} 
	 with absolutely convergent sums and integrals. 
\end{lemma} 

\begin{proof} 
	We have already observed that the bound~\eqref{eq:keyconv} holds true, therefore the convergence of the integrals and the sums on the left-hand side of~\eqref{eq:rightleft} follows from the inequality~\eqref{eq:trick} and the integrability condition $\int_{\mathbb X} |\e^{-h} - 1|\e^a \dd \lambda_z <\infty$. On the left-hand side of~\eqref{eq:rightleft} the only non-zero contributions to integrals come from $\vect x$ with $x_j \in \Lambda_M$ for some $j$. 
	Straightforward computations together with the symmetry of the Ursell functions $\varphi_n^\mathsf T$ show that  the sum inside the exponential can be rewritten as 
	\be\label{eq:ja2}
		 \sum_{n=1}^\infty \frac{1}{n!}\int_{\mathbb X^n} (\e^{ - \sum_{j=1}^n h (x_j)} - 1)  \varphi_n^\mathsf T(\vect x) \dd \lambda_z^n(\vect x) = 	\sum_{m=1}^\infty \frac{1}{m!}\int_{\Lambda_M^m} (\e^{- \sum_{i=1}^m h(x_i)} - 1) \vartheta_m(x_1,\ldots,x_m) \dd \lambda_z^m(\vect x)
	\ee
	The proof of the equation requires some exchange of order of summation, which is justified with~\eqref{eq:trick} and estimates similar to ~\eqref{eq:thetasti}. The right-hand side of Eq.~\eqref{eq:ja2} is actually convergent also without the factor $-1$. Indeed, from $|\exp( - h)|= \exp(- \Re \, h) \leq \exp( - t)$ and~\eqref{eq:thetasti}, we get 
	\bes
		\sum_{m=1}^\infty \frac{1}{m!}\int_{\Lambda_M^m} \bigl| \e^{- \sum_{i=1}^m h(x_i)} \vartheta_m(x_1,\ldots,x_m)\bigr |\,  \dd \lambda_z^m(\vect x) \leq \int_{\Lambda_M}\e^a \dd \lambda_z < \infty. 
	\ees
	Exponentiating~\eqref{eq:ja2}, we find that the left-hand side of Eq.~\eqref{eq:rightleft} is given by 
	\bes
		 \exp \left( - C_z(\Lambda_M) + \sum_{m=1}^\infty \frac{1}{m!}\int_{\Lambda_M^m} \bigl| \e^{- \sum_{i=1}^m h(x_i)} \vartheta_m(x_1,\ldots,x_m)\bigr |\,  \dd \lambda_z^m(\vect x) \leq \int_{\Lambda_M}\e^a \dd \lambda_z\right)
	\ees
	and the proof is concluded with a standard identity on exponential generating functions of set partitions (see Eq.~\eqref{eq:exponentialformula}) below). 
\end{proof} 
	
Next we turn to the left side of Eq.~\eqref{eq:sp}. Let $h$ and $\Lambda_M$ be as in Lemma~\ref{lem:right} and $(j_{n,\Lambda_M})_{n\in \N_0}$ such that \be \label{eq:janossy} 
		\mathsf E\bigl[\e^{- \int_\mathbb X h \dd \eta} \bigr]  = j_{0,\Lambda_M} + \sum_{n=1}^\infty \frac{1}{n!} \int_{\Lambda_M^n} \e^{- \sum_{j=1}^n h(x_j)} j_{n,\Lambda_M}(\vect x) \dd \lambda_z^n (\vect x) =:\mathcal L(h)
\ee 	
whenever $h$ is supported in $\Lambda_M$ and the right-hand side converges (with absolute values in the integrand), for example, if $h\geq 0$. Thus $j_{0,\Lambda_M}$ is an avoidance probability and $j_{n,\Lambda_M}$, $n\geq 1$, are the \emph{Janossy densities}~\cite[Chapter 5.3]{daley-verejones2008vol2}, also called \emph{density distributions}~\cite{ruelle1970superstable},  of $\mathsf P$ in $\Lambda_M$ with respect to the reference measure $\lambda_z$. 
 In our setup the Janossy densities exist and satisfy 
\be \label{eq:janossyinv}
	j_{n,\Lambda_M}(\vect x) \prod_{i=1}^n z(x_i) = \rho_n(x_1,\ldots,x_n) + \sum_{k=1}^\infty \frac{(-1)^k}{k!} \int_{\Lambda_M^k} \rho_n(x_1,\ldots,x_n,y_1,\ldots,y_k) \dd \lambda ^k(\vect y)
\ee
with the convention $\rho_0=1$. The integrals and series are absolutely convergent because of~\eqref{eq:nonnegbound}. Eq.~\eqref{eq:janossyinv} is the well-known inversion formula for Janossy densities and correlation functions, see~\cite[Chapter 5, Eq.~(5.4.11)]{daley-verejones2008vol2} and the first equation after Eq.~(5.33) in~\cite{ruelle1970superstable}.  The odd-looking product $z(x_1)\cdots z(x_n)$ in~\eqref{eq:janossy} appears because the $\rho_n$'s are defined with the reference measure $\lambda$ rather than $\lambda_z$.

\begin {lemma} \label{lem:almostthere}
	Let $\Lambda_M$, $h$,  and $(\vartheta_m)_{m\in \N_0}$ be as in Lemma~\ref{lem:right}. Then the Janossy densities $(j_{n,\Lambda_M})_{n\in \N_0}$ are given by 
	\be \label{eq:jatheta}
		j_{0,\Lambda_M} =  \e^{- C_z(\Lambda_M)},\qquad  j_{n,\Lambda_M}(\vect x) = \e^{- C_z(\Lambda_M)}\sum_{r=1}^n \sum_{\{V_1,\ldots,V_r\}\in \mathcal P_n}\vartheta_{\#V_k}(\vect x_{V_k})
\ee
	for all $n\in \N$ and $\lambda_z^n$-almost all $(x_1,\ldots,x_n)\in \Lambda_M^n$. 
	Moreover Eq.~\eqref{eq:sp} extends to all functions $h$ supported in $\Lambda_M$ that take values in $[-t,\infty) \cup \{\infty\}$ or in $\{s\in \mathbb C\mid \Re\, s \geq - t\}$. 
\end{lemma} 

\begin{proof} 
	By Lemma~\ref{lem:log-laplace-nonneg} and Eq.~\eqref{eq:sp}, the functionals $\mathcal L(h)$ and $\mathcal R(h)$ from Eqs.~\eqref{eq:janossy} and~\ref{lem:right} have to coincide for all nonnegative, 
 bounded $h$ supported in $\Lambda_M$ and subject to~\eqref{eq:hcondition}, which is enough to ensure that ~\eqref{eq:jatheta} holds true. It follows that the identity $\mathcal L(h) = \mathcal R(h)$ extends to the functions $h$ that satisfy the assumptions of Lemma~\ref{lem:right}.  
\end{proof} 

Now we can complete the proof of Theorem~\ref{thm:log-laplace}.

\begin{proof}[Proof of Theorem~\ref{thm:log-laplace}] 
	We have already observed that the bound~\eqref{eq:keyconv} follows from Theorem~\ref{thm:correlations} and the identity $\mathcal D_{1,n} = \mathcal C_n$. 	
	Let $h$ be a measurable function that satisfies the integrability condition~\eqref{eq:hcondition} and take values in either $[-t,\infty)\cup \{\infty\}$ or in $\{s\in \mathbb C:\ \Re\, s\geq t\}$.  
 For $k\in \N$, set  $h_k:=h \1_{B(0,k)}\1_{\{a\leq k \}}$. By Lemma~\ref{lem:almostthere}, we have 
 	\be \label{eq:log-laplace-k}
 		\log \mathsf E\Bigl[ \e^{-\int_\mathbb X h_k \dd \eta}\Bigr] = \sum_{n=1}^\infty \frac{1}{n!}\int_{\mathbb X^n} (\e^{- \sum_{j=1}^n h_k(x_j)} - 1) \varphi_n^\mathsf T(x_1,\ldots,x_n)\dd \lambda^n(\vect x). 
 	\ee
	In order to pass to the limit on the left side, we write 
 	\bes
		\mathsf E\Bigl[ \e^{-\int_\mathbb X h_k \dd \eta}\Bigr]  =  1 + \sum_{n=1}^\infty \frac{1}{n!} \int_{\mathbb X^n} \prod_{i=1}^n (\e^{-h_k(x_i)} - 1) \rho_n(\vect x) \dd \lambda^n(\vect x),
	\ees	
	a 	similar identity holds true with $h_k$ replaced by $h$. Moreover $h_k\to h$ pointwise with
	\bes
		|\e^{-h_k}-1| =|\e^{-h} - 1|\1_{B(0,k)}\1_{\{a\leq k \}}\leq |\e^{-h} - 1|
	\ees	
	and
	\bes	
		 1 + \sum_{n=1}^\infty \frac{1}{n!} \int_{\mathbb X^n} \prod_{i=1}^n |\e^{-h(x_i)} - 1| \rho_n(\vect x) \dd \lambda^n(\vect x) \leq \exp\Bigl( \int_\mathbb X |\e^{- h} - 1|\dd \lambda_z\Bigr)< \infty. 
	\ees
	Dominated convergence thus yields $\mathsf E[\exp( - \int_\mathbb X h_k \dd \eta)]\to \mathsf E[\exp( - \int_\mathbb X h \dd \eta)]$, i.e., we can exchange limits and integration  in the left-hand side of Eq.~\eqref{eq:log-laplace-k}. 	
On the right-hand side of~\eqref{eq:log-laplace-k}, we can can exchange limits and integration too because of the inequality 
	\be
		|\e^{- \sum_{i=1}^n h_k(x_i)} - 1|\leq  \e^{(n-1)t} \sum_{i=1}^n |\e^{- h_k(x_i)} - 1|\leq \e^{(n-1)t} \sum_{i=1}^n |\e^{- h(x_i)} - 1|
	\ee
	(remember Eq.~\eqref{eq:trick} !), 
	the bound~\eqref{eq:keyconv} and dominated convergence. Altogether, passing to the limit $k\to \infty$ in~\eqref{eq:log-laplace-k}, we find that~\eqref{eq:log-laplace} holds true for $h$. 
\end{proof}

\begin{proof}[Proof of Theorem~\ref{thm:truncated}]
	Let $h:\mathbb X\to \R_+$. It follows from Corollary~9.5.VIII in~\cite{daley-verejones2008vol2} that 
	\be \label{eq:truncated-generating}  
		\log \mathsf E\Bigl[\e^{- \int_\mathbb X h \dd \eta}\Bigr] 
		 = \sum_{\ell=1}^\infty \frac{1}{\ell!} \int_{\mathbb X^\ell} \prod_{j=1}^\ell (\e^{- h(x_j)}-1) \rho_\ell^\mathsf T(x_1,\ldots,x_\ell) \dd \lambda^\ell(\vect x)
	\ee 
	whenever the right-hand side is absolutely convergent and $h$ is non-negative with bounded support. In the right-hand side of Eq.~\eqref{eq:log-laplace} in Theorem~\ref{thm:log-laplace}, 	
	we express the log-Laplace functional in terms of $\exp( - h)-1$. We have 
	\be
		\e^{- \sum_{i=1}^n h(x_i)}-1 = \sum_{\substack{I\subset [n]:\\ I\neq \varnothing}}\  \prod_{i\in I} (\e^{- h(x_i)}-1)
	\ee
	and, if $|h|\leq t$,
	\be \label{eq:trickcoral} 
		\sum_{\substack{I\subset [n]:\\ I\neq \varnothing}}\  \prod_{i\in I} \bigl| \e^{- h(x_i)}-1\bigr| 
		= \prod_{i=1}^n \bigl( 1+ |\e^{- h(x_i)} - 1|\bigr) -1 \leq \e^{(n-1)t}\sum_{j=1}^n |\e^{-h (x_j)}-1|,
	\ee
	compare Eq.~\eqref{eq:trick}. 
	Exploiting the symmetry of the Ursell functions, we deduce from Eq.~\eqref{eq:log-laplace} in Theorem~\ref{thm:log-laplace} that whenever $|h|\leq t$ and $h$ satisfies condition~\eqref{eq:hcondition}, we have 
	\begin{align}\label{eq:coral2} 
		\log \mathsf E\Bigl[\e^{- \int_\mathbb X h \dd \eta}\Bigr] 
		&= \sum_{n=1}^\infty \frac{1}{n!}\sum_{\ell=1}^n \binom{n}{\ell}\int_{\mathbb X^n} \prod_{i=1}^\ell(\e^{- h (x_i)} - 1) \varphi_n^\mathsf T(x_1,\ldots,x_n) \dd\lambda_z^n(\vect x) \\
		& = \sum_{\ell=1}^\infty\frac{1}{\ell!} \int_{\mathbb X^\ell} \prod_{j=1}^\ell (\e^{- h(x_j)}-1)\prod_{j=1}^\ell z(x_j) \notag \\
		&\qquad \qquad \times \left( \sum_{k=0}^\infty \frac{1}{k!}\int_{\mathbb X^k} \varphi_{\ell+k}^\mathsf T(x_1,\ldots,x_\ell,y_1,\ldots,y_k) \dd\lambda_z^k(\vect y) \right) \dd \lambda^\ell(\vect x)
	\end{align}
	The exchange of the order of summation is justified with~\eqref{eq:trickcoral} and absolute convergence. 
	Changing variables to $u(x) = 1 - \e^{-h(x)}$ and comparing~\eqref{eq:truncated-generating} and~\eqref{eq:coral2}, we find 
	\begin{multline}
		\int_{\mathbb X^\ell} \prod_{j=1}^\ell u(x_j) \rho_\ell^\mathsf T(x_1,\ldots,x_\ell) \dd \lambda^\ell(\vect x) \\
			= \int_{\mathbb X^\ell} \prod_{j=1}^\ell u(x_j) \left(\sum_{k=0}^\infty \frac{1}{k!}\int_{\mathbb X^k} \varphi_{\ell+k}^\mathsf T(x_1,\ldots,x_\ell,y_1,\ldots,y_k) \dd\lambda_z^k(\vect y)\right) \dd \lambda^\ell(\vect x)
	\end{multline} 
	for all $u:\mathbb X\to [0,1]$ that have bounded support and satisfy $\int_\mathbb X u \e^{a}\dd\lambda_z< \infty$. Eq.~\eqref{eq:truncated} follows. 
	For $s\in \mathbb C$ with $|s|\leq \e^t - 1$, we have 
	\begin{align*}
		&\sum_{\ell=0}^\infty \frac{|s|^\ell}{\ell!} \sum_{k=0}^\infty \frac{1}{k!} \int_{\mathbb X^{\ell+k}} |\varphi_{\ell+k+1}^\mathsf T(x_0,x_1,\ldots,x_{\ell+k}) |\dd \lambda_z(x_1)\cdots \dd \lambda_z(x_{\ell+k} )\\
		&\quad \leq \sum_{m=0}^\infty \frac{1}{m!} (1 + |s|)^m \int_{\mathbb X^m} |\varphi_{m+1}^\mathsf T(x_0,x_1,\ldots,x_m) \dd \lambda_z(x_1)\cdots \dd\lambda_z(x_m) \leq \e^{a(x_0)}. 
	\end{align*} 
	Taking contour integrals along the circle centered at the origin with radius $\e^t - 1$, we deduce 
	\begin{align*}
		\frac{1}{\ell!} \sum_{k=0}^\infty \frac{1}{k!} \int_{\mathbb X^{\ell+k}} |\varphi_{\ell+k+1}^\mathsf T(x_0,x_1,\ldots,x_{n+k}) |\dd \lambda_z(x_1)\cdots \dd \lambda_z(x_{\ell+k} )
		&\leq \Bigl|\frac{1}{2\pi \mathrm{i}}\oint \frac{\dd s}{s^{\ell+1}}\Bigr|\e^{a(x_0)} \\
		& \leq \frac{1}{(\e^t - 1)^\ell}\, \e^{a(x_0)}
	\end{align*}
	and the bound~\eqref{eq:truncated-convergence} follows. 
\end{proof}

\section{Gibbs measures in finite volume. Estimates for $t=0$} \label{sec:finite-volume} 

Here we prove Theorems~\ref{thm:finite-volume} and~\ref{thm:finite-volume-corr} on the partition function and correlation functions in finite volume. The proof of Theorem~\ref{thm:finite-volume} is fairly standard, once the bound~\eqref{eq:keyconv} is available. The proof of Theorem~\ref{thm:finite-volume-corr} explains the appearance of the class $\mathcal D_{k,n}$ of multirooted graphs directly, without any reference to Kirkwood-Salsburg equations. The theorems are proven in reverse order: first Theorem~\ref{thm:finite-volume}, then Theorem~\ref{thm:finite-volume-corr} and finally Theorem~\ref{thm:log-laplace-fivo}. 

\begin{proof} [Proof of Theorem~\ref{thm:finite-volume}] 
	First we note that the bound~\eqref{eq:keyconv} with $t=0$ follows from the identity $\varphi_{n}^\mathsf T = \psi_{1,n}$ and  Proposition~\ref{prop:picard} applied to $k=1$. The bound~\eqref{eq:fivo2} is an immediate consequence of~\eqref{eq:keyconv}. The proof of Theorem~\ref{thm:finite-volume} then follows standard arguments~\cite{ruelle1969book} which we reproduce for the reader's convenience. 
	Let $\mathcal C(V)$ be the collection of connected graphs with vertex set $V$ (thus $\mathcal C_n = \mathcal C([n])$). We have 
	$$
		\e^{- H_n(x_1,\ldots,x_n)} = \prod_{1\leq i < j \leq n}\bigl(1+ f(x_i,x_j)\bigr) 
			= \sum_{E\subset \{\{i,j\}\mid 1 \leq i < j \leq n\}} \prod_{\{i,j\}\in E} f(x_i,x_j).
	$$
	The sum over subsets $E$ can be reinterpreted as a sum over graphs with vertex set $\{1,\ldots,n\}$ and edge sets $E$ so that 
	\be \label{eq:keygraphs}
		\e^{- H_n(x_1,\ldots,x_n)}  =  \sum_{G\in \mathcal G_n} w(G;x_1,\ldots,x_n).
	\ee
	As the weight of a graph is the product of the weights of its connected components, we deduce
	\begin{align*} 
		\e^{- H_n(x_1,\ldots,x_n)}	& = \sum_{r=1}^n \sum_{\{V_1,\ldots,V_r\}\in \mathcal P_n} \prod_{k=1}^r \sum_{G_k\in \mathcal C(V_k)} w(G_k;\vect x_{V_k}) \\
			& =  \sum_{r=1}^n \sum_{\{V_1,\ldots,V_r\}\in \mathcal P_n} \prod_{k=1}^r \varphi_{\#V_k}^\mathsf T (\vect x_{V_k}).
	\end{align*} 
	We integrate both sides over $\Lambda^n$ and note that integrals factorize as well, moreover on the right-hand side they depend on the cardinality of $V_i$ alone. Thus 
	$$
		\Xi_\Lambda(z) = 1+\sum_{n=1}^\infty \frac{1}{n!}\sum_{r=1}^n \sum_{\{V_1,\ldots,V_r\}\in \mathcal P_n} A_{\#V_1}\cdots A_{\#V_k} 
	$$
	with 
	\be \label{eq:anchoice}
		A_n:= \int_{\Lambda^n} \Bigl( \sum_{G\in \mathcal C_n} w(G;x_1,\ldots,x_n) \Bigr) \dd \lambda_z^n(\vect x) = \int_{\Lambda^n} \varphi_n^\mathsf T(x_1,\ldots,x_n) \dd \lambda_z^n(\vect x).  
	\ee
	It is a general combinatorial fact that 
	\be \label{eq:exponentialformula}
		1+ \sum_{n=1}^\infty \frac{1}{n!}\sum_{r=1}^n \sum_{\{V_1,\ldots,V_r\}\in \mathcal P_n} A_{\#V_1}\cdots A_{\#V_k} = \exp\Bigl( \sum_{n=1}^\infty \frac{1}{n!} A_n\Bigr)
	\ee
	whenever $\sum_{n=1}^\infty \frac{1}{n!}|A_n| < \infty$. The latter condition is satisfied for the concrete choice~\eqref{eq:anchoice} because of~\eqref{eq:keyconv} with $t=0$ and we obtain 
	\bes
		\Xi_\Lambda(z)  = \exp\left( \sum_{n=1}^\infty \frac{1}{n!} \int_{\Lambda^n} \varphi_n^\mathsf T(x_1,\ldots,x_n) \dd \lambda_z^n(\vect x)\right). \qedhere
	\ees
\end{proof} 

\begin{proof}[Proof of Theorem~\ref{thm:finite-volume-corr}] 
	The bound~\eqref{eq:correlations-bound} with $t=0$ follows from Proposition~\ref{prop:picard}. In order to go from Theorem~\ref{thm:finite-volume} to correlation functions we note that the generating functional of the latter is
	 \be \label{eq:fivb}
	 	 \mathsf E_\Lambda\Bigl[\prod_{x\in \Lambda} (1+u(x))^{\eta(\{x\})}\Bigr] = 1 + \sum_{k=1}^\infty \frac{1}{k!}\int_{\Lambda^k} \prod_{j=1}^k u(x_j) \rho_{k,\Lambda}(\vect x) \dd \lambda^k(\vect x).
	 \ee
	 The identity holds true for all measurable  $u:\Lambda\to [-1,0]$, the right-hand side of~\eqref{eq:fivb} is absolutely convergent because of the bound~\eqref{eq:nonnegbound}. By the definition of partition functions, 
	\be \label{eq:fivc}
		\mathsf E_\Lambda\Bigl[\prod_{x\in \Lambda} (1+u(x))^{\eta(\{x\})}\Bigr] =\frac{ \Xi_\Lambda\bigl( z(1+u)\bigr)}{\Xi_\Lambda(z)}. 
	\ee
	In order to evaluate $\Xi_\Lambda(z(1+u))$, we remember~\eqref{eq:keygraphs} and expand 
	\begin{align*}
		\prod_{j=1}^n (1+u(x_j)) \sum_{G\in \mathcal G_n} w(G;x_1,\ldots,x_n) 
		= \sum_{J\subset [n]} \prod_{j\in J} u(x_j) \sum_{G\in \mathcal G_n} w(G;x_1,\ldots,x_n).
	\end{align*} 
	Given a non-empty subset $J\subset [n]$, we establish a one-to-one correspondence between graphs $G$ and pairs $(G_1,G_2)$ that consist of a graph $G_1$ on some subset $V\supset J$ subject to connectivity constraints, and a graph $G_2$ on $[n]\subset V$, as follows. 	
	
		For $G\in \mathcal G_n$ and a non-empty subset $J\subset [n]$, let $V\subset[n]$ be those vertices that connect to some vertex $j\in J$ by a path in $G$ (thus $J\subset V\subset [n]$), and $W= [n]\setminus V$. Clearly $G$ cannot have any edge connecting $V$ and $W$. Let $G|_V$ be the graph with vertex set $V$ and $W$ and edge set $\{\{i,j\}\mid i,j\in V,\, \{i,j\}\in E(G)\}$, similarly for $G|_W$. 	
	 Conversely, let $\mathcal D_J(V)$ be the graphs on $V$ for which every vertex in $V$ connects to some vertex $j\in J$. Given two graphs $G_1\in \mathcal D_J(V)$, $G_2\in \mathcal G(W)$ such that every vertex in $V$ connects to some elements $j\in J$ by a path in $G_1$, we obtain a graph $G\in \mathcal G_n$ with edge set $E(G) = E(G_1)\cup E(G_2)$. In this way we have a one-to-one correspondence, and we find 
	 \begin{multline*}
	 	\prod_{j=1}^n (1+u(x_j)) \sum_{G\in \mathcal G_n} w(G;x_1,\ldots,x_n) 
	 		= \sum_{G\in \mathcal G_n} w(G;x_1,\ldots,x_n) \\
	 		+	 		
	 		 \sum_{\substack{J\subset [n]:\\ J\neq \varnothing}} \sum_{\substack{\{V,W\}\in \mathcal P_n:\\ V\supset J}} \Bigl( \prod_{j\in J} u(x_j) \sum_{G_1\in \mathcal D_J(V)} w(G_1;\vect x_V)\Bigr) \Bigl(\sum_{G_2\in \mathcal G(W)} w(G_2; \vect x_W)  \Bigr).
	 \end{multline*}
	 Exploiting the symmetry of the graph weights, we deduce 
	 \be\label{eq:fiva}
	 	\Xi_\Lambda\bigl( z(1+u)\bigr)  = \Xi_\Lambda(z)\left( 1+ \sum_{k=1}^\infty \frac{1}{k!} \int_{\Lambda^k} \prod_{j=1}^k u(x_j) \Bigl(\sum_{m=0}^\infty \frac{1}{m!}\int_{\Lambda^m} \psi_{k,k+m}(x_1,\ldots,x_k,y_1,\ldots,y_m)\dd \lambda_z^m(\vect y)\Bigr)\right)
	 \ee
	 (think $k=\#J$, $m=\# (V\setminus J)$). 	 
	 The expression of the correlation functions  follows from Eqs.~\eqref{eq:fivb},~\eqref{eq:fivc} and~\eqref{eq:fiva}. 
\end{proof} 

\begin{proof} [Proof of Theorem~\ref{thm:log-laplace-fivo}] 
	The validity of~\eqref{eq:fivo2} has been checked at the beginning of the proof of Theorem~\ref{thm:finite-volume}. The representation for the log-Laplace functional builds on Theorem~\ref{thm:finite-volume-corr} and is similar to the proof of Theorem~\ref{thm:log-laplace} provided in Section~\ref{sec:truncated}. 
\end{proof}

\section{Kirkwood-Salsburg equations for trees and forests}\label{sec:trees}

Here we explain the connection between the Kirkwood-Salsburg equation from Lemma~\ref{lem:ks} and generating functions for trees. The key point is that the \emph{non-linear} fixed point equation~\eqref{eq:fptrees} for tree generating functions translates into a \emph{linear} set of equations for the generating functions of forests that is similar to the Kirkwood-Salsburg equations. 
As a by-product, we obtain a new proof of the implication (ii) $\Rightarrow$ (i) in  Proposition~\ref{prop:fptrees} that clarifies the relation between Kirkwood-Salsburg equations and inductive convergence proofs.

 Let $B_k(q;x_1,\ldots,x_k)$, $k \in \N_0$, be the family of non-negative weight functions given by 
\be
	B_0(q) := 1, \quad B_k(q;x_1,\ldots,x_k) = \prod_{i=1}^k |f(q;x_i)| \prod_{1\leq i < j \leq k} (1+ f(x_i,x_j)),
\ee
Set $F_0:=1$ and for $k\in \N$ and $\vect q\in \mathbb X^k$, set 
\bes
 F_k(q_1,\ldots,q_k;z) = \prod_{i=1}^k \tilde T_{q_i}^\bullet (z).
\ees	
$F_k$ is the generating function for $k$-forests, i.e., $k$-tuples $F = ((T_1,r_1),\ldots,(T_k,r_k))$ of rooted trees such that their respective vertex sets $V_1,\ldots,V_k$ form a set partition of $V$. The weight of a forest is the product of the weights of the trees, the root colors are prescribed as $x_{r_i}=q_i$. Clearly 
\bes
	F_{n+1}(q_0,\ldots,q_{n};z) = \tilde T_{q_0}^\bullet (z) F_{n}(q_1,\ldots,q_n;z).
\ees
Using the \emph{non-linear} fixed point equation~\eqref{eq:fptrees} for trees, we obtain a \emph{linear} set of equations for forests: we have 
\begin{multline*} 
	F_{n+1}(q_0,\ldots,q_{n};z)
	 = z(q_0) \Bigl( B_0(q_0) F_n(q_1,\ldots,q_n;z) \\ + \sum_{\ell=1}^\infty \frac{1}{\ell!} \int_{\mathbb X^\ell} B_\ell(q_0;x_1,\ldots,x_\ell) F_{n+\ell}(q_1,\ldots,q_n,x_1,\ldots,x_\ell;z) \dd \lambda^\ell(\vect x) \Bigr),
\end{multline*}
a variant of~\eqref{eq:ks} for forests. Let $F_k^{(N)}(q_1,\ldots,q_k;z)$ be the partial sums of the generating functions where we sum only over forests with vertex sets $[n]$, $n\leq N$. Proceeding as in the proof of Proposition~\ref{prop:picard}, we find $F_1^{(1)}(q_1) = z(q_1)$, $F_k^{(1)} =0$ for $k \geq 2$, and 
\begin{multline*} 
	F_{n+1}^{(N+1)}(q_0,\ldots,q_{n};z)
	 = z(q_0) \Bigl( B_0(q_0) F_n^{(N)} (q_1,\ldots,q_n;z) \\ + \sum_{\ell=1}^\infty \frac{1}{\ell!} \int_{\mathbb X^\ell} B_\ell(q_0;x_1,\ldots,x_\ell) F_{n+\ell}^{(N)} (q_1,\ldots,q_n,x_1,\ldots,x_\ell;z) \dd \lambda^\ell(\vect x) \Bigr). 
\end{multline*}
An induction over $N$ then shows that if 
\bes
	B_0(q)+ \sum_{k=1}^\infty \frac{1}{k!}\int_{\mathbb X^k}B_k(q;x_1,\ldots,x_k) \e^{\sum_{i=1}^k a(x_i)} \dd \lambda_z^k(\vect x) \leq \e^{a(x_i)}, 
\ees
then $F_n^{(N)} (q_1,\ldots,q_n;z) \leq \prod_{i=1}^n z(q_i) \e^{a(q_i)}$ for all $N,n, \vect q$. Passing to the limit $N\to \infty$, we find that the forest generating functions are convergent. In particular, the generating function for trees (i.e., $1$-forests) is convergent as well. This proves the implication (ii)$\Rightarrow$(i) in Proposition~\ref{prop:fptrees}. 

%

\section{Cumulants of second-order stochastic integrals} \label{sec:stochmoment}

Here we prove Propositions~\ref{prop:stochmoment} and~\ref{prop:cumulants-diagrams}.

\begin{proof} [Proof of Proposition~\ref{prop:stochmoment}]
 Pick $\eta = \sum_{i=1}^\kappa \delta_{x_i} \in \mathcal N$ and set $[\kappa]=\N$ if $\kappa = \infty$. Assume that $u$ is non-negative or that $\int_{\mathbb X} |u|\dd\eta^{(2)} < \infty$. 
 A reasoning similar to the proof of Eq.~\eqref{eq:doublexp} shows
\bes
	\Bigl( \sum_{i <j} u(x_i,x_j)\Bigr)^m = 
	 \sum_{\gamma\in \mathcal M([\kappa], [m])} \prod_{\{i,j\}\in \mathcal E_2([\kappa])} u(x_i,x_j)^{m_{ij}(\gamma)}.
\ees
Notice that the number of edges $\{i,j\}$ with $m_{ij}(\gamma)\geq 1$ is at most $m$ hence, in particular, finite even when $\kappa = \infty$, so the product is finite. 

For $\gamma\in \mathcal M([\kappa], [m])$, let $V\subset [\kappa]$ be the set of vertices that belong to some edge of $\gamma$. The cardinality $n= \#V$ satisfies $2 \leq n \leq 2m$ and the restriction of $\gamma$ to $V$ is a spanning multigraph on $V$. We have 
\bes
	\Bigl( \sum_{i <j} u(x_i,x_j)\Bigr)^m = \sum_{V\subset [\kappa]} \sum_{\gamma\in \mathcal M_s(V, [m])} \prod_{\{i,j\}\in \mathcal E_2([\kappa])} u(x_i,x_j)^{m_{ij}(\gamma)}  = \sum_{V\subset [\kappa]} F_{nm}(\vect x_V)
\ees
with 
\bes
	 F_{nm}(x_1,\ldots,x_n) = \sum_{\gamma\in \mathcal M_s([n], [m])} \prod_{1 \leq i < j \leq n} u(x_i,x_j)^{m_{ij}(\gamma)}. 
\ees
$F_{nm}$ vanishes unless $2 \leq n \leq 2m$. Each $F_{nm}$ is integrable with respect to $\lambda_z^n$ by the assumption on $u$. Therefore 
\bes
	\Bigl( \int_{\mathbb X^2} u \dd \eta^{(2)} \Bigr)^m = \sum_{n=2}^{2m}\frac{1}{n!} \int_{\mathbb X^n} F_{nm}\dd \eta^{(n)}
\ees
with integrals that are almost surely absolutely convergent, and 
\bes
	\mathbb E\Bigl[ \Bigl( \int_{\mathbb X^2} u \dd \eta^{(2)} \Bigr)^m \Bigr] 
		 = \sum_{n=2}^{2m} \frac{1}{n!} \int_{\mathbb X^n} F_{nm}(\vect x) \dd \lambda_z^n(\vect x). 
\ees
This proves the representation for the moments and shows that the moments of order $s \leq m$ of $\int_{\mathbb X^2} |u| \dd \eta^{(2)}$ are finite. 

For the cumulants, we use the multiplicativity of the graph weights $\bar w(\gamma;\vect x) = \prod_{\{i,j\}} u(x_i,x_j)^{m_{ij}(\gamma)}$: Each $\gamma\in \mathcal M_s([n],[m])$ decomposes into connected multigraphs $\gamma_i\in \mathcal M_c(V_i,A_i)$, $i=1,\ldots,r$, where $\{V_1,\ldots,V_r\}$ and $\{A_1,\ldots,A_r\}$ form set partitions of $[n]$ and $[m]$, and the weight of $\gamma$ is the product of the weights of the connected components. Therefore 
\bes
	\sum_{\gamma\in \mathcal M_s([n],[m])} \bar w(\gamma;x_1,\ldots,x_n) = \sum_{r=1}^m \sum_{\substack{\{V_1,\ldots,V_r\}\in \mathcal P_n\\ \{A_1,\ldots,A_r\}\in \mathcal P_m}}  \prod_{\ell=1}^r \Bigl( \sum_{\gamma_\ell\in \mathcal M_c(V_\ell,A_\ell)} \bar w(\gamma_\ell;\vect x_{V_\ell})\Bigr)
\ees
and, formally,
\begin{align}
	& \mathbb E\Bigl[ \Bigl( \int_{\mathbb X} u \dd \eta^{(2)} \Bigr)^m \Bigr] \notag \\
	&\quad  =\sum_{n=2}^{2m} \frac{1}{n!} \sum_{r=1}^n \sum_{\{A_1,\ldots,A_r\}\in \mathcal P_m} \sum_{\substack{n_1,\ldots,n_r\in \N_0:\\n_1+\cdots + n_r = n}} \binom{n}{n_1,\ldots,n_r} \sum_{\gamma \in \mathcal M_c([n_\ell],A_\ell)} \int_{\mathbb X^{n_\ell}} \bar w(\gamma;\vect x) \dd \lambda_z^{n_\ell}(\vect x) \notag \\
	&\quad = \sum_{r=1}^n \sum_{\{A_1,\ldots,A_r\}\in \mathcal P_m}  \prod_{\ell=1}^r \Bigl( \sum_{n_\ell=2}^{2 \#A_\ell} \frac{1}{n_\ell!}\int_{\mathbb X^{n_\ell}} \sum_{\gamma\in \mathcal M_c([n_\ell],A_\ell)}\bar w(\gamma;\vect x) \dd \lambda_z^{n_\ell}(\vect x)\Bigr) \notag \\
	&\quad = \sum_{r=1}^n \sum_{\{A_1,\ldots,A_r\}\in \mathcal P_m} \kappa_{\#A_1}\cdots \kappa_{\#A_r} \label{eq:cum} 
\end{align} 
with $\kappa_s$ defined by the equation in Proposition~\ref{prop:stochmoment}. The identity is formal for now because we need to check the convergence of the integrals in $\kappa_\ell$. But the convergence can be checked by induction over $s\in \{1,\ldots,m\}$, using that the moments $\mathbb E[(\int_{\mathbb X^2} |u|\dd \eta^{(2)})^s]$ are finite and that Eq.~\eqref{eq:cum} holds true if we replace $m$ by any $s\leq  m$. The set of equations~\eqref{eq:cum} with $s\leq m$ instead of $m$ is exactly the set of equations satisfied by the cumulants (see, for example,~\cite{peccati-taqqu2011book})  and the solution of the set of equations is unique, so we identify $\kappa_1,\ldots,\kappa_s$ as the first $s$ cumulants. 
\end{proof} 

\begin{proof}[Proof of Proposition~\ref{prop:cumulants-diagrams}]
	From Proposition~\ref{prop:stochmoment} and the general observation $\kappa_m(X) = 2^m \kappa_m(\frac12 X)$ we get
	\be \label{eq:diagram1}
		\kappa_m\Bigl( \int_{\mathbb X^2} u \dd \eta^{(2)}\Bigr) = 
		\sum_{n=2}^{2m} \frac{2^m}{n!}\int_{\mathbb X^n} \sum_{\gamma \in \mathcal M_c([n],[m])}\prod_{1 \leq i < j \leq n} u(x_i,x_j)^{m_{ij}(\gamma)}
	\dd \lambda^n_z.
	\ee
	For $\gamma\in \mathcal M_c([n],[m])$, let $S^\gamma := \{ (v,a)\in [n]\times [m]\mid v\in \gamma(a)\}$. Let $\pi^\gamma$ be the partition of $S^ \gamma$ with blocks $B_a^\gamma = \{ (v,a) \mid v\in \gamma(a)\}$, $a=1,\ldots,m$, and $\sigma^\gamma$ the partition with blocks $T^\gamma_v = \{(v,a)\mid a \in A: \, \gamma(a)\ni v\}$, $v=1,\ldots,n$.  We have already observed that $(\pi^\gamma,\sigma^\gamma)$ is non-flat and connected. The partition $\pi^\gamma$ has $m$ blocks of cardinality $2$ each, and $S^\gamma$ has cardinality $2m$, so it is natural to map $S^\gamma$ to $\{1,\ldots, 2m\}=S_m$ and work with pairs of partitions on $S_m$ instead. 
	
	In order to do this in a systematic way, we note that there are exactly $2^m$ ways of choosing a total order $\prec$ on $S^\gamma$ that respects the ordering of the blocks $B^\gamma_1,\ldots,B^\gamma_m$, i.e., such that for all $1\leq i <j \leq m$ and $s_i\in B^\gamma_i$, $s_j\in B^\gamma_j$, we have $s_i \prec s_j$. Indeed, the only degree of freedom we have is choosing the order within blocks. Since there are $m$ blocks of cardinality $2$ each, the overall number of choices is $2^m$. 
	
	Given an order $\prec$, we define $\varphi_{\prec}: S^\gamma \to \{1,2,\ldots,2m\}= S_m$ as the unique order-preserving bijection. Clearly $\varphi_\prec(B_1) = \{1,2\}$, $\varphi_{\prec} (B_2) = \{3,4\}$, etc. Put differently, $\varphi_{\prec}$ maps the partition $\pi^\gamma$ of $S^\gamma$ to the partition $\{\{1,2\}, \{3,4\}, \ldots, \{2m-1,2m\}\} = \pi_m$ of $S_m$. Let $\hat \sigma := (\varphi_{\prec}(T^\gamma_1),\ldots,\varphi_{\prec}(T^\gamma_n))$ be the ordered partition of $S_m$ obtained from $\sigma$ and the map $\varphi_{\prec}$. Let $\sigma$ be the underlying non-ordered set partition of $S_m$. It is easily checked that $(\pi_m,\sigma)$ is non-flat and connected. 
	
	Summarizing, consider the following two types of objects at fixed $n$ and $m$: 
	\begin{itemize} 
		\item pairs $(\gamma, \prec)$ consisting of a multigraph $\gamma\in \mathcal M_c([n],[m])$ and a total order $\gamma$ on $S^\gamma$ that preserves the order of the blocks $B^\gamma _1,\ldots,B^\gamma_m$;
		\item ordered set partitions $\hat \sigma = (T_1,\ldots, T_n)$ of $S_m$ with $n$ blocks such that $(\pi_m,\sigma)$ is non-flat and connected. 
	\end{itemize} 
	We have defined a systematic assignment $(\gamma,\prec)\mapsto \hat \sigma$, a careful examination shows that this is in fact a one-to-one correspondence. 
	
	Let us come back to Eq.~\eqref{eq:diagram1}. Write $\sideset{}{^n}\sum_{(\gamma, \prec)}$ and $\sideset{}{^n}\sum_{\hat \sigma}$ for sums over the types of objects introduced above. 	
	As there are $2^m$ choices of order $\prec$, we may rewrite Eq.~\eqref{eq:diagram1} as 
	\bes
		\kappa_m\Bigl( \int_{\mathbb X^2} u \dd \eta^{(2)}\Bigr) = 
		\sum_{n=2}^{2m} \frac{1}{n!}\sideset{}{^n}\sum_{(\gamma, \prec)} \int_{\mathbb X^n} \prod_{1 \leq i < j \leq n} u(x_i,x_j)^{m_{ij}(\gamma)}
	\dd \lambda^n_z(\vect x).
	\ees	
	The integral can be rewritten with the help of the ordered partition $\hat \sigma$ associated with $(\gamma,\prec)$  as 
	\bes	
			 \int_{\mathbb X^n} \prod_{1 \leq i < j \leq n} u(x_i,x_j)^{m_{ij}(\gamma)}
	\dd \lambda^n_z(\vect x) = \int_{\mathbb X^n} (u\otimes\cdots\otimes u)_{\hat \sigma}\, (x_1,\ldots,x_n) \dd \lambda_z^ n(\vect x),
	\ees
	hence 
	\bes
		\kappa_m\Bigl( \int_{\mathbb X^2} u \dd \eta^{(2)}\Bigr) = 
		\sum_{n=2}^{2m} \frac{1}{n!}\sideset{}{^n}\sum_{\hat \sigma}
			\int_{\mathbb X^n} (u\otimes\cdots\otimes u)_{\hat \sigma}\, (x_1,\ldots,x_n) \dd \lambda_z^ n(\vect x).
	\ees
	To conclude, we note that for a given set partition $\sigma$ of $n$ blocks, each of the $n!$ ordered set partitions $\hat \sigma$ gives rise to an integral with exactly the same value, so we may  replace the sum over ordered partition $\hat \sigma$ by a sum over non-ordered partitions $\sigma$ and drop the prefactor $1/n!$. The proposition follows. 
\end{proof}


\appendix

\section{Proof of Lemma~\ref{lem:correp} and Lemma~\ref{lem:ks}} \label{app:knownstuff}

\begin{proof} [Proof of Lemma~\ref{lem:correp}]
	Let $\eta^{(m)}$ be the $m$-th factorial measure of $\eta$, i.e., if $\eta = \sum_{j=1}^\kappa \delta_{x_j}$. Set 
	\bes
		\e^{-H(x_1,\ldots,x_m\mid \eta)} = \prod_{1\leq i<j \leq m}(1+f(x_i,x_j)) \prod_{i=1}^m \prod_{y\in \mathbb X} (1+ f(x_i,y))^{\eta(\{y\})}.
	\ees
   An induction over $m$, starting from~\eqref{eq:gnz}, shows for all $m\in \N$ and all measurable $F:\mathbb X^m\times \mathcal N\to \R_+$, we have 
	\begin{multline} \label{eq:mgnz} \tag{MGNZ}
		\mathsf E\Bigl[ \int_{\mathbb X^m} F(\vect x;\eta) \dd \eta^{(m)}(\vect x)\Bigr] 
		= \int_{\mathbb X^m} \mathsf E\Bigl[ F(\vect x;\eta + \delta_{x_1}+\cdots +\delta_{x_m}) \e^{-H(\delta_{x_1}+\cdots + \delta_{x_m}\mid \eta)} \Bigr] \dd \lambda_z^m(\vect x)
	\end{multline} 
	For $F(\vect x,\eta) = g(\vect x)$ with $g:\mathbb X^m\to \R_+$, we obtain 
	\be
		\int_{\mathbb X^m} g\dd \alpha_m = \mathsf E\Bigl[ \int_{\mathbb X^m} g \dd \eta^{(m)}\Bigr] = \int_{\mathbb X^m}\mathsf E\Bigl[\e^{-  H(x_1,\ldots,x_m\mid \eta)} \Bigr]  g(\vect x) \dd \lambda_z^m(\vect x).
	\ee
	It follows that the factorial moment measure $\alpha_m$ is absolutely continuous with respect to $\lambda^m$ with Radon-Nikod{\'y}m derivative 
	\bes
		\rho_m(\vect x) = \frac{\dd\alpha_m}{\dd \lambda^m}(\vect x) =\prod_{j=1}^m z(x_j) \mathsf E\Bigl[\e^{- H(x_1,\ldots,x_m\mid \eta)}\Bigr] \qquad \lambda^m\text{-a.e.} \qedhere
	\ees
\end{proof} 

\begin{proof}[Proof of Lemma~\ref{lem:ks}] Let $H$ be as in the proof of Lemma~\ref{lem:correp}, 	We decompose 
\bes
	H(x_0\cdots x_{n}\mid \eta) = W(x_0;x_1,\ldots,x_n) + W(x_0\mid \eta) + H(x_1,\ldots,x_n\mid \eta)
\ees
and obtain from Lemma~\ref{lem:correp} that
\begin{align*}
	\rho_{n+1}(x_0,x_1,\ldots,x_n) & = z (x_0)\e^{-\beta W(x_0;x_1,\ldots,x_n)}\mathsf E\Bigl[\prod_{j=1}^n z(x_j) \e^{-\beta [W(x_0;\eta)+H(x_1,\ldots,x_n\mid \eta)]} \Bigr] 
\end{align*} 
We can further expand 
\be \label{eq:ksaux}
	\e^{- W(x_0;\eta)} = 1+ \sum_{k=1}^\infty \frac{1}{k!}\int_{\mathbb X^k} \prod_{j=1}^k     f(x_0,y_j) \dd \eta^{(k)}(\vect y),
\ee
which is absolutely convergent for $\mathsf P$-almost all $\eta\in \mathcal N$ because of 
\begin{multline} \label{eq:ksaux2}
	 \mathsf E\left[1 + \sum_{k=1}^\infty \frac{1}{k!}\int_{\mathbb X^k} \left|\prod_{j=1}^k f(x_0,y_j) \right| \dd \eta^{(k)}(\vect y)\right]  \\
	 = 1+ \sum_{k=1}^\infty \frac{1}{k!}\int_{\mathbb X^k} \left|\prod_{j=1}^kf(x_0,y_j)) \right| \rho_{k}(\vect y)\dd \lambda^k (\vect y) 
	 \leq \exp\left(\int_\mathbb X |f(x_0,y)| \dd\lambda_z(y) \right) < \infty. 
\end{multline} 
In the last line we have used the integrability assumption on $f$ and the bound  $\rho_k(y_1,\ldots,y_k)\leq \prod_{j=1}^k z(y_j)$, which follows from Lemma~\ref{lem:correp} and $1+ f\leq 1$.   Applying~\eqref{eq:mgnz} we get 
\begin{align*}
& 	\mathsf E\Bigl[ \prod_{j=1}^n z(x_j) \int_{\mathbb X^k} \prod_{j=1}^k f(x_0,y_j) \e^{-H(x_1,\ldots,x_n\mid \eta)} \dd \eta^{(k)}(\vect y) \Bigr] \\
&\qquad = \int_{\mathbb X^k} \mathsf E\Bigl[ \prod_{j=1}^kf(x_0,y_j) z^{n+k} \e^{-[H(x_1,\ldots,x_n\mid \delta_{y_1}+\cdots + \delta_{y_k}+  \eta) + H(y_1,\ldots,y_k\mid \eta)]} \Bigr]  \prod_{j=1}^n z(x_j)  \prod_{j=1}^k z(y_j)\dd \lambda^k(\vect y)\\
&\qquad = \int_{\mathbb X^k}\prod_{j=1}^k f(x_0,y_j) \mathsf E\Bigl[ \prod_{j=1}^n z(x_j)  \prod_{j=1}^k z(y_j)  \e^{- H(x_1,\ldots,x_n,y_1,\ldots,y_k\mid  \eta)} \Bigr] \dd \lambda^k(\vect y) \\
&\qquad = \int_{\mathbb X^k}\prod_{j=1}^k f(x_0,y_j) \rho_{n+k}(x_1,\ldots,x_n,\vect y) \dd \lambda^k(\vect y).  
\end{align*} 
Because of~\eqref{eq:ksaux2}, we can exchange summation and integration in~\eqref{eq:ksaux} and~\eqref{eq:ks} follows. 
\end{proof} 

\section{Existence of Gibbs measures} \label{app:existence} 

Remember that $\mathbb X$ is a complete metric separable space. 

\begin{theorem} \label{thm:existence}
	Let $z:\mathbb X\to \R_+$ be measurable and such that $\lambda_z(B)=\int_B z \dd \lambda<\infty$ for all $B\in \mathcal X_\mathsf b$. 
	Let $v:\mathbb X\times \mathbb X\to \R\cup\{\infty\}$ be a measurable pair potential. Assume that:
	\begin{itemize} 
		\item [(i)] $v(x,y)\geq 0$ for $\lambda_z^2$-almost all $(x,y)\in\mathbb X^2$. 
		\item [(ii)] The potential $v$ satisfies the integrability condition 
		\be \label{eq:integrability} \tag{$\mathsf I$}
	\forall x\in \mathbb X:\quad	\int_\mathbb X |f(x,y)|\,  \dd \lambda_z(y) < \infty,\quad\text{for $\lambda$-almost all }x\in \mathbb X. 
\ee
	where $f(x,y) = \e^{ - v(x,y)}-1$. 
	\end{itemize}
	Then $\mathscr G(z) \neq \varnothing$. 
\end{theorem} 

\noindent The theorem is a variant of results by Ruelle~\cite{ruelle1970superstable} and Kondratiev and Kuna~\cite{kuna-kondratiev2003}, for the reader's convenience we include a complete proof.  Ruelle deals with translationally invariant, superstable, possibly negative interactions and simple point processes in $\R^d$. Kondratiev and Kuna work with simple point processes in more general spaces $\mathbb X$ with superstable, possibly non-negative interactions without translational invariance, under the slightly more restrictive integrability condition $\mathrm{ess\,sup}_{x\in \mathbb X}\int_\mathbb X|f(x,y)|\dd \lambda_z(y)<\infty$ (essential supremum with respect to $\lambda_z$).  The proof follows a standard scheme: 
\begin{enumerate}
	\item Choose a topology on the space  of probability measures  on $(\mathcal N,\mathfrak N)$. The choice involves trade-offs discussed at the beginning of \cite[Chapter 4]{georgii-book}. For point processes with locally finite intensity measure, a good choice is the topology $\mathcal \tau_\mathscr L$ of local convergence from~\cite{georgii-zessin1993}. 
	\item Show that the collection of finite-volume Gibbs measures is sequentially relatively compact, i.e., every sequence admits a convergent subsequence. Some criteria for relative compactness are: Ruelle's moment bounds for correlation functions~\cite{ruelle1970superstable,kuna-kondratiev2003}, entropy bounds~\cite{georgii-zessin1993,dereudre-drouilhet-georgii2012}, existence of Lyapunov functionals~\cite{kondratiev-pasurek-roeckner2012,conache-daletskii-kondratiev-pasurek2018}. Relative compactness is often easy for non-negative or locally stable interactions but can be difficult to prove for stable attractive interactions that favor accumulation of many points in small regions. 
	\item Show that every accumulation point of finite-volume Gibbs measures is an infinite-volume Gibbs measure. This is done by passing to the limit in one of the structural equations characterizing the set of Gibbs measures; usually, the DLR conditions or Kirkwood-Salsburg equations. The passage to the limit requires additional justification because the structural equations may involve functions that are not local (e.g., when the potential has infinite range).  \emph{Quasi-local specifications} provide  conceptual framework particularly helpful for lattice systems~\cite{georgii-book,preston-book}. 
 Additional conditions on the interaction potential are needed to ensure that all functions involved can be approximated by local functions. Typically this involves some decay of interactions, formulated as pointwise bounds or as an integrability condition similar to~\eqref{eq:integrability}~\cite{ruelle1969book,ruelle1970superstable}. 
\end{enumerate} 

\noindent We introduce some additional notation.  For $\eta \in \mathcal N$ and 
$\Delta\in \mathcal X$, let $\eta_\Delta$ be the restriction of $\eta$ to $\Delta$, i.e., 
$$
	\eta_\Delta(A) = \eta(A\cap \Delta) \qquad (A \in \mathcal X). 
$$
A function $h:\mathcal N\to \R$ is \emph{local} if there exists $\Delta \in \mathcal X_\mathsf b$ such that $h$ is \emph{$\Delta$-local}, i.e., 
$
	h(\eta) = h(\eta_\Delta)
$
for all $\eta\in \mathcal N$. The \emph{algebra of local events} is 
$$
	\mathcal A = \{  A\in \mathfrak N\mid \1_A \text{ is local}\}.
$$
Notice that the indicator of a measurable set $A\in \mathfrak N$ is $\Delta$-local if and only if it is in the $\sigma$-algebra $\mathfrak N_\Delta$ generated by the counting variables $\eta\mapsto \eta(B)$ with measurable $B \subset \Delta$, thus $\mathcal A = \bigcup_{\Delta\in \mathcal X_\mathsf b} \mathfrak N_\Delta$. 
A function $h:\mathcal N\to \R$ is \emph{exponentially bounded} if there exist 
$\Delta\in \mathcal X_\mathsf b$ and $C,c>0$ such that 
$$
	|h(\eta)|\leq C \exp(c\, \eta(\Delta))
$$
for all $\eta \in \mathcal N$, and
\emph{tame} if instead $|h(\eta)|\leq C(1+ \eta(\Delta))$. Every exponentially bounded function is tame. Let $\mathscr L$ be the class of measurable tame local functions $h:\mathcal N\to \R$ and $\mathscr P$ the set of probability measures $\mathsf P$ on $(\mathcal N,\mathfrak N)$ with locally finite intensity, i.e., $\mathsf E[\eta(\Delta)]<\infty$ for all $\Delta\in \mathcal X_\mathsf b$. The \emph{topology $\tau_\mathscr{L}$ of local convergence} is the smallest topology on $\mathscr P$ for which the maps $e_h:\, P\mapsto \int_\mathcal N h\, \dd P$, $h \in \mathscr L$, are continuous. 

Next let $(\Lambda_n)_{n\in \N}$ be an increasing sequence of non-empty bounded measurable sets such that every bounded set $B$ is eventually contained in some $\Lambda_n$ (the sequence is \emph{cofinal} \cite[Chapter 1.2]{georgii-book}). Write $\mathsf P_n$ for the finite-volume Gibbs measure in $\Lambda_n$ with empty boundary conditions and activity $z$. Thus $\mathsf P_n$ is absolutely continuous with respect to the law of the Poisson point process with intensity measure $\1_{\Lambda_n}\,\lambda_z $, and the Radon-Nikod{\'y}m derivative is proportional to $\exp( - H(\eta))$. 

\begin{lemma} \label{lem:subsequence}
	The sequence $(\mathsf P_n)_{n\in \N}$ admits a subsequence $(\mathsf P_{n_r})_{r\in \N}$  such that for some probability measure $\mathsf P$ on $(\mathcal N,\mathfrak N)$ and all measurable, local, exponentially bounded $h:\mathcal N\to \R$, we have 
	$$
		\lim_{r\to \infty} \int_\mathcal N h\, \dd \mathsf P_{n_r} = \int_\mathcal N h \, \dd \mathsf P.
	$$
\end{lemma} 

\noindent In particular, $(\mathsf P_{n_r})_{r\in \N}$ converges to $\mathsf P$ in the topology $\tau_\mathscr L$ of local convergence.

\begin{proof} 
	\emph{Step 1: Convergence of correlation functions.}
	For $k,n\in \N$, set
	\be\label{eq:correlationstilde}
		\bar \rho_k^{(n)}(x_1,\ldots,x_k):= \e^{- H_k(x_1,\ldots,x_k)} \, \mathsf E_n 	\Bigl[ \e^{- \sum_{i=1}^k W(x_i,\eta)}  \Bigr] \qquad (x_1,\ldots, x_k \in \mathbb X). 
	\ee
	By Lemma~\ref{lem:correp}, the function $\bar \rho_k^{(n)}$ is the Radon-Nikod{\'y}m derivative of the $k$-th factorial moment measure $\alpha_k^{(n)}$ with respect to $\lambda_z^{\otimes k}$. Because of $v\geq 0$, we have the pointwise inequality
		\be \label{eq:cobund}
			0 \leq \bar \rho_k^{(n)}\leq 1 \qquad \lambda_z^k\text{-a.e.}
		\ee 
	 for all $k,n\in \N$. For fixed $k\in \N$, we may view  $(\bar \rho_k^{(n)})_{n\in \N}$ as a bounded sequence in $L^\infty( \mathbb X^k, \mathcal X^{\otimes k}, \lambda_z^k)$. Since $L^\infty(\mathbb X^k,\mathcal X^{\otimes k}, \lambda_z^k)$ is the dual of the Banach space $L^1(\mathbb X^k,\mathcal X^{\otimes k}, \lambda_z^k)$, the Banach-Alaoglu theorem ensures the existence of a weak$^*$-convergent subsequence $(\bar \rho_{k}^{(n_r)})_{r\in \N} \stackrel{*}{\rightharpoonup} \bar \rho_k$, i.e., for all $f\in L^1(\mathbb X^k,\mathcal X^{\otimes k}, \lambda_z^k)$, we have 
	\be \label{eq:coconv}
		\lim_{r\to \infty} \int_{\mathbb X^k}f\, \bar \rho_k^{(n_r)} \dd \lambda_z^k = \int_{\mathbb X^k}f\, \bar \rho_k\, \dd \lambda_z^k,
	\ee
	moreover $0 \leq \bar \rho_k (\vect x) \leq 1$ for $\lambda_z^k$-almost all $\vect x \in \mathbb X^k$. A straightforward argument with diagonal sequences yields the existence of a subsequence $(\mathsf P_{n_r})_{r\in \N}$ such that the convergence~\eqref{eq:coconv} holds true for all $k\in \N$. Note, however, that we do not yet know that the functions $\bar \rho_k$ are factorial moment densities of some measure $\mathsf P$. \\
	
	\noindent\emph{Step 2: Convergence of Janossy densities and avoidance probabilities.} 	
	For $\Delta\in \mathcal X_\mathsf b$, $k,n\in \N$, and $(x_1,\ldots,x_k)\in \mathbb X^k$, define
	\be \label{eq:jadef}
		j_{k,\Delta}^{(n)} (x_1,\ldots, x_k) := \bar \rho_k^{(n)}(x_1,\ldots, x_k) + \sum_{m=1}^\infty\frac{(-1)^m}{m!} \int_{\Delta^m} \bar \rho_{k+m}^{(n)}(x_1,\ldots,x_k,y_1,\ldots,y_m) \dd \lambda_z^m(\vect y).
	\ee
	The right-hand side is absolutely convergent because of the pointwise inequality $0\leq \rho_\ell^{(n)} \leq 1$ and because of $\lambda_z(\Delta)<\infty$. The family $(j_{k,\Delta}^{(n)})_{k\in \N}$ forms a system of Janossy densities (with respect to $\lambda_z$) localized to $\Delta$ of $\mathsf P_n$ \cite[Eq.. (5.4.14)]{daley-verejones2003vol1}, consequently $j_{k,\Delta}(\vect x) \geq 0$ for $\lambda_z^k$-almost all $\vect x\in \Delta^k$.  For $k=0$, we define 
	$$
		j_{0,\Delta}^{(n)} := 1 + \sum_{m=1}^\infty\frac{(-1)^m}{m!} \int_{\Delta^m} \bar \rho_{m}^{(n)}(y_1,\ldots,y_m) \dd \lambda_z^m(\vect y)
	$$
	and note that $j_{0,\Delta}^{(n)} = \mathsf P_n( \eta(\Delta) =0)\geq 0$ is an avoidance probability~\cite[Example 5.4(a)]{daley-verejones2003vol1}.  Define $j_{k,\Delta}$ by formulas analogous to $j_{k,\Delta}^{(n)}$ but with $\bar \rho_k^{(n)}$ replaced by $\bar \rho_k$. It follows from~\eqref{eq:coconv} that
	\be \label{eq:jaconv}
		j_{0,\Delta}^{(n_r)}\to j_{0,\Delta},\quad j_{k,\Delta}^{(n_r)} \stackrel{*}{\rightharpoonup}  j_{k,\Delta} \quad (k\in \N)
	\ee
	as $r\to \infty$. (Here the weak$*$ convergence is with respect to the $L^\infty$ space with $\Delta^k$ instead of $\mathbb X^k$.)  Consequently the $j_{k,\Delta}$'s, $k\in \N$, and $j_{0,\Delta}$  are non-negative (up to null sets), however we do not yet know that they are the Janossy densities and avoidance probabilites, respectively, of some probability measure $\mathsf P$. \\
	
	\noindent \emph{Step 3: Convergence of expectations.} 	
	Let $h:\mathcal N\to \R$ be measurable, $\Delta$-local, and exponentially bounded, $|h(\eta)|\leq C \exp(c\, \eta(\Delta))$. Then the map $(x_1,\ldots,x_k)\mapsto h(\delta_{x_1}+\cdots + \delta_{x_k})$ from $\Delta^k\to \R$ is integrable against $\1_{\Delta^k} \lambda_z^k$, therefore by ~\eqref{eq:jaconv}, for all $k\in \N$, 
	$$
		\lim_{r\to \infty} \int_{\Delta^k} h(\delta_{x_1}+\cdots + \delta_{x_k}) j_{k,\Delta}^{(n_r)}(\vect x)\,  \dd \lambda_z^k(\vect x) = \int_{\Delta^k} h(\delta_{x_1}+\cdots + \delta_{x_k}) j_{k,\Delta}(\vect x)\,  \dd \lambda_z^k(\vect x). 
	$$
	Furthermore, by~\eqref{eq:jadef} and the bound~\eqref{eq:cobund},
	\begin{align*} 
		 \frac{1}{k!} \int_{\Delta^k}\bigl| h(\delta_{x_1}+\cdots + \delta_{x_k})\bigr|\, j_{k,\Delta}^{(n_r)}(\vect x) \,  \dd \lambda_z^k(\vect x)
		 	\leq \e^{\lambda_z(\Delta)} \, \frac{1}{k!} \bigl( C\,\e^c\, \lambda_z(\Delta)\bigr)^k.
	\end{align*} 
	The right-hand side is independent of $n_r$ and the sum over $k$ is finite, therefore we can exchange limits and summation over $k$ and obtain 
	\begin{align*} 
		\lim_{r\to \infty} \int_{\mathcal N} h \dd \mathsf P_n & = 
			\lim_{r\to \infty} \Biggl( h(0) j_{0,\Delta}^{(n_r)} + \sum_{k=1}^\infty \frac{1}{k!}\int_{\Delta^k} h(\delta_{x_1}+\cdots + \delta_{x_k}) j_{k,\Delta}^{(n_r)}(\vect x) \dd \lambda_z^k(\vect x) \Biggr)\\
			& = h(0) j_{0,\Delta} + \sum_{k=1}^\infty \frac{1}{k!}\int_{\Delta^k} h(\delta_{x_1}+\cdots + \delta_{x_k}) j_{k,\Delta}(\vect x) \dd \lambda_z^k(\vect x).
	\end{align*} 
	This applies in particular to the constant function $h(\eta) =1$, hence 
	$$
		1 =  j_{0,\Delta} + \sum_{k=1}^\infty \frac{1}{k!}\int_{\Delta^k} j_{k,\Delta}(\vect x) \dd \lambda_z^k(\vect x).
	$$
	
	\noindent \emph{Step 4: Kolmogorov consistency condition. Conclusion.} 
	Let $\mathcal N_\Delta = \{\eta \in \mathcal N\mid \eta(\mathbb X\setminus\Delta) = 0\}$ be the set of configurations supported in $\Delta$. Equip $\mathcal N_\Delta$ with the trace of the $\sigma$-algebra $\mathfrak N$. Let $\mathsf Q_\Delta$ be the uniquely defined probability measure on $\mathcal N_\Delta$ with Janossy densities $j_{k,\Delta}$, $k\in \N$. By Step 3, we have 
	$$
		\lim_{r\to \infty} \int_{\mathcal N} h \dd \mathsf P_n = \int_{\mathcal N_\Delta} h \dd \mathsf Q_\Delta
	$$ 
	for all $\Delta$-local, measurable, exponentially bounded $h$. If $\Delta_1\subset \Delta_2$ are two sets in $\mathcal X_\mathsf b$, then every function $h$ that is $\Delta_1$-local  is also $\Delta_2$-local, and we deduce the consistency property 
	$$
		\int_{\mathcal N_{\Delta_1}} h \dd \mathsf Q_{\Delta_1} = \int_{\mathcal N_{\Delta_2}} h \dd \mathsf Q_{\Delta_2} \quad\text{for all bounded measurable $\Delta_1$-local $h$}.
	$$
	Since $\mathbb X$ is Polish, it follows that there exists a uniquely defined measure $\mathsf P$ on $(\mathcal N,\mathfrak N)$ such that for all $\Delta\in \mathcal X_\mathsf b$, the image of $\mathsf P$ under the projection $\pi_\Delta:\mathcal  N\to \mathcal N_\Delta$, $\eta\mapsto \eta_\Delta$  is $\mathsf P\circ \pi_\Delta^{-1} = \mathsf Q_\Delta$, and the proof of the lemma is easily concluded. 
\end{proof} 

\noindent Next we check that we can pass to the limit in the structural equation defining $\mathscr G(z)$. As mentioned above, the passage to the limit is often done with DLR conditions or the Kirkwood-Salsburg equation, but can also be done in the GNZ equation, compare Zessin's construction for Papangelou kernels that satisfy a Feller property (reminiscent of the quasilocality of specifications)~\cite{zessin2009}. Let 
\be \label{eq:regular-delta}
	b(x):= \int_\mathbb X|f(x,y)|\dd \lambda_z(x), \quad \mathcal X_\mathsf b^*:=\Bigl\{\Delta\in \mathcal X_\mathsf b\,\Big|\, \int_\Delta \e^b \dd \lambda_z <\infty\Bigr \}. 
\ee

\begin{lemma}  \label{lem:gnzconvergence}
	Let $\mathsf P$ and $(\mathsf P_{n_r})_{r\in \N}$ be as in Lemma~\ref{lem:subsequence}.
	Let $F:\mathbb X\times \mathcal N\to \R$ be measurable and bounded. Assume that (i) $F(x,\cdot)$ is local, for each $x\in \mathbb X$, and (ii) there exists $\Delta\in \mathcal X_\mathsf b^*$ such that $F(\cdot, \eta)$ vanishes on $\mathbb X\setminus \Delta$, for all $\eta\in \mathcal N$.  Then 
	\begin{align*}
			\lim_{r\to \infty} \mathsf E_{n_r} \Bigl[\int_\mathbb X F(x,\eta) \dd\eta( x)\Bigr] & =  \mathsf E \Bigl[\int_\mathbb X F(x,\eta) \dd \eta( x)\Bigr],\\
		\lim_{r\to \infty} \int_\mathbb X  \mathsf E_{n_r} \Bigl[ \e^{- W(x,\eta)} F(x,\eta+\delta_x)\Bigr]  \dd \lambda_z(x) & = \int_\mathbb X \mathsf E \Bigl[ \e^{- W(x,\eta)} F(x,\eta+\delta_x)\Bigr] \dd \lambda_z(x).
	\end{align*} 
\end{lemma} 

\begin{proof} 	
	The map $\eta\mapsto \int_\mathbb X F(x,\eta)\dd  \eta(x)$ is clearly local. It is also tame (hence exponentially bounded), since it is bounded by $\eta(\Delta) \sup |F|$. Therefore Lemma~\ref{lem:subsequence} yields the first equality. For the second equality, set 
	$$
		G(\eta): = \int_\mathbb X \e^{- W(x,\eta)} F(x,\eta+ \delta_x) \dd \lambda_z(x).
	$$
	Because of $W\geq 0$, the map $G$ is bounded by $\lambda_z(\Delta) \sup |F|$, however in general it is not local (unless the interaction has finite range), therefore we approximate  by local functions. For $B\in \mathcal X_\mathsf b$, set
	$$
		G_B(\eta):= \int_\mathbb X  \Biggl( 1+\sum_{k=1}^\infty \frac{1}{k!} \int_{B^k} \prod_{i=1}^k f(x,y_i) \dd \eta^{(k)}(\vect y)  \Biggr) F(x,\eta+ \delta_x) \dd \lambda_z(x). 
	$$
	In view of $|f(x,y)| = |\exp( - v(x,y))- 1|\leq 1$, we may bound
	$$
			\int_\mathbb X  \Biggl( 1+\sum_{k=1}^\infty \frac{1}{k!} \int_{B^k} \prod_{i=1}^k \bigl|f(x,y_i)\bigr| \dd \eta^{(k)}(\vect y)  \Biggr) \bigl|F(x,\eta+ \delta_x)\bigr| \dd \lambda_z(x)	\leq \e^{\eta(B)} \lambda_z(\Delta) \sup |F|,
	$$
	hence $G_B$ is well-defined and exponentially bounded. $G_B$ is clearly local.  As a consequence, 
	\be \label{eq:gbconv}
		\lim_{r\to \infty} \mathsf E_{n_r}[G_B]= \mathsf E[G_B],\qquad (B\in \mathcal X_\mathsf b).
	\ee
	Next we bound the expected value of $G - G_B$ with respect to $\mathsf P_n$ and $\mathsf P$. Using the expansion for $\exp( - W(x,\eta))$ from~\eqref{eq:ksaux}  and the bound $\bar \rho_k^n \leq 1$ (see~\eqref{eq:cobund}), we obtain
	\be \label{eq:unibound}
		\mathsf E_n\bigl[ |G- G_B|\bigr]  \leq \Bigl( \sup|F|\Bigr) \int_\Delta \Biggl( \sum_{k=1}^ \infty \frac{1}{k!} \int_{\mathbb X^k\setminus B^k}\prod_{i=1}^k\bigl| f(x,y_i)\bigr|  \dd \lambda_z^k(\vect y) \Biggr) \dd \lambda_z(x).
	\ee
	Notice that the right-hand side does not depend on $n$. The same bound holds true for the expected value with respect to $\mathsf P$. 
	Because of $\Delta \in \mathcal X_\mathsf b^*$ defined in~\eqref{eq:regular-delta},  we have 
	$$
		\int_\Delta \Biggl( \sum_{k=1}^ \infty \frac{1}{k!} \int_{\mathbb X^k}\prod_{i=1}^k\bigl| f(x,y_i)\bigr|  \dd \lambda_z^k(\vect y) \Biggr) \dd \lambda_z(x)
		= \int_\Delta \exp\Bigl( \int_\mathbb X |f(x,y)|\, \dd \lambda_z(y)\Bigr) \dd \lambda_z(x)< \infty.
	$$
	 Thus by dominated convergence applied to the right-hand side of~\eqref{eq:unibound}, for every increasing sequence $(B_j)_{j\in \N}$ with $B_j\nearrow \mathbb X$, we have 
	$$
		\lim_{j\to \infty} \sup_{n\in \N} \mathsf E_n\bigl[ |G- G_{B_j}|\bigr]   =0,
	$$
	similarly for expectation with respect to $\mathsf P$. Therefore given $\eps>0$, we can find $B\in \mathcal X_\mathsf b$ such that 
	$$
		\sup_{n\in \N} \mathsf E_n\bigl[ |G- G_{B}|\bigr] \leq\frac{\eps}{3}, \quad 
 \mathsf E\bigl[ |G- G_{B}|\bigr] \leq \frac{\eps}{3}.
	$$
	The proof is concluded with the help of~\eqref{eq:gbconv} and a straightforward $\eps/3$-argument. 
\end{proof} 

\begin{proof}[Proof of Theorem~\ref{thm:existence}] 
	Let $(\Lambda_n)_{n\in \N}$ be a cofinal sequence of non-empty bounded sets, $(\mathsf P_n)_{n\in \N}$ the finite-volume Gibbs measures with empty boundary conditions, and $(\mathsf P_{n_r})_{r\in \N}$ the convergent subsequence from Lemma~\ref{lem:subsequence}. We show that the accumulation point $\mathsf P$ is in $\mathscr G(z)$. The finite-volume Gibbs measures satisfy
	\be \label{eq:gnzfivon}
		\mathsf E_n\Bigl[\int_{\Lambda_n} F(x,\eta)\dd \eta(x)\Bigr] = \int_{\Lambda_n}\, \mathsf E_n\Bigl[\e^{- W(x,\eta)} F(x,\eta+ \delta_x)\Bigr] \dd\lambda_z(x),
	\ee
	for all $F:\mathbb X\times \mathcal N\to \R_+$ measurable. Suppose that  $F$ is as in Lemma~\ref{lem:gnzconvergence}, i.e., $F$ is bounded, $F(x,\cdot)$ is local, and $F(\cdot,\eta)$ vanishes outside a bounded set $\Delta\in \mathcal X_\mathsf b$ satisfying~\eqref{eq:regular-delta}.  Since $\Delta\subset \Lambda_n$ for all sufficiently large $n$, we can replace integration over $\Lambda_n$ in~\eqref{eq:gnzfivon} by integration over $\mathbb X$, and then take limits along the subsequence $(n_r)$. Lemma~\ref{lem:subsequence} shows that 
	\be \label{eq:gnzprefinal}
			\mathsf E\Bigl[\int_{\mathbb X} F(x,\eta)\dd \eta(x)\Bigr] = \int_{\mathbb X}\, \mathsf E\Bigl[\e^{- W(x,\eta)} F(x,\eta+ \delta_x)\Bigr] \dd\lambda_z(x).
	\ee
	The identity is extended to  all non-negative measurable $F$ with a $\pi$-$\lambda$ theorem for $\sigma$-finite measures. Consider the measures on $(\mathbb X\times \mathcal N, \mathcal X\otimes \mathfrak N)$ defined by 
	$$
		\mathscr C(B) = \mathsf E\Bigl[ \int_{\mathbb X}\1_B(x,\eta)\dd \eta(x)\Bigr],\quad \mu(B) = \int_{\mathbb X} \, \mathsf E\Bigl[\e^{-W(x,\eta)} \1_B(x,\eta+\delta_x)\Bigr]\, \dd \lambda_z(x).
	$$
	$\mathscr C$ is the Campbell measure of $\mathsf P$. 	
 Applying Eq.~\eqref{eq:gnzprefinal} to $F(x,\eta) = \1_\Delta(x) \1_A(x)$ with $\Delta\in \mathcal X_\mathsf b^*$  and $A\in \mathcal A$ a local event, we find that 
	$$	
		\mathscr C(\Delta\times A) = \mu(\Delta\times A) \quad (\Delta\in \mathcal X_\mathsf b^*,\, A\in \mathcal A). 
	$$
	Because of $\bar \rho_1 \leq 1$, we have 
	  $$\mathscr C (\Delta\times \mathcal N) = \mathsf E[\eta (\Delta)] = \int_\Delta \bar \rho_1\dd \lambda_z(\Delta) \leq \lambda_z(\Delta)<\infty
	  $$
	for all $\Delta \in \mathcal X_\mathsf b$, hence $\mathscr C$ is $\sigma$-finite. The measure $\mu$ is $\sigma$-finite as well, since $\mu(\Delta \times \mathcal N) = \mathscr C(\Delta\times \mathcal N)<\infty$ for all $\Delta \in \mathcal X_\mathsf b$ and $\mathbb X$ can be written as a countable union of sets in $\mathcal X_\mathsf b^*$. To conclude, we note that the collection of sets of the form $\Delta \times A$ with $\Delta \in \mathcal X_\mathsf b^*$ and
	$A\in \mathcal A$ is a $\pi$-system that generates $\mathcal X\otimes \mathfrak N$. Since the measures $\mathscr C$ and $\mu$ are $\sigma$-finite and coincide on a generating $\pi$-system, they must be equal. It follows that~\eqref{eq:gnzprefinal} extends to all indicators of measurable sets and then to all measurable non-negative functions. Thus $\mathsf P$ satisfies~\eqref{eq:gnz} and we have proven $\mathsf P\in \mathscr G(z)$. In particular, $\mathscr G(z)\neq \varnothing$. 
\end{proof} 

\section{Random connection model} \label{app:RCM} 

Here we provide some details for the random connection model discussed in Section~\ref{sec:boolean}.  Theorem~\ref{thm:RCM} below is a variant of a standard result for translationally invariant random connection model. The theorem follows by a straightforward adaptation of a construction with pruned branching random walk from~\cite{meester-penrose-sarkar1997}, the details are left to the reader. 

Let $\eta$ be a Poisson point process with intensity $\lambda_z$, defined on some probability space $(\Omega,\mathcal F, \mathbb P)$.  
For simplicity let us assume that $\lambda_z$ has no atoms so that $\eta$ is simple. For $q\in \mathbb X$ write $\eta^q = \eta+ \delta_x$. 
Assume we are given a measurable map $\varphi:\mathbb X\times \mathbb X\to [0,1]$ with $\varphi(x,y) = \varphi(y,x)$ and $\int_{\mathbb X} \varphi(x,y) \dd \lambda_z(y) <\infty$; for example, $\varphi(x,y) = |f(x,y)|$ with $f(x,y)$ satisfying~\eqref{eq:bfinite}.   The \emph{random connection model} is a random graph $G$ whose vertices are the points of $\eta$. Roughly, it is constructed as follows: Conditional on $\eta = \sum_{i=1}^\kappa \delta_{x_i}$, the events $\{ \{x_i,x_j\}\in E(G)\}$, $i<j$, are independent and the probability that $\{x_i,x_j\}$ belongs to the graph is given by $\varphi(x_i,x_j)$.  For $q\in \mathbb X$,  a random graph $G^q$ whose vertices are the points of $\eta+\delta_q$ is defined, intuitively, in a similar fashion; we write $\mathscr C(q,\eta)$ for the connected component of $q$ in  $G^q$. Precise definitions are found in~\cite{meester-roy1996book,last-ziesche2017} for the  translationally invariant case in $\R^d$ and in~\cite{last-nestmann-schulte2018} for the general setup.

Remember the branching process $(\eta_n^q)_{n\in \N_0}$ introduced above Proposition~\ref{prop:extinction}. It is a well-known result~\cite[Chapter 6]{meester-roy1996book} that extinction of a dominating branching process implies absence of percolation in the random connection model. 

\begin{theorem} \label{thm:RCM}
	Suppose that for $\lambda_z$-almost all $q$, 
	$(\eta_n^q)_{n\in \N}$ goes extinct with probability $1$. Then for $\lambda_z$-almost all $q\in \mathbb X$, $\mathbb P( |\mathscr C(q,\eta)|<\infty) = 1$.
\end{theorem}

\noindent The proof is a straightforward adaptation of the proof in~\cite{meester-roy1996book, meester-penrose-sarkar1997} and therefore omitted. Combining Theorem~\ref{thm:RCM} with Proposition~\ref{prop:extinction}, we obtain the following corollary. 

\begin{cor} 
	Let $z:\mathbb X\to \R_+$ be a measurable function with $\int_\mathbb X z\dd \lambda <\infty$ for all $B\in \mathcal X_\mathsf b$, and $f(x,y) = \exp( - v(x,y)) - 1$ with $v:\mathbb X\times\mathbb X\to [0,\infty) \cup \{\infty\}$ measurable. 
	Suppose that the sufficient convergence condition~\eqref{eq:suff2} with $t=0$ holds true. Then in the random connection model with connectivity probability $\varphi(x,y) = |f(x,y)|$, driven by a Poisson point process with intensity measure $\lambda_z$, we have $\mathbb P( |\mathscr C(q,\eta)|<\infty) = 1$  for $\lambda_z$-almost all $q\in \mathbb X$. 
\end{cor} 

\noindent We leave as an open problem whether a similar result holds true for the convergence criterion~\eqref{eq:suff1}. 

\subsubsection*{Acknowledgments} 
The generalization to the continuum set-up of the convergence condition by Fern{\'a}ndez and Procacci grew out of discussions with Dimitrios Tsagkarogiannis. I am also indebted to  Christoph Hofer-Temmel and  Matthias Schulte for alerting me to expansion techniques in the theory of  point processes. The present work benefitted from the stimulating atmosphere at the Oberwolfach mini-workshop \emph{Cluster expansions: from combinatorics to analysis via probability} (February 2017). I gratefully acknowledge support by the DFG scientific network \emph{Cumulants, concentration, and superconcentration}. 

\bibliographystyle{amsalpha}
\bibliography{clusterbiblio}

\end{document}